\documentclass[12pt,twoside]{amsart}
\usepackage[margin=3cm]{geometry}
\usepackage[colorlinks=false]{hyperref}
\usepackage[english]{babel}
\usepackage{graphicx,titling}
\usepackage{float}
\usepackage{amsmath,amsfonts,amssymb,amsthm}
\usepackage{lipsum}
\usepackage[T1]{fontenc}
\usepackage{fourier}
\usepackage{color}
\usepackage[latin1]{inputenc}
\usepackage{esint}
\usepackage{caption}
\usepackage{ccicons}

\makeatletter
\def\blfootnote{\xdef\@thefnmark{}\@footnotetext}
\makeatother

\newcommand\ccnote{
    \blfootnote{\copyright\,\, Camillo De Lellis and Anna Skorobogatova}
    \blfootnote{\ccLogo\, \ccAttribution\,\, Licensed under a \href{https://creativecommons.org/licenses/by/4.0/}{Creative Commons Attribution License (CC-BY)}.}
}

\usepackage[export]{adjustbox}
\numberwithin{equation}{section}
\usepackage{setspace}

\renewcommand{\leq}{\leqslant}

\renewcommand{\geq}{\geqslant}
\renewcommand{\mathbb}{\varmathbb}
\usepackage{fancyhdr}
\newtheorem{theorem}{Theorem}[section]
\newtheorem{lemma}[theorem]{Lemma}
\newtheorem{corollary}[theorem]{Corollary}
\newtheorem{proposition}[theorem]{Proposition}
\newtheorem{definition}[theorem]{Definition}
\newtheorem{remark}[theorem]{Remark}
\usepackage[initials,msc-links]{amsrefs}
\usepackage{url}
\usepackage[dvipsnames]{xcolor}
\usepackage[english=american]{csquotes}
\usepackage{enumerate}
\usepackage{textcomp}
\usepackage[normalem]{ulem}
\usepackage{mathrsfs}
\usepackage{mathtools}
\definecolor{Gray}{gray}{0.9}
\usepackage{bbm}
\usepackage{booktabs,colortbl, array}
\usepackage{pgfplotstable}
\pgfplotsset{compat=1.8}

\makeatletter
\DeclareFontFamily{OMX}{MnSymbolE}{}
\DeclareSymbolFont{MnLargeSymbols}{OMX}{MnSymbolE}{m}{n}
\SetSymbolFont{MnLargeSymbols}{bold}{OMX}{MnSymbolE}{b}{n}
\DeclareFontShape{OMX}{MnSymbolE}{m}{n}{
    <-6>  MnSymbolE5
    <6-7>  MnSymbolE6
    <7-8>  MnSymbolE7
    <8-9>  MnSymbolE8
    <9-10> MnSymbolE9
    <10-12> MnSymbolE10
    <12->   MnSymbolE12
}{}
\DeclareFontShape{OMX}{MnSymbolE}{b}{n}{
    <-6>  MnSymbolE-Bold5
    <6-7>  MnSymbolE-Bold6
    <7-8>  MnSymbolE-Bold7
    <8-9>  MnSymbolE-Bold8
    <9-10> MnSymbolE-Bold9
    <10-12> MnSymbolE-Bold10
    <12->   MnSymbolE-Bold12
}{}
\let\llangle\@undefined
\let\rrangle\@undefined
\DeclareMathDelimiter{\llangle}{\mathopen}%
{MnLargeSymbols}{'164}{MnLargeSymbols}{'164}
\DeclareMathDelimiter{\rrangle}{\mathclose}%
{MnLargeSymbols}{'171}{MnLargeSymbols}{'171}
\makeatother

\theoremstyle{definition}
\newtheorem{example}[theorem]{Example}

\newtheorem{assumption}[theorem]{Assumption}

\newcommand{\Crm}{\mathrm{C}}

\newcommand{\Irm}{\mathrm{I}}

\newcommand{\Trm}{\mathrm{T}}

\newcommand{\Wrm}{\mathrm{W}}

\newcommand{\Acal}{\mathcal{A}}
\newcommand{\Bcal}{\mathcal B}

\newcommand{\Fcal}{\mathcal{F}}
\newcommand{\Gcal}{\mathcal{G}}
\newcommand{\Hcal}{\mathcal{H}}

\newcommand{\Kcal}{\mathcal{K}}
\newcommand{\Lcal}{\mathcal{L}}
\newcommand{\Mcal}{\mathcal{M}}

\newcommand{\Scal}{\mathcal{S}}

\newcommand{\Wscr}{\mathscr{W}}

\newcommand{\Abf}{\mathbf{A}}
\newcommand{\Bbf}{\mathbf{B}}
\newcommand{\Cbf}{\mathbf{C}}
\newcommand{\Dbf}{\mathbf{D}}
\newcommand{\Ebf}{\mathbf{E}}

\newcommand{\Gbf}{\mathbf{G}}
\newcommand{\Hbf}{\mathbf{H}}
\newcommand{\Ibf}{\mathbf{I}}

\newcommand{\Sbf}{\mathbf{S}}
\newcommand{\Tbf}{\mathbf{T}}

\newcommand{\Ombf}{\boldsymbol{\Omega}}

\newcommand{\cl}[1]{\overline{#1}}

\newcommand{\dd}{\;\mathrm{d}}

\newcommand{\N}{\mathbb{N}}
\newcommand{\R}{\mathbb{R}}

\newcommand{\loc}{\mathrm{loc}}

\newcommand{\spt}{\mathrm{spt}}

\newcommand{\Sing}{\mathrm{Sing}}

\newcommand{\toweakstar}{\overset{*}\rightharpoonup}

\newcommand{\todown}{\downarrow}

\newcommand{\BV}{\mathrm{BV}}
\newcommand{\TV}{\mathrm{TV}}

\newcommand{\eps}{\epsilon}

\renewcommand{\eps}{\varepsilon}
\newcommand{\vphi}{\varphi}

\DeclareMathOperator{\Lip}{Lip}
\newcommand{\mres}{\mathbin{\vrule height 1.6ex depth 0pt width
        0.13ex\vrule height 0.13ex depth 0pt width 1.3ex}}
\newcommand{\flatS}{\mathfrak{F}}

\newcommand{\aveint}[2]{\mathchoice%
    {\mathop{\kern 0.2em\vrule width 0.6em height 0.69678ex depth -0.58065ex
            \kern -0.8em \intop}\nolimits_{\kern -0.45em#1}^{#2}}%
    {\mathop{\kern 0.1em\vrule width 0.5em height 0.69678ex depth -0.60387ex
            \kern -0.6em \intop}\nolimits_{#1}^{#2}}%
    {\mathop{\kern 0.1em\vrule width 0.5em height 0.69678ex depth -0.60387ex
            \kern -0.6em \intop}\nolimits_{#1}^{#2}}%
    {\mathop{\kern 0.1em\vrule width 0.5em height 0.69678ex depth -0.60387ex
            \kern -0.6em \intop}\nolimits_{#1}^{#2}}}

\newcommand\res{\mathop{\hbox{\vrule height 7pt width .3pt depth 0pt\vrule height .3pt width 5pt depth 0pt}}\nolimits}

\DeclareMathOperator{\Err}{Err}

\address{Camillo De Lellis, School of Mathematics, Institute for Advanced Study, 1 Einstein Dr., Princeton NJ 05840, USA}
\email{camillo.delellis@ias.edu}

\address{Anna Skorobogatova, Institute for Theoretical Sciences, ETH Z\"{u}rich, Scheuchzerstrasse 70, 8006 Z\"{u}rich, Switzerland}
\email{anna.skorobogatova@eth-its.ethz.ch}

\begin{document}

\thispagestyle{empty}

\begin{minipage}{0.28\textwidth}
\begin{figure}[H]
%\centering
\includegraphics[width=2.5cm,height=2.5cm,left]{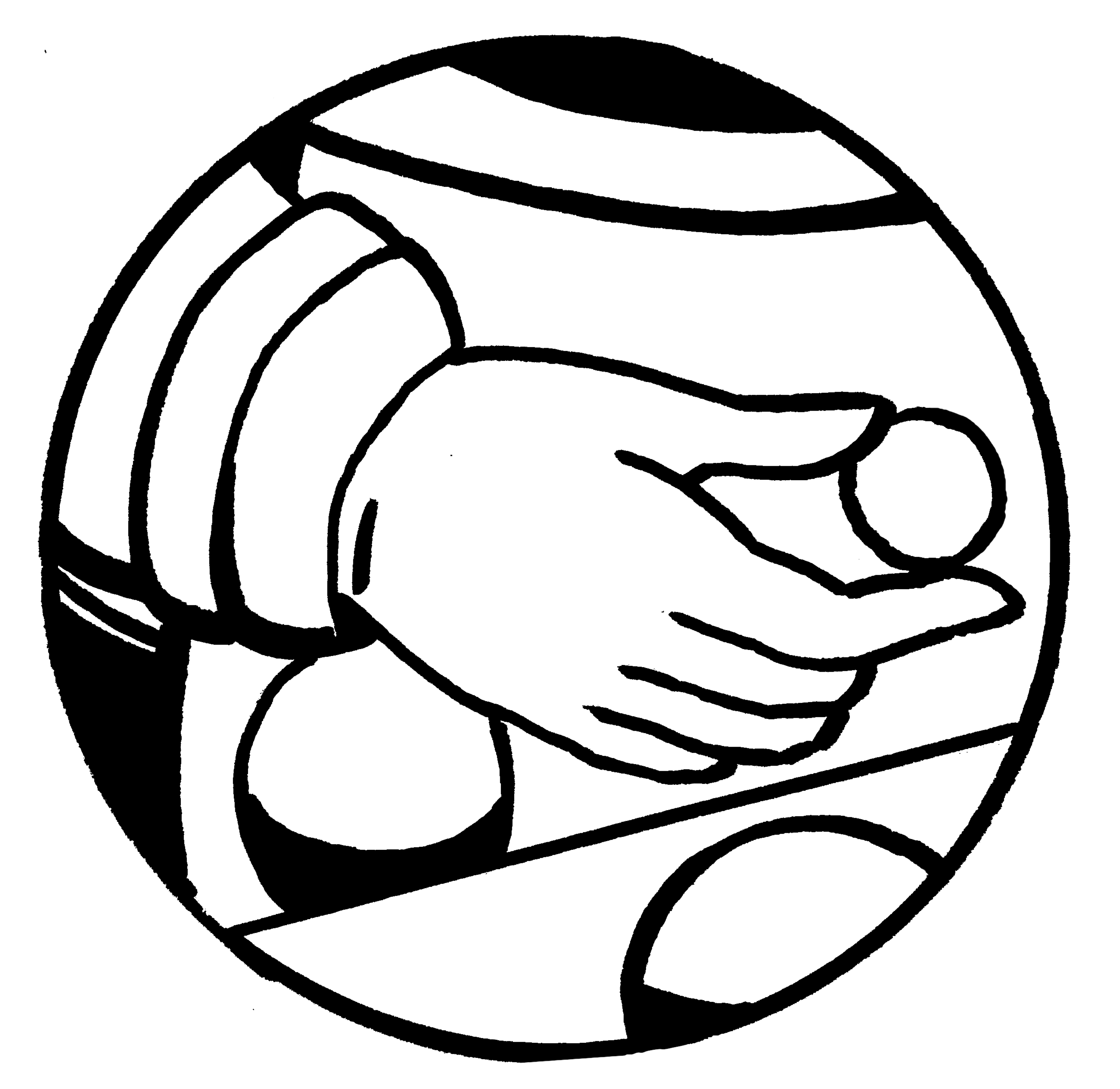}
\end{figure}
\end{minipage}
\begin{minipage}{0.7\textwidth} 
\begin{flushright}
%% The following metadata, in particular
%% the Paper No. and the DOI will be inserted by the journal
Ars Inveniendi Analytica (2025), Paper No. 3, 61 pp.
\\
DOI 10.15781/y8xg-xv26
\\
ISSN: 2769-8505
\end{flushright}
\end{minipage}

\ccnote

\vspace{1cm}

%%      -------------------------------------------------------------------------------
%%      -------------------------- TITLE ----------------------------
%%      -------------------------------------------------------------------------------
%% Authors, please put here the full title of the article

\begin{center}
\begin{huge}
\textit{The fine structure of the singular set of area-minimizing integral currents I: the singularity degree of flat singular points}

%\textit{some titles take two lines}

\end{huge}
\end{center}

\vspace{1cm}

%%      -------------------------------------------------------------------------------
%%      -------------------------- AUTHORS AND AFFILIATIONS ----------------------------
%%      -------------------------------------------------------------------------------
%% Authors, please put here your full names and affiliations

\begin{minipage}[t]{.28\textwidth}
\begin{center}
{\large{\bf{Camillo De Lellis}}} \\
\vskip0.15cm
\footnotesize{Institute for Advanced Study}
\end{center}
\end{minipage}
\hfill
\noindent
\begin{minipage}[t]{.28\textwidth}
\begin{center}
{\large{\bf{Anna Skorobogatova}}} \\
\vskip0.15cm
\footnotesize{ETH Z\"{u}rich}
\end{center}
\end{minipage}

\vspace{1cm}

%%% Please replace "James Mustard" below 
%%% with the name of the managing editor for your submission.
%%% If you are unsure about their identity
%%% please ask an editor-in-chief about.

\begin{center}
\noindent \em{Communicated by Guido De Philippis}
\end{center}
\vspace{1cm}

%%      -------------------------------------------------------------------------------
%%      -------------------------- BEGIN ABSTRACT ----------------------------
%%      -------------------------------------------------------------------------------
%% Authors, please put here the ABSTRACT and KEYBOARDS

\noindent \textbf{Abstract.} \textit{We consider an area-minimizing integral current of dimension $m$ and codimension at least $2$ and fix an arbitrary interior singular point $q$ where at least one tangent cone is flat. For any vanishing sequence of scales around $q$ along which the rescaled currents converge to a flat cone, we define a suitable ``singularity degree" of the rescalings, which is a real number bigger than or equal to $1$. We show that this number is independent of the chosen sequence and we prove several interesting properties linked to its value. Our study prepares the ground for two companion works, where we show that the singular set is $(m-2)$-rectifiable and the tangent cone is unique at $\mathcal{H}^{m-2}$-a.e. point.}
\vskip0.3cm

\vspace{0.5cm}

%%      -------------------------------------------------------------------------------
%%      -------------------------- BEGIN ARTICLE ----------------------------
%%      -------------------------------------------------------------------------------
%% Authors, copy the body of your paper here

\tableofcontents
\allowdisplaybreaks
\section{Introduction}

Suppose that $T$ is an $m$-dimensional integral current in a complete smooth Riemannian manifold $\Sigma$.
%, which for simplicity we will assume to be properly embedded in an open subset of a sufficiently large Euclidean space. 
We assume that $T$ is area-minimizing in some (relatively) open $\Omega\subset \Sigma$, i.e. 
\[
\mathbf{M} (T+\partial S) \geq \mathbf{M} (T)
\]
for any $(m+1)$-dimensional integral current $S$ supported in $\Omega$. A point $p\in \spt (T)$ is called an interior regular point if there is a ball $\Bbf_r (p)$ in which $T$ is, up to multiplicity, an {\em embedded} submanifold of $\Sigma$ without boundary in $\Bbf_r (p)$. Its complement in $\spt (T)\setminus \spt (\partial T)$ is called the interior singular set and from now on will be denoted by $\Sing (T)$.

Determining the size and structure of $\Sing (T)$ is a problem that has attracted a lot of interest for several decades. The answer depends sensibly on the codimension of $T$ in $\Sigma$. If the codimension is one, the works of De Giorgi, Fleming, Almgren, Simons, and Federer in the sixties and early seventies show that the Hausdorff dimension of $\Sing (T)$ is at most $m-7$, cf. \cite{Federer}. Moreover, the bound is optimal in view of the famous Simons' cone, cf. \cites{Simons, BDG}. The monograph of Almgren~\cite{Almgren_regularity} showed in the early eighties that when the codimension is higher than one, the Hausdorff dimension of $\Sing (T)$ is at most $m-2$, and Almgren's theory has since been simplified and made more transparent in the series of works~\cites{DLS_MAMS, DLS_multiple_valued, DLS14Lp, DLS16centermfld, DLS16blowup}. Almgren's bound is also sharp, given that every holomorphic subvariety of a K\"ahler manifold is an area-minimizing integral current.

In the nineties Simon proved (see~\cite{Simon_rectifiability}) that in codimension one, $\Sing (T)$ is $(m-7)$-rectifiable. Much more recently, Naber and Valtorta in \cite{NV_varifolds} showed that it has locally finite $\mathcal{H}^{m-7}$-measure. In fact \cite{NV_varifolds} exploits the groundbreaking ideas of the earlier work \cite{NV_Annals} to recover at the same time the latter information {\em and} the rectifiability, using independent techniques to Simon. The work of Simon, however, implies also the uniqueness of the tangent cone at $\mathcal{H}^{m-7}$-a.e. point in $\spt (T)\setminus \spt (\partial T)$. 
The aim of this and its two companion works \cites{DLSk2, DMS} is to prove the following counterpart of Simon's theorem in higher codimension.

\begin{theorem}\label{t:big-one}
Let $T$ be an $m$-dimensional area-minimizing current in a $C^{3,\kappa_0}$ complete Riemannian manifold of dimension $m+\bar{n}\geq m+2$, with $\kappa_0>0$. Then $\Sing (T)$ is $(m-2)$-rectifiable and there is a unique tangent cone at $\mathcal{H}^{m-2}$-a.e. $q\in \Sing (T)$. 
\end{theorem}

%The tools introduced in \cites{NV_Annals,NV_varifolds} will play a pivotal role in our forthcoming works \cites{DLSk2,DLSk3}. 
%In the present note we assume that the codimension is strictly larger than one. 

Theorem \ref{t:big-one} can in fact be improved in the case of $m=2$, in which it is known that the singularities are isolated, cf. \cite{SXChang} and \cites{DLSS1,DLSS2,DLSS3}. Note also that the uniqueness of tangent cones in the latter case is known since the work of White in the eighties, cf. \cite{White}. In higher dimensions the regularity of $\Sing (T)$ given by Theorem \ref{t:big-one} is optimal, as the recent work \cite{Liu} shows that $\Sing (T)$ can be a fractal with arbitrary dimension $\kappa\leq m-2$. It is however possible to improve the rectifiability statement if one takes a less stringent definition of $\Sing (T)$, because the examples of \cite{Liu} are locally {\em immersed} submanifolds. Moreover, our techniques are far from showing that $\Sing (T)$ has locally finite $\mathcal{H}^{m-2}$-measure, which could be expected, and the general uniqueness of tangent cones remains widely open. 

\subsection{Flat singularities}  The main issue is to establish the $(m-2)$-rectifiability of those singular points where at least one tangent cone is supported in an $m$-dimensional plane, since the remaining portion of the singular set is, by \cite{NV_varifolds}, $(m-2)$-rectifiable. However, we independently establish the $(m-2)$-rectifiability of the singular points with non-flat tangent cones as a consequence of our work \cite{DMS}. From now on if a tangent cone is supported in an $m$-dimensional plane we will call it {\em flat} and a $p\in \Sing (T)$ with at least one flat tangent cone will be called a {\em flat singular point}. We know from the constancy theorem (cf. \cite{Federer}) that a flat tangent cone at a point $q$ must be an oriented $m$-dimensional plane counted with a positive integer multiplicity $Q$. The latter is indeed the density of the current at $q$ and Allard's celebrated regularity theorem \cite{Allard_72} guarantees that if $Q=1$ the point is regular. 
We emphasize that the striking difference in complexity between the codimension one case and the case of higher codimension hinges on the fact that, in higher codimension, flat singular points might exist, while they cannot in codimension one. The latter phenomenon is due to the local characterization of integral hypercurrents as superpositions of boundaries of Caccioppoli sets (cf.~\cite{Simon_GMT}*{Theorem~27.6, Corollary~27.8}), which is very specific to the codimension one setting. The typical examples of area minimizers with flat singular points in higher codimension are branching singularities of holomorphic subvarieties of K\"ahler manifolds. Note moreover that the uniqueness of the tangent cone is still unknown at flat singular points, even under the stronger assumption that {\em all} tangent cones at the considered point are flat. 

In this paper we will be concerned with the definition and properties of a suitable notion of ``singularity degree'' of $T$ at flat singular points. This is a real parameter which will be then used to suitably subdivide the set of flat singular points of $T$. 

\begin{example}\label{e:key-example}
We illustrate the intuition behind the singularity degree in the example of a holomorphic curve in $\mathbb C^2$, defined by
\[
    \Lambda := \{w^Q = z^p : (z,w)\in \mathbb C^2\}\, .
\]
In this example we require that:
\begin{itemize}
\item $p>Q\geq 2$ are coprime integers;
%\item $h$ and $k$ are holomorphic functions;
\item $k (0)\neq 0$.
\end{itemize}
Recall that, by Federer's classical theorem, $\Lambda$ (with the standard orientation given by the complex structure) induces a $2$-dimensional integral area-minimizing current $T=\llbracket \Lambda \rrbracket$ in $\mathbb R^4 \cong \mathbb C^2$. Since $p$ is not a multiple of $Q$ and the latter is strictly larger than $1$, the origin is an interior singular point of $T$. Moreover, since $p$ and $Q$ are coprime and $p$ is larger than $Q$, the (unique) tangent cone to $T$ at $0$ is given by $Q \llbracket\{w=0\}\rrbracket$. In this particular example our notion of singularity degree of $T$ at the flat singular point $0$ gives the number $p/Q$. 
\end{example}

\subsection{Singularity degree} A priori we have very little knowledge of the structure of the singularities at a general flat singular point of an area-minimizing current of arbitrary dimension and codimension. Thus, our definition of singularity degree will necessarily be somewhat involved. In particular, given a flat singular point $q$, we will first identify a suitable analytical definition of singularity degree for a given infinitesimal sequence $\{r_k\}$ of blow-up scales along which the rescaled currents $T_{q,r_k}$ (cf. Section \ref{s:setup} for the definition) converge to a flat tangent cone. These numbers, which might depend on $\{r_k\}$, will be called {\em singular frequency values}, cf. Definition \ref{def:freq_value}. The singularity degree of $T$ at a flat singular point $x$ will then be defined as the infimum of the singular frequency values at $x$, cf. Definition \ref{def:degree}. We will prove a series of interesting properties related to the singularity degree, among which we select the following three:
\begin{itemize}
    \item[(i)] we will show that the singularity degree is necessarily at least $1$, due to the Hardt-Simon inequality and we will show that the singular frequency values all coincide with the singularity degree, i.e. they are the same number, independent of the subsequence, cf. Theorem \ref{thm:uniquefv};
    \item[(ii)] for each infinitesimal blow-up scale we will, up to extraction of a subsequence, identify a suitable rescaled limit, which will be an homogeneous multivalued function and whose degree of homogeneity is indeed the singularity degree, cf. Theorem \ref{t:consequences}(i); 
    \item[(iii)] when the singularity degree is strictly larger than $1$ we will show that the (flat) tangent cone at $x$ is unique and the current decays to it polynomially fast, cf. Theorem \ref{t:consequences}(iv).
\end{itemize}
In the work~\cite{DLSk2} we will then show that the set of flat singular points where the singularity degree is strictly larger than $1$ is $(m-2)$-rectifiable while in \cite{DMS} we will complete the proof by showing that the set of flat singular points where the singularity degree is $1$ is $\mathcal{H}^{m-2}$ negligible. Concerning the uniqueness of the tangent cone, in this paper we show that it is unique at flat singular points where the singularity degree is strictly larger than $1$, while \cite{DMS} will complete the proof by showing $\mathcal{H}^{m-2}$-a.e. uniqueness. 

The three properties (i)-(ii)-(iii) will be fundamental in establishing the proof of Theorem \ref{t:big-one}, however they are not the only important points from this paper which will be heavily used in \cites{DLSk2,DMS}, for instance the BV estimate of Proposition \ref{prop:bv} is crucial for \cite{DLSk2}.

\subsection{Comparison with the work of Krummel \& Wickramasekera} 

At the same time this and the accompanying works \cites{DLSk2,DMS} were being finished, Krummel \& Wickramasekera independently were completing a program also establishing Theorem \ref{t:big-one} (see \cites{KW1,KW2,KW3}). Here we take a moment to discuss the differences and similarities between the two programs, each point addressing a key aspect of each of the three papers in each of the programs. One underlying theme in both programs is to relate structural properties of the singular set to the rate of decay of the current at certain points to its tangent cone. 
\begin{itemize}
    \item In both approaches a monotonicity formula plays an important role in the first step. In our approach, Almgren's monotonicity formula enters to associate to flat singular points (namely, singular points at which at least one tangent cone is supported on a plane) a real number, referred to as the singularity degree, which takes values at least 1. This number is morally the infinitesimal homogeneity of the current relative to the average of its ``sheets'' (the role of which is played by center manifolds which are possibly varying with the scale). A byproduct is that, when the singularity degree is strictly larger than 1, the rate of decay to the tangent plane is at least a power law. This is accomplished in the present paper.
    In their approach, Krummel \& Wickramasekera 
 define a ``planar frequency function'' at the level of the current (see \cite{KW1}), whose definition does not require the introduction of a center manifold, and show that it satisfies a suitable approximate monotonicity whenever the current is decaying to a plane on some interval of radii about a given point. Using this, they prove a certain decomposition theorem holds for the singular set, namely that locally about points of density $Q$ (for given $Q\in \mathbb{Z}_{\geq 1}$), the singular set splits into two disjoint sets, namely a relatively closed set (denoted in \cite{KW1} by $\mathcal{B}$) where the current is decaying with a power law at all scales to a tangent plane with a fixed lower bound on the decay rate, and a set which satisfies a uniform weak approximation property. The latter set could still contain flat singular points. In our approach the analogous set to $\mathcal{B}$ would be the intersection of $\flatS_{Q,\geq 1+\delta}(T)$ with some appropriately small ball and for some appropriate choice of the small threshold $\delta$ (we refer the reader to \cite{DLSk2} for the precise definition). Strictly speaking the two sets do not coincide because the set $\mathcal{B}$ in \cite{KW1} has some uniform control in the prefactor of the power-lay decay to the unique flat tangent. This uniform control could possibly be achieved by making some of our arguments more quantitative.
    \item In both cases, one exploits the power law decay rate at each ``good'' flat singular point (i.e. points where the singularity degree is strictly larger than 1 in our setting, whilst for Krummel \& Wickramasekera it is the subset $\mathcal{B}$ described above), in order to prove $(m-2)$-rectifiability for this subset. For our program, this is achieved in \cite{DLSk2}, whilst for Krummel \& Wickramasekera this is achieved in forthcoming work \cite{KW3}. However, in Krummel \& Wickramasekera's work, the construction of a center manifold is only needed to study flat singular points where not only is the tangent plane unique, but additionally the current is decaying at least quadratically to this tangent plane. In such a setting, the center manifold construction is much simpler (one does not need to deal with intervals of flattening or changing center manifolds as described in Section \ref{s:setup}, for example). The reason for this is that they are able to study the set of flat singular points in the set $\mathcal{B}$ described above at which the decay rate to the tangent plane is a power law with order strictly less than 2 via their planar frequency function. See Section \ref{ss:comparison} for a more in-depth discussion of this matter.
    \item In both approaches one must also deal with ``slowly decaying'' flat singular points; in our works this is when the decay value is exactly $1$ and for Krummel \& Wickramasekera these points are contained in the second set of their decomposition theorem described above. This part is highly non-trivial, and in both programs it is shown that the relevant set is $\mathcal{H}^{m-2}$-null. For us, this is addressed in \cite{DMS} and for Krummel \& Wickramasekera this is handled in \cite{KW2}.
\end{itemize}
It should be noted that aside from the definition of our singularity degree a priori requiring center manifolds (which are a posteriori not necessary in the slow decay case), the order of the last two points above is irrelevant for concluding the program. One could conduct them in either order, and indeed in our case the last point above is chronologically the last step whilst in Krummel \& Wickramasekera's program it is the second step.

One difference between the two sets of works is that our results are all in the general setting of a sufficiently smooth ambient Riemannian manifold, whilst the statements of \cites{KW1,KW2,KW3} are in the Euclidean setting. However, we believe that this is also just a technical matter and not a substantial difference.

Two other differences have already been pointed out above:
\begin{itemize}
    \item[(i)] Whilst Krummel \& Wickramasekera show that the set of singular points without a power law decay rate of some fixed small order to a unique tangent plane is $\mathcal{H}^{m-2}$-null, we show that the set of points with singularity degree exactly equal to $1$ is $\mathcal{H}^{m-2}$-null. The former corresponds to points where our singularity degree is between $1$ and $1+\delta$, for a sufficiently small choice of $\delta>0$. 
    \item[(ii)] Whilst Krummel \& Wickramasekera get a uniform decay estimate for their set $\mathcal{B}$, we do not pursue this for the corresponding set $\flatS_{Q,\geq 1+\delta}(T)$ in our approach and we instead subdivide it in a countable unions of sets for which the rate and the starting scale for the decay is uniform. In \cite{DLSk2} these sets are denoted by $\mathfrak{S}_{K,J}$ for those points with subquadratic decay, and a single set $\Sbf$ for the points with superquadratic decay (here the starting decay scale is shown to be locally uniform). 
\end{itemize}
The combination of (i) and (ii) allow Krummel \& Wickramasekera to achieve the additional conclusion that in fact the set of flat singular points in a sufficiently small neighborhood $U$ of a point of density $Q$ can be decomposed into the union of finitely many sets, say $F_1\cup \ldots\, F_N$, each of which has {\em locally} finite $\mathcal{H}^{m-2}$ measure. In fact they show that $\mathcal{B}$ enjoys the latter structure while the flat singularities in its complement form an $\mathcal{H}^{m-2}$-null set. We caution the reader that this decomposition does not yield the finiteness of the measure of the whole set of flat singular points in $U$ because the sets $F_i$ are not apriori closed. 

\medskip

This raises the natural question of whether our approach is also amenable to yield similar conclusions. We in fact do not believe that (i) is a substantial obstacle for our approach and we think that it is possible to achieve an analogous statement (see \cite{DMS} for a more detailed explanation). Concerning point (ii) we also believe that a suitable refinement of our argument can achieve a uniform decay estimate directly for $\flatS_{Q, \geq 1+\delta}(T)$ in a sufficiently small neighborhood of a point of density $Q$. These considerations are obviously influenced by the insight learned from the works of Krummel \& Wickramasekera.

Provided one can prove the analogous statements to (i) and (ii) in our case (or using the estimates of Krummel and Wickramasekera in combination with our techniques, when the ambient is the Euclidean space), our approach in \cite{DLSk2} would yield the conclusion that $\flatS_{Q,\geq 1+\delta}(T)$ can be decomposed into two sets with locally finite $\mathcal{H}^{m-2}$ measure and that the flat singular points in its complement form an $\mathcal{H}^{m-2}$-null set.
In fact, since in our paper we use a modification of the Naber-Valtorta approach, these two sets would have locally finite $(m-2)$-dimensional Minkowski content. In order to tackle the question of whether $\flatS_{Q, \geq 1+\delta}(T)$ itself has locally finite Minkowski content, one would need instead to suitably modify the arguments in \cite{DLSk2} in order to tackle low frequency and high frequency points at the same time, a task which is certainly more challenging.

Finally, Krummel \& Wickramasekera additionally establish the existence of a unique non-zero (multi-valued) Dirichlet-minimizing tangent function at $\mathcal{H}^{m-2}$-a.e. flat singular point of the current. This is inherently different from our approach in \cite{DLSk2}, given that one major point of the Naber-Valtorta technique is being able to tackle the rectifiability question without addressing the uniqueness of the tangent functions.

\subsection*{Acknowledgments}
C.D.L. and A.S. acknowledge the support of the National Science Foundation through the grant FRG-1854147.

\section{Main statements}\label{s:setup}

In this section we define the singular fequency values and the singularity degree and give the main statements. We follow heavily the notation and terminology of the papers \cites{DLS16centermfld,DLS16blowup} and from now on we will always make the following assumption.

\begin{assumption}\label{asm:1}
    $T$ is an $m$-dimensional integral current in $\Sigma\cap \Omega$ with $\partial T\mres \Omega = 0$, where $\Omega$ is an open set of $\mathbb R^{m+n} = \mathbb R^{m+\bar n + l}$ and $\Sigma$ is an $(m + \bar{n})$-dimensional embedded submanifold of class $\Crm^{3,\kappa_0}$ with $\kappa_0>0$. $T$ is area-minimizing in $\Sigma\cap \Omega$ and $\bar{n} \geq 2$. $0\in \Omega$ is a flat singular point of $T$ and $Q\in \mathbb N\setminus \{0,1\}$ is the density of $T$ at $0$. 
\end{assumption}

We will henceforth let $C$ and $C_0$ denote dimensional constants, depending only on $m,n,Q$. The currents $T_{q,r}$ will denote the dilations $(\iota_{q,r})_\sharp T$, where $\iota_{q,r} (x):= \frac{x-q}{r}$. Since our statements are invariant under dilations, we can also assume that 

\begin{assumption}\label{asm:2}
$T$ and $\Sigma$ satisfy Assumption \ref{asm:1} with $\Omega = \Bbf_{7\sqrt{m}}$ and $\Sigma \cap \Bbf_{7\sqrt{m}}(p)$ is the graph of a $\Crm^{3,\kappa_0}$ function $\Psi_p : \Trm_p\Sigma \cap \Bbf_{7\sqrt{m}}(p) \to \Trm_p\Sigma^\perp$ for every $p \in \Sigma\cap\Bbf_{7\sqrt{m}}$. Moreover
    \[
        \boldsymbol{c}(\Sigma)\coloneqq\sup_{p \in \Sigma \cap \Bbf_{7\sqrt{m}}}\|D\Psi_p\|_{\Crm^{2,\kappa_0}} \leq \bar{\eps},
    \]
    where $\bar\eps$ is a small positive constant which will be specified later.
\end{assumption}
    
In particular the following uniform control on the second fundamental form $A_\Sigma$ of $\Sigma\cap\Bbf_{7\sqrt{m}}$ holds:
\[
        \Abf \coloneqq \|A_\Sigma\|_{\Crm^0(\Sigma\cap\Bbf_{7\sqrt{m}})} \leq C_0\boldsymbol{c} (\Sigma) \leq C_0 \bar\varepsilon.
\]
Following \cite{DLS16blowup}*{Section 2} we introduce appropriate disjoint intervals $]s_j, t_j]\subset ]0,1]$, called {\em intervals of flattening}, the union of which contains\footnote{It is not necessarily true that the inequality $\mathbf{E} (T, \Bbf_{6\sqrt{m} r})\leq \varepsilon_3^2$ holds for all $r\in ]s_j, t_j]$. However the inequality certainly holds at all $r= t_j$, while for the remaining radii in the interval holds up to a suitably fixed constant $C$, cf. \cite{DLS16blowup}.} those radii $r$ such that the spherical excess $\mathbf{E} (T, \Bbf_{6\sqrt{m} r})$ (cf. \cite{DLS16centermfld}*{Definition 1.2} for the definition) falls below a positive fixed threshold $\varepsilon_3^2$. Arguing as in \cite{DLS16blowup}*{Section 2} for each rescaled current $T_{0, t_j}$ and rescaled ambient manifold $\Sigma_{0, t_j}$ we follow the algorithm detailed in \cite{DLS16centermfld} to produce a {\em center manifold} $\mathcal{M}$ and an appropriate multivalued map $N : \mathcal{M} \to \mathcal{A}_Q (\mathbb R^{m+n})$. The latter takes values in the normal bundle of $\mathcal{M}$ and gives an efficient approximation of the current $T_{0, t_j}$ in $\Bbf_3\setminus \Bbf_{s_j/t_j}$. For technical reasons, however, we will use a slightly different definition for the parameter $\boldsymbol{m}_0$ in \cite{DLS16centermfld}*{Assumption 1.3}. Our $\boldsymbol{m}_0$, which we denote by $\boldsymbol{m}_{0,j}$ to underline the dependence on $j$, is defined as
\begin{equation}\label{eq:m_0}
\boldsymbol{m}_{0,j} := \max \{ \mathbf{E} (T_{0, t_j}, \mathbf{B}_{6\sqrt{m}}), \bar\varepsilon^2 t_j^{2-2\delta_2}\}\, , 
\end{equation}
where $\delta_2>0$ is the parameter in \cite{DLS16centermfld}*{Assumption 1.8}.
It can be readily checked that this change is of no consequence for the conclusions of \cites{DLS16centermfld,DLS16blowup}, the relevant point is that, because of simple scaling considerations, $\mathbf{c} (\Sigma_{0, t_j}) \leq \boldsymbol{m}_{0,j}$, therefore all the estimates claimed in \cites{DLS16centermfld,DLS16blowup} are valid with our different choice of parameter $\boldsymbol{m}_{0,j}$, provided we choose it to fall below the same threshold $\eps_3$ as in \cite{DLS16blowup}. In light of this, we will henceforth make the following assumption.

\begin{assumption}\label{asm:3}
    $T$ and $\Sigma$ satisfy Assumption \ref{asm:2}. The parameter $\bar \eps$ is chosen small enough so that $\boldsymbol{m}_{0,0} \leq \eps_3^2$.
\end{assumption}

Before proceeding we record a fact proved in \cite{DLS16blowup}, which is however not explicitly stated there.

\begin{lemma}\label{l:finite-vs-infinite}
Suppose that $T$ and $\Sigma$ are as in Assumption \ref{asm:3}. If $\{j_i\}\subset \mathbb N$ is the set of indices such that $t_{j_i}< s_{j_i-1}$, then either the latter is finite (i.e. $\bigcup_j ]s_j, t_j]$ contains some open interval $]0, \rho[$), or 
\begin{equation}\label{e:restarting_positive}
\liminf_i \mathbf{E} (T_{0, t_{j_i}}, \mathbf{B}_{6\sqrt{m}}) \geq \varepsilon_3^2\, . 
\end{equation}
\end{lemma}

For the sake of clarity, we prove this again here; see Section~\ref{ss:proof-finite-vs-infinite}. Since we will repeatedly use it throughout the rest of the paper, it is convenient to introduce the following terminology.

\begin{definition}\label{d:blow-up-sequence}
Let $T$ and $\Sigma$ be as in Assumption \ref{asm:1}. A {\em blow-up sequence of radii} $\{r_k\}$ is a vanishing sequence of positive real numbers such that $T_{0,r_k}$ converges to a flat tangent cone. 
\end{definition}

Of course a similar concept can be introduced by considering a different flat singular point $x$ instead of the origin. In that case we will say that the sequence is a {\em blow-up sequence at the flat singular point $x$}.

Note that, having fixed a blow-up sequence $\{r_k\}$, for every $k$ sufficiently large there is a unique $j(k)$ such that $r_k \in ]s_{j(k)}, t_{j(k)}]$ and we use the following shorthand notations:
\begin{itemize}
    \item $T_k$ and $\Sigma_k$ for the rescaled currents $T_{0, t_{j(k)}} \res \Bbf_{6\sqrt{m}}$ and ambient manifolds $\Sigma_{0, t_{j(k)}}$;
    \item $\Mcal_k$ and $N_k$ for the corresponding center manifolds and normal approximations of $T_k$.
\end{itemize}

\subsection{Compactness procedure}\label{ss:compactness} Let $T$ satisfy Assumption \ref{asm:3} and let $\frac{\bar{s}_k}{t_{j(k)}} \in \big]\frac{3r_k}{2 t_{j(k)}}, \frac{3r_k}{t_{j(k)}}\big]$ be the scale at which the reverse Sobolev inequality~\cite{DLS16blowup}*{Corollary~5.3} holds for $r = \frac{r_k}{t_{j(k)}}$. Then let $\bar{r}_k \coloneqq \frac{2\bar{s}_k}{3t_{j(k)}} \in \big]\frac{r_k}{t_{j(k)}}, \frac{2r_k}{t_{j(k)}}\big]$. We rescale further the currents $T_k$, the ambient manifolds $\Sigma_k$ and the center manifolds to
\[
    \bar{T}_k \coloneqq (\iota_{0,\bar{r}_k})_\sharp T_k =  \big((\iota_{0,\bar{r}_k t_{j(k)}})_\sharp T\big)\mres \Bbf_{\frac{6\sqrt{m}}{\bar{r}_k}}, \qquad \bar{\Sigma}_k \coloneqq \iota_{0,\bar{r}_k} (\Sigma_k), \qquad \bar{\Mcal}_k \coloneqq \iota_{0,\bar{r}_k} (\Mcal_k)\, .
\]
%and let
%\[
%    \bar{\boldsymbol{m}}_0^{(k)} \coloneqq \max\{\boldsymbol{c}(\bar{\Sigma}_k)^2, \Ebf(\bar{T}_k, \Bbf_{6\sqrt{m}})\}.
%\]
Define
\[
    \bar{N}_k: \bar{\Mcal}_k \to {\Acal_Q(\R^{m+n})}, \qquad \bar{N}_k(p) \coloneqq \frac{1}{\bar{r}_k} N_k(\bar{r}_k p),
\]
and let
\[
    u_k \coloneqq \frac{\bar{N}_k \circ \textbf{e}_k}{\mathbf{h}_k}, \qquad u_k:\pi_k \supset B_3 \to \Acal_Q(\R^{m+n}),
\]
where $\textbf{e}_k$ is the exponential map at $p_k \coloneqq \frac{\boldsymbol{\Phi}_k(0)}{\bar{r}_k} \in \bar{\Mcal}_k$ defined on $B_3 \subset \pi_k \simeq T_{p_k} \bar{\Mcal}_k$ and $\mathbf{h}_k \coloneqq \|\bar{N}_k\|_{L^2(\Bcal_{3/2})}$. The reverse Sobolev inequality of \cite{DLS16blowup}*{Corollary~5.3} gives a uniform control on the $W^{1,2}$ norm of $u_k$ on $B_{3/2} (0, \pi_k)$ (which denotes the unit disk of $\pi_k$ centered at $0$ and with radius $3/2$).

Then, following the proof of~\cite{DLS16blowup}*{Theorem~6.2}, there exists a subsequence (not relabeled) a limiting $m$-plane $\pi_0$ and a Dir-minimizing map $u \in \Wrm^{1,2}(B_{3/2}(0, \pi_0);\Acal_Q(\pi_0^\perp))$ with $\boldsymbol{\eta}\circ u = 0$ and $\|u\|_{L^2(B_{3/2})} = 1$, such that (after we apply a suitable rotation to map $\pi_k$ onto $\pi$)
\begin{equation}\label{eq:compactness}
    u_k \longrightarrow u \quad \text{strongly in $\Wrm^{1,2}_\loc\cap L^2$}.
\end{equation}
Recall that Almgren's famous frequency function for Dir-minimizers $u: \Omega \subset \R^m \to \Acal_Q(\R^n)$ at a center point $x \in \Omega$ and scale $r > 0$ is defined by
\[
\frac{r \int_{B_r(x)} |Du|^2}{\int_{\partial B_r(x)} |u|^2}\, .
\]
We refer the reader to~\cite{DLS_MAMS}*{Chapter~3} for the basic properties of the frequency function. The monotonicity of the frequency function~\cite{DLS_MAMS}*{Theorem~3.15} for Dir-minimizers yields existence of the limit as $r\downarrow 0$. It is more convenient to work with a smoother version of the frequency function, which has more robust convergence properties. Following \cite{DLS16centermfld} we consider a Lipschitz cut-off function $\phi: [0,\infty) \to [0,1]$ which vanishes identically for $t$ sufficiently large, equals $1$ for $t$ sufficiently small and is monotone nonincreasing. We then introduce
\begin{align*}
D_{u} (x,r) &:= \int |Du (y)|^2 \phi \left(\frac{|y-x|}{r}\right)\, dy\, ,\\
H_{u} (x,r) &:= -\int \frac{|u(y)|^2}{|y-x|} \phi' \left(\frac{|y-x|}{r}\right)\, dy\, , \\
I_{u} (x,r) &:= \frac{r D_{u} (x,r)}{H_{u} (x,r)}\, .
\end{align*}
The same computations showing the monotonicity of Almgren's frequency function for Dir-minimizers apply to the latter smoothed variant (cf. for instance
\cite{DLS16centermfld}*{Section 3}; note that Almgren's frequency function corresponds, formally, to the choice $\phi = {\mathbf{1}}_{[0,1]}$). Moreover, it can be readily checked that all these smoothed frequency functions are constant when the map is radially homogeneous, and this constant is the degree of homogeneity of the map. It follows then from the arguments in \cite{DLS_MAMS}*{Section 3.3, Section 3.5} that the limit
\[
I_{x,u} (0) = \lim_{r\downarrow 0} I_u (x, r)
\]
is independent of the weight $\phi$, and $I_{x,u}(0) \geq c(m,Q)>0$ whenever $u(x) = Q\llbracket 0 \rrbracket$. For the rest of the paper we will fix a convenient specific choice of $\phi$, given by
\begin{equation}\label{e:def_phi}
\phi (t) =
\left\{
\begin{array}{ll}
1 \qquad &\mbox{for $0\leq t \leq \frac{1}{2}$}\\
2-2t \quad &\mbox{for $\frac{1}{2}\leq t \leq 1$}\\
0 &\mbox{otherwise}\, .
\end{array}
\right.
\end{equation}
When $x=0$, we will omit the dependency on $x$ for $I$ and related quantities, and will merely write $I_{u}(r)$.

\begin{definition}\label{def:freq_value}
Any map $u$ as defined by the above compactness procedure is called a {\em fine blow-up} limit along the sequence $r_k$ and the set
\[
    \Fcal(T,0) \coloneqq \{ I_{u}(0) \ : \ \text{$u$ is a fine blow-up along some $r_k \todown 0$}\}
\]
is the \emph{set of singular frequency values of $T$ at $0$}.
\end{definition}
\begin{remark}
In the rest of the notes we will often omit the adjective ``singular''. 
The reason for using the adjective ``fine'' is that later on we will also introduce a notion of {\em coarse} blow-up, cf. Definition \ref{d:coarse}.
\end{remark}

\begin{definition}\label{def:degree}
The \emph{singularity degree} of $T$ at the flat singular point $0$ is defined as 
\[
    \Irm(T,0) \coloneqq \inf \{\alpha : \alpha \in \Fcal(T,0)\}\, .
\]
\end{definition}

A simple translation allows to extend all the definitions above to any flat interior singular point $x$ of $T$. We will therefore use $\Irm (T,x)$ and $\mathcal{F} (T,x)$ for the singularity degree and the frequency values of $T$ at such an $x$.

\subsection{Main results}\label{sec:main}
We are now in a position to state the main results of this article. Our primary result here is the following.

\begin{theorem}\label{thm:uniquefv}
Assume that $T$ satisfies Assumption \ref{asm:3}. Then $\Irm (T, 0)\geq 1$ and $\Fcal(T,0) = \{\Irm(T,0)\}$, i.e. there is one unique frequency value for $T$ at $0$ and it coincides with the singularity degree.
\end{theorem}

However, our analysis delivers a number of additional pieces of information. We report them here even though some statements will need notions which will be only introduced in the next sections.

\begin{theorem}\label{t:consequences}
Under the same assumptions of Theorem \ref{thm:uniquefv} the following holds:
\begin{enumerate}[\normalfont(i)]
    \item\label{itm:consequences1} All fine blow-ups are radially homogeneous and their homogeneity degree is $\Irm (T, 0)$. 
    \item\label{itm:consequences2} If $s_{j_0}=0$ for some $j_0$, then $\lim_{r\downarrow 0} {\mathbf{I}_{N_{j_0}}} (r) = \Irm (T, 0)$ (see below for the definition of {$\mathbf{I}_{N_{j_0}}$}).
    \item\label{itm:consequences3} If $\{s_j\}$ is infinite, then the functions {$\mathbf{I}_{N_j}$} converge uniformly to $\Irm (T,0)$ if $\Irm (T,0) >1$, while, when $\Irm (T,0)=1$, $\lim_{k\to\infty} {\mathbf{I}_{N_{j(k)}}(\tfrac{r_k}{t_{j(k)}})} = \Irm (T, 0) = 1$ for every blow-up sequence $r_k$ (recall that $j(k)$ is such that $r_k \in ]s_{j(k)}, t_{j(k)}]$). 
    \item\label{itm:consequences4} If $\Irm (T,0)>1$, then $T_{0,r}$ converge polynomially fast to a unique {flat} tangent cone as $r\downarrow 0$.
    \item\label{itm:consequences5} If $\Irm (T, 0) >  2-\delta_2$, then $s_{j_0} =0$ for some $j_0$.
    \item\label{itm:consequences6} If $\Irm (T,0) < 2-\delta_2$ then $\{s_j\}$ is infinite and $\inf_j \frac{s_j}{t_j} > 0$.
%   \item If $\Irm (T, 0) < 2$, fine blow-ups coincide, up to a positive factor, with the average-free parts of coarse blow-ups (cf. Definition \ref{d:coarse}). 
\end{enumerate}
\end{theorem}

\subsection{Rectifiability} Following Almgren (cf. also \cite{WhiteStrat}), the set ${\rm spt}\, (T)\setminus {\rm spt}\, (\partial T)$ can be stratified through
\[
\Scal^{(k)}(T) \coloneqq \left\{x \in \spt\, (T)\setminus\spt\, (\partial T) \ : \ \substack{\text{any tangent cone of T at $x$ splits off} \\ \text{no more than a $k$-dimensional subspace}}\right\}\, ,
\]
where $k =0, 1, \ldots , m$. In particular 
\[
\Scal^{(0)} (T) \subset \Scal^{(1)} (T) \subset \cdots \subset \Scal^{(m-1)} (T) \subset \Scal^{(m)} (T) = {\rm spt} (T)\setminus {\rm spt} (\partial T)\, .
\]
Almgren's argument (which can be seen as a suitable generalization of Federer's reduction argument, cf. \cite{Federer1970}) showed that 
\[
    \dim_{\Hcal}\big(\Scal^{(k)}(T)\big) \leq k\, .
\]
In their recent groundbreaking work~\cite{NV_varifolds}, Naber and Valtorta further proved that $\Scal^{(k)}(T)$ is $k$-rectifiable.
%and has locally finite $k$-dimensional upper Minkowski content. 
Moreover, due to the classification of one-dimensional area-minimizing cones (which are necessarily $1$-dimensional lines with integer multiplicity), $\Scal^{(m-1)}(T)\setminus \Scal^{(m-2)}(T) = \emptyset$. Finally, the set of flat singular points of $T$ (from now on denoted by $\flatS (T)$) is given by 
\[
\flatS (T) = {\rm Sing} (T) \setminus \Scal^{(m-1)} (T) = {\rm Sing} (T)\setminus \Scal^{(m-2)} (T)\, .
\]
Thus, proving the $(m-2)$-rectifiability of ${\rm Sing} (T)$ is equivalent to proving the $(m-2)$-rectifiability of $\flatS (T)$. In our forthcoming works \cites{DLSk2,DMS} the singularity degree will be used to further stratify $\flatS (T)$. The main result of \cite{DLSk2} will be the following

\begin{theorem}\label{t:rectifiability}
Let $T$ be as in Theorem \ref{t:big-one} Then the set $\{q\in \flatS (T) : \Irm (T,q) >1\}$
is $(m-2)$-rectifiable. 
\end{theorem}

Clearly, in view of the above theorem and of Theorem \ref{thm:uniquefv}, the remaining (challenging) step to prove the rectifiability of ${\rm Sing}\, (T)$ is to show that the set $\{q\in \flatS (T): \Irm (T,0)=1\}$ is $(m-2)$-rectifiable. In \cite{DMS} we will then show

\begin{theorem}\label{t:nullity}
Let $T$ be as in Theorem \ref{t:big-one}. Then
$\mathcal{H}^{m-2} (\{q\in \flatS (T) : \Irm (T,q) = 1\}) = 0$.
\end{theorem}

Combined with Theorem \ref{t:consequences} Theorem \ref{t:nullity} implies the uniqueness of the flat tangent cone at $\mathcal{H}^{m-2}$-a.e. flat singular point. To conclude the proof of Theorem \ref{t:big-one} in \cite{DMS} we will also show

\begin{theorem}\label{t:unique-cones}
The tangent cone is unique at $\mathcal{H}^{m-2}$-a.e. $p\in \mathcal{S}^{(m-2)} (T)$.
\end{theorem}

%Our heuristic behind the above conjecture is the following. Consider an idealized scenario in which we would be given the information that the tangent cones at singular points are unique. At points where the singularity degree is $1$ we know that
%the fine blow-ups would then be nontrivial 1-homogeneous multivalued Dir-minimizing maps with vanishing average $\mathbf{\eta}\circ u$. In particular we have the following alternative
%\begin{itemize}
%    \item[(a)] either $u = \sum_i \llbracket u_i\rrbracket$, where the $u_i$ are classical linear maps; 
%    \item[(b)] or $u$ can be at most $(m-3)$-symmetric, i.e. the dimension of the affine space of directions $v$ along which $u$ is invariant has dimension at most $m-3$. 
%\end{itemize}
%We would expect that the first alternative is excluded, since at an infinitesimal sequence of scales the current would then be much closer to a union of different planes than it is to a single plane with high multiplicity, which we hope would contradict the assumption that the unique cone is flat. This would then mean that the alternative (b) must hold, which in turn suggests that the set of such points has Hausdorff dimension at most $m-3$.

\subsection{Frequency function}\label{ss:frequency} We end the section by introducing a pivotal object in our arguments, the $\phi$-regularized frequency function of the normal approximation of $T$, cf. \cite{DLS16blowup}. Recalling the function $\phi: [0,\infty[$ of \eqref{e:def_phi}, for a given center manifold $\Mcal$ with corresponding $\Mcal$-normal approximation $N:\Mcal \to \Acal_Q(\R^{m+n})$, the $\phi$-regularized frequency function ${\Ibf_N(x,r)}$ of $N$ at a center point $x\in\Mcal$ and scale $r> 0$ is defined as follows:
\begin{align*}
    \Ibf_N(x,r) &\coloneqq \frac{r \Dbf_N(x,r)}{\Hbf_N(x,r)}\, ,
\end{align*}
where
\[
    \Dbf_N(x,r) \coloneqq \int_\Mcal |DN|^2\phi\left(\frac{d(y,x)}{r}\right) \dd y\, ,
\]
and
\[
    \Hbf_N(x,r) \coloneqq - \int_\Mcal \frac{|\nabla_y d(y, x)|^2}{d(y,x)} |N|^2\phi'\left(\frac{d(y, x)}{r}\right) \dd y\, \\
\]
Here $d$ is the geodesic distance on the center manifold $\Mcal$ {and we simply write $d(y)$ for the geodesic distance $d(0,y)$.} We additionally let $\mathbf{p}$ denote the orthogonal projection on $\Mcal$ (and we recall that, by the estimates in \cite{DLS16centermfld}, the points $x$ of interest, which belong to the support of $T$, are in a regular tubular neighborhood of $\Mcal$). Since we will often take the above quantities to be centered at $x=0$, we will omit the implicit dependency on $x$ most of the time.

A major starting point of our paper is the fact that the frequency function is bounded away from infinity and $0$ (independently of the choice of center manifold and corresponding normal approximation). The rightmost inequality is the most important analytical estimate of Almgren's regularity theory, while the left has been established only recently by the second author in \cite{Sk21}. More precisely, the following holds:

\begin{theorem}
Under the assumptions of Theorem \ref{thm:uniquefv}, 
\begin{equation}\label{e:frequency-twosided-bounds}
0 < \inf_j \inf_{r\in {\big]\tfrac{s_j}{t_j}, 3\big]}} \Ibf_{N_j} (r)
\leq \sup_j \sup_{r\in {\big]\tfrac{s_j}{t_j}, 3\big]}} \Ibf_{N_j} (r) < \infty\, .
\end{equation}
\end{theorem}

\subsection{Proof of Lemma \ref{l:finite-vs-infinite}}\label{ss:proof-finite-vs-infinite} The argument is taken from \cite{DLS16blowup}*{Proof of Theorem 5.1}, where the statement is shown in a step in the proof of the theorem. Observe that, by definition, we have
\[
\mathbf{E} (T_{0,r}, \mathbf{B}_{6\sqrt{m}}) > \varepsilon_3^2
\]
for all $r\in ]t_{j_i}, s_{j_i-1}[$. Pick a sequence $r_i\in ]t_{j_i}, s_{j_i-1}[$ with the property that $\frac{r_i}{t_{j_i}} \to 1$. Up to extraction of a subsequence, not relabeled, we can assume that $T_{0,t_{j_i}}$ converges to a tangent cone $S$ to $T$ at $0$. Note that $T_{0, r_i}$ converge to the same cone. Moreover, by the area minimizing property, we have that $\|T_{0, r_i}\| \toweakstar \|C\|$ and $\|T_{0, t_{j_i}}\| \toweakstar \|C\|$. Since $\|C\| (\partial \mathbf{B}_r) =0$ for every $r$, it follows immediately that $\|T_{0, r_i}\|\res \mathbf{B}_{6\sqrt{m}}\toweakstar \|C\|\res  \mathbf{B}_{6\sqrt{m}}$ and $\|T_{0, t_{j_i}}\|\res \mathbf{B}_{6\sqrt{m}}\toweakstar \|C\|\res \mathbf{B}_{6\sqrt{m}}$. These convergences can be easily seen to imply 
\begin{align*}
\lim_i \mathbf{E} (T_{0, t_{j_i}}, \mathbf{B}_{6\sqrt{m}}) &= \mathbf{E} (C, \mathbf{B}_{6\sqrt{m}}) = \lim_i \mathbf{E} (T_{0, r_i}, \mathbf{B}_{6\sqrt{m}}) \geq \varepsilon_3^2\, .
\end{align*}

\subsection{Comparison of this article with \texorpdfstring{\cite{KW1}}{KW}}\label{ss:comparison}

Let us compare in more detail the present article with its analogue \cite{KW1} in the program implemented by Krummel \& Wickramasekera discussed in the introduction. In both \cite{KW1} and this paper an almost monotone quantity plays a pivotal role. Here, this is Almgren's frequency function as defined in \cite{DLS16centermfld}. Instead in \cite{KW1} the authors introduce a new ``planar frequency function''. Rather than capturing the degree of singularity of the current at a flat singular point, the planar frequency function identifies the order of contact of the current with the flat tangent cone. Let us consider Example \ref{e:key-example} for an intuition: our singularity degree there is the number $p/Q$, while the planar frequency function at scale $0$ (with respect to the tangent plane $\{w=0\}$) coincides with $p/Q$ if the latter is smaller than the degree of the first nontrivial homogeneous polynomial in the Taylor expansion of $h$ at the origin. Otherwise, it coincides with the latter degree.

In fact, given that $\flatS_{Q,>1} (T)$ identifies the set of flat singular points at which there is a unique tangent cone to which the current decays with a power law rate, the latter coincides with those singular points where there is one plane for which the Krummel-Wickramasekera planar frequency function converges to a number larger than $1$, as the radius goes to $0$. 

As pointed out in the introduction, one significant difference of the approach in \cite{KW1} is that they avoid the requirement of introducing changing center manifolds at appropriate scales around those flat singular points where the decay to the cone is slow. As mentioned in \cites{KW1,KW2}, this in addition avoids the need for quite a few technical issues even to prove Almgren's original dimension bound. Indeed, here we a posteriori conclude that blowing up relative to center manifolds is not necessary for points with singularity degree between 1 and $2-2\delta_2$ (see Corollary \ref{c:coarse=fine}, \cite{DMS}*{Proposition 2.2}), but nevertheless for us the use of center manifolds is unavoidable to deduce this.

In the current work we instead establish a BV estimate on the frequency function (relative to varying center manifolds) which keeps the errors due to the change of center manifolds under control. In doing this, we capture the homogeneity of the first singular order in the expansion of the current. This way, we may use the same frequency function (relative to the center manifolds) in all of our arguments. We expect that, to conclude the rectifiability of those flat singular points which have a high order of contact with the tangent plane, in their forthcoming work \cite{KW3} Krummel \& Wickramasekera will need to resort to the frequency function with respect to the center manifold also, albeit only in the simpler setting.
Common to both approaches is that a suitable closeness of the current to a suitable reference plane is needed to get an almost monotonicity estimate for both frequency functions. 

The planar frequency function in \cite{KW1} depends only on the current and the reference plane, while the ones used here (and in the works \cites{DLS16centermfld,DLS16blowup}) depend on the current, the center manifold, {\em and the normal approximation}. Taking inspiration from \cite{KW1}, we believe that it is possible to eliminate the dependence on the latter approximation. If we denote by $\mathbf{p}$ the orthogonal projection on $\mathcal{M}$, we can substitute $r \mathbf{D}_N (x,r)$ with the ``curvilinear excess''
\[
r \int_{\Bbf_{2r} (x)} |\vec{T} (z) - \vec{\mathcal{M}} (\mathbf{p} (z))|^2 \phi \left(\frac{d (\mathbf{p} (z), x)}{r}\right) d\|T\| (z)
\]
and the height $\mathbf{H}_N (x,r)$ with a suitable squared $L^2$ distance of the current from $\mathcal{M}$
\[
\int_{\Bbf_{2r} (x)} |z-\mathbf{p} (z)|^2 \frac{|\nabla_y d (\mathbf{p} (z), x)|^2}{d (\mathbf{p} (z), x)} \phi' \left(\frac{d (\mathbf{p} (z), x)}{r}\right) d\|T\| (z)\, .
\]
The ratio of these two quantities differs from $\mathbf{I}_N (x,r)$ only by errors which can be bounded with suitable powers of the planar excess, as follows from the estimates in \cites{DLS16centermfld,DLS16blowup}. In particular this implies the almost monotonicity of the ``intrinsic ratio'' through the almost monotonicity of $\mathbf{I}_N (x,r)$. But in fact it is highly likely that appropriate variants of the computations in \cites{DLS16centermfld,DLS16blowup} prove directly the monotonicity of the ``intrinsic ratio''.

This also suggests the possibility of introducing a general frequency function, where $\mathcal{M}$ is replaced by any sufficiently regular surface with the same dimension as the current $T$. In view of the Taylor expansion of the area functional (see e.g. \cite{DLS_multiple_valued}), it is tempting to speculate that a suitable almost monotonicity will hold if one has a multi-valued map on the normal bundle of $\Mcal$ which approximates the current with a sufficiently high degree of accuracy and if one of the following two properties (or a suitable combination of the two) holds:
\begin{itemize}
    \item[(i)] The mean curvature of $\mathcal{M}$ vanishes, or it is asymptotically small as we approach the central point $x$;
    \item[(ii)] The average of the multi-valued approximation is asymptotically small as we approach the central point $x$. 
\end{itemize}

\section{The Hardt-Simon inequality and coarse blow-ups}\label{s:HS}

\subsection{Coarse blow-ups}\label{ss:coarse} Consider a blow-up sequence $\{r_k\}_k$ at the flat singular point $0$ and let:
\begin{itemize}
    \item $T_{0, r_k}$ be the corresponding rescaled currents;
    \item $\Sigma_{0, r_k}$ be the corresponding rescaled manifolds.
\end{itemize}
Without loss of generality we can assume that $T_{0,r_k}$ converges to $Q\llbracket \pi_0\rrbracket$ with $\pi_0=\mathbb R^m\times \{0\}$. {For $r_k:=\frac{r_k}{t_{j(k)}}$, where $]s_{j(k)},t_{j(k)}]$ is the interval of flattening containing $r_k$,} let $M > 0$ be large enough such that $\Bbf_L \subset \Cbf_{4M\bar{r}_k}$ for any $L \in {\Wscr^{j(k)}}$ with $L\cap \cl{B}_{\bar{r}_k}(0,\pi_0) \neq \emptyset$ (cf. \cite{DLS16centermfld} for the definitions). Consider further a sequence of planes $\pi_k$ with the property that $\pi_k$ optimizes the excess of $T_{0, r_k}$ in $\mathbf{B}_{8M}$ and {observe that for $k$ sufficiently large,}
\begin{align}\label{eq:E_k}
\Ebf(T_{0, r_k}, \Cbf_{4M}(\pi_k), \pi_k) \leq C \mathbf{E} (T_{0, r_k}, \mathbf{B}_{8M}) =: E_k \to 0\, , 
\end{align}
and define $\Abf_k \coloneqq \Abf_{\Sigma_{0, r_k}}$.
Clearly we must have $\pi_k\to \pi_0$. By applying a rotation which is infinitesimally close to the identity we can map $\pi_k$ to $\pi_0$. We then push forward the current $T_{0,r_k}$ under this rotation so that we can assume $\pi_k=\pi_0$, while, with a slight abuse of notation, we keep using $T_{0,r_k}$ and $\Sigma_{0,r_k}$ for the rotated objects.

If $k_0 \in \N$ is large enough, we can ensure that 
\begin{equation}\label{eq:smallexcess}
    E_k + \Abf_k^2 < \min\Big\{\eps_1, \frac{1}{2}\Big\} \qquad \text{for every $k \geq k_0$},
\end{equation}
where $\eps_1$ is the threshold in~\cite{DLS14Lp}*{Theorem~2.4}. We can therefore let $f_k : B_1(0,\pi_0) \to \Acal_Q(\pi_0^\perp)$ be the strong Lipschitz approximation of~\cite{DLS14Lp}*{Theorem~2.4} for $T_{0, r_k}$ and define the rescaled maps
\begin{equation}\label{eq:HSfreq}
        \bar{f}_k \coloneqq \frac{f_k}{E_k^{1/2}}\, .
    \end{equation}
We will make the additional assumption that 
\begin{align}
&\Abf_k^2 \leq C r_k^2 = o (E_k)\, .\label{e:A-E-infinitesimal}
\end{align}
Note that this does not need to hold in general, but we will verify that it holds whenever the sequence of blowup scales $r_k$ remain comparable to the stopping scales in their respective intervals of flattening; see Proposition \ref{p:coarse=fine}. It then follows from \cite{DLS14Lp} that, up to subsequences, 
\begin{itemize}
    \item[(i)] $\bar f_k$ converges strongly in $L^2\cap W^{1,2}_\loc (B_1 (0, \pi_0))$ to a Dir-minimizing map $\bar{f}: B_1(0,\pi_0) \to \Acal_Q(\pi_0^\perp)$,
    \item[(ii)] $\bar f$ takes values in the orthogonal complement to $\pi_0$ in $T_0 \Sigma$, 
    \item[(iii)] $\bar f (0) = Q \llbracket 0\rrbracket$.
 \end{itemize}
 The first two conclusions follow from \cite{DLS14Lp}*{Theorem 2.4, Theorem 2.6}, while the last conclusion is a consequence of the Hardt-Simon inequality \cite{Spolaor_15}*{(1.7)} for $T_{0,r_k}$, passed to the graphical approximation $f_k$ (see \cite{Sk21}*{Lemma 5.14} for analogous reasoning for the normal approximation).
 Note that there is no guarantee that the blow-up is nontrivial: the nontriviality of $\bar f$ is in fact equivalent (cf. \cite{DLS14Lp}) to
\begin{equation}\label{e:non-triviality}
{\liminf_{k\to\infty}} \frac{\Ebf (T_{0, r_k}, \Cbf_\rho, \pi_0)}{E_k} \geq \bar c> 0\,
\end{equation}
for some $\rho \in (0,1)$ and some $\bar c$.

\begin{definition}\label{d:coarse}
A Dir-minimizing map $\bar f$ as above will be called a {\em coarse} blow-up (at $0$). Its average free part is given by the map
\begin{equation}\label{e:average-free}
v (x) := \sum_i \llbracket \bar f_i (x) - \boldsymbol{\eta} \circ \bar f (x)\rrbracket\, .
\end{equation}
We say that $\bar{f}$ is nontrivial if it does not vanish identically.
\end{definition}

Obviously, if we focus our attention on some other flat singular point $q$, an obvious modification of the above procedure defines a notion of {\em coarse blow-up at $q$}. Our main claim for coarse blow-ups, which (as already pointed out) is a consequence of the Hardt-Simon inequality, is the following.

\begin{theorem}\label{thm:HS}
Let $T$ be as in Assumption \ref{asm:3}, $\bar f$ be a nontrivial coarse blow-up, and $v$ be its average-free part. Then $I_{\bar{f}} (0) \geq 1$ and,
if $v$ does not vanish identically, $I_v (0)\geq 1$.
\end{theorem}

In this section we prove Theorem \ref{thm:HS}. 

\subsection{Closure under rescalings} Before coming to the proof of Theorem \ref{thm:HS} we need the following elementary observation, which verifies that the property of being a coarse blow-up is closed under normalized $L^2$ limits.

\begin{lemma}\label{l:rescaling-closure}
Let $T$ be as in Assumption \ref{asm:3} and $\bar f$ be a nontrivial coarse blow-up. Let $\rho_j\downarrow 0$ be any vanishing sequence, let 
\[
D(j) := \int_{B_{\rho_j}} |D\bar f|^2,
\]
and define the rescaled maps $\bar f_j (x) := (\rho_j^{2-m} D(j))^{-1/2} \bar f (\rho_j x)$. If $\bar f_\infty$ is the $L^2$ limit of any subsequence of $\{\bar f_j\}$ on $B_1$, then $\bar f_\infty$ is (up to a nonzero multiplicative factor) also a nontrivial coarse blow-up.
\end{lemma}

\begin{proof} 
Let $r_k$ be a blow-up sequence with the property that the maps $\bar f_k$ defined in the previous section converge to $\bar f$ and fix constants $\bar \rho$ and $\bar c$ so that \eqref{e:non-triviality} holds. We consider a sequence $r'_j:= \rho_j r_{k(j)}$ %and $\pi_j := \pi_{k(j)}$
and we will show that, for an appropriate choice of $k (j)$, the following holds:
\begin{itemize}
    \item[(a)] $r'_j$ is a blow-up sequence, i.e. $T_{0, r'_j}$ converges to $Q\llbracket\pi_0\rrbracket$;
    \item[(b)] $\tilde{E}_j := \Ebf (T_{0, r'_j}, \Cbf_4, \pi_0)$ converges to $0$;
    \item[(c)] The conditions \eqref{e:A-E-infinitesimal} and \eqref{e:non-triviality} hold for this new blow-up sequence;
    \item[(d)] If $f_j$ are the approximating maps given by \cite{DLS14Lp}*{Theorem~2.4}, then $\tilde{E}_j^{-1/2} f_j$ converges (up to subsequences) to $\lambda \bar f_\infty$ for some nonzero scalar $\lambda$.
\end{itemize}
The argument is a classical diagonal one and in order to deal efficiently will all the conditions, it is useful to decouple the two indices and introduce the radii $r_{j,k} := \rho_j r_k$. We introduce then the corresponding excess $E_{j,k}:= \Ebf (T_{0, r_{j,k}}, \Cbf_4, \pi_0)$ and $\Abf_{j,k} := \Abf_{\Sigma_{0, r_{j,k}}}$. Combining the estimates of \cite{DLS14Lp}*{Theorem 2.4} with \eqref{e:A-E-infinitesimal} we immediately see that there are two positive constants $\tilde{c}^+$ and $\tilde{c}^-$ such that
\begin{equation}\label{e:rate-2}
c^- \leq \liminf_{j\to \infty} \frac{E_{j,k}\rho_j^m}{E_k D(j)}\leq \limsup_{j\to \infty} \frac{E_{j,k}\rho_j^m}{E_k D(j)}\leq c^+\, ,
\end{equation}
where $E_k$ is as in \eqref{eq:E_k}. Note that \eqref{e:A-E-infinitesimal} is required to control the $\Abf_{j,k}$ terms in the estimates of \cite{DLS14Lp}*{Theorem 2.4}. Moreover, obvious scaling arguments show that $\Abf_{j,k} \leq C \rho_j^2 \Abf_k$. It is then pretty obvious that the conditions corresponding to (a), (b), and (c) above hold for any sequence $\{r_{j,k}\}_j$ once we keep $k$ fixed. Observe also that for (c) we can choose constants which are independent of $k$: the radius $\bar \rho$ can in fact be taken equal to $\frac{1}{2}$, while the constant $\bar c$ will depend only upon $c^-$. In particular, for any sequence $\{k(j)\}_j$ which converges to infinity sufficiently fast, (a), (b), and (c) will hold. 

%Next we apply, as above, suitable rotations and assume that all the planes $\pi_k$ coincide with $\pi_0$ (without changing notation for the various objects introduced). 
We consider the rescaled maps
\[
\tilde f_{j,k} (x) := \rho_j^{-1} f_k (\rho_j x)
\]
and let instead $f_{j,k}:B_1 (0, \pi_0)\to \mathcal{A}_Q (\pi_0^\perp)$ be the Lipschitz approximations which are given by \cite{DLS14Lp}*{Theorem 2.4} applied to $T_{0, r_{j,k}}$. Observe that, by the estimates in \cite{DLS14Lp}*{Theorem~2.4},
\[
\lim_{j\to \infty} E_{k,j}^{-1/2} \|\mathcal{G} (f_{j,k}, \tilde{f}_{j,k})\|_{L^2} = 0\, .
\]
On the other hand, for every fixed $k$, the limit of $E_{k,j}^{-1} \tilde{f}_{j,k}$ is clearly a scalar multiple $\lambda (k)$ of $\bar f_k$, and it is easy to see that this scalar multiple has a fixed range $[\lambda^-, \lambda^+]$ for positive constants $\lambda^\pm$ depending upon $c^\pm$ and upon the constant $\bar c$ in condition \eqref{e:non-triviality} for $r_k$. It follows therefore that (d) holds for any $k(j)$ which diverges sufficiently fast.
\end{proof}

\subsection{Proof of Theorem \ref{thm:HS}} Recalling \cite{DLS_MAMS}*{Theorem~3.19}, the frequency value $\alpha$ at $0$ of any non-trivial Dir-minimizer $f$ with $f(0)=Q\llbracket 0\rrbracket$ is a strictly positive number and by \cite{DLS_MAMS}*{Corollary~3.18}, we have that
\begin{align}
&\lim_{\rho\to 0} \rho^{2-2\bar\alpha-m} \int_{B_\rho} |Df|^2 = \infty \qquad \forall \bar \alpha > \alpha\, ,\\
&\lim_{\rho\to 0} \rho^{2-2\bar\alpha-m} \int_{B_\rho} |Df|^2 = 0\qquad \forall \bar\alpha < \alpha\, .
\end{align}
On the other hand, since
\[
    |D\bar f|^2 = |Dv|^2 + Q |D(\boldsymbol\eta\circ \bar f)|^2,
\]
where $v$ is the average free part of $\bar f$, for any coarse blow-up $\bar f$ we conclude that $I_{\bar f} (0) = \min \{I_v (0), I_{\boldsymbol{\eta}\circ \bar f} (0)\}$ if $\boldsymbol{\eta}\circ \bar f$ is not identically vanishing, otherwise $\bar f = v$ and so $I_{\bar f} (0) = I_v (0)$. Recall that $\boldsymbol{\eta}\circ \bar f$ is a classical harmonic function with $\boldsymbol{\eta}\circ \bar f(0)=0$ and hence $I_{\boldsymbol{\eta}\circ \bar f}(0)$ is a positive integer if $\boldsymbol{\eta}\circ \bar f \not\equiv 0$. Thus, in order to prove that $I_{v} (0)\geq 1$, it suffices to show that $I_{\bar f} (0) \geq 1$. Introduce now 
\[
\bar f_r := r^{\frac{m-2}{2}} \frac{ f (rx)}{\sqrt{{\rm Dir} (f, B_r)}}\, .
\]
and apply Lemma \ref{l:rescaling-closure} to conclude that, if there is a coarse blow-up $\bar f$ with $\alpha = I_{\bar f} (0)$, then there is a coarse blow-up which is $\alpha$-homogeneous.

We will now prove that, if $\bar f$ is an $\alpha$-homogeneous coarse blow-up, then necessarily $\alpha\geq 1$. This is in fact the same argument used in  \cite{WicJDG}*{Proposition 3.10} and we report it for the reader's convenience.
Consider thus such a coarse blow-up and fix a blow-up sequence $r_k$ leading to it, according to the procedure explained above. In order to simplify our notation we denote by $T_k$ the current $T_{0,r_k}$.

First of all, recall that since $\Theta(T_k,0) \geq Q$, the error from the monotonicity formula for mass ratios gives the estimate
\begin{equation}\label{eq:HS}
\int_{\Bbf_{4}} \frac{1}{|q|^m}\left|\frac{q^\perp}{|q|}\right|^2\dd \|T_k\|(q) \leq C E_k + C\Abf_k^2\, .
\end{equation}
See, for example,~\cite{Spolaor_15} for a derivation of this. The only subtlety compared to the classical literature (cf. for instance \cite{Simon_GMT}) is that the usual derivation of the above estimate is reduced to the one for varifolds with bounded mean curvature. The latter is not good enough for us because it would give a linear dependence on $\Abf_k$, rather than a quadratic one. The quadratic improvement, which is possible using the stronger information that our current induces a {\em stationary varifold} in a Riemannian submanifold, is remarked in \cite{DLS14Lp}*{Appendix A}. 

As described in the procedure leading to coarse blow-ups we rotate the currents suitably so that $\pi_k = \pi_0$. We next pass the inequality~\eqref{eq:HS} to the Lipschitz approximations $f_k$ given by \cite{DLS14Lp}*{Theorem~2.4}.
We let $\sum_i \llbracket (f_k)_i \rrbracket$ be a (measurable) selection for the $f_k$ as in~\cite{DLS_MAMS}*{Theorem~0.4}. We then write
\begin{equation}\label{eq:HSgraph}
\int_{K_k} \sum_i \frac{|\big(x + (f_k)_i(x)\big)^\perp|^2}{|x + (f_k)_i(x)|^{m+2}}\dd x \leq C (E_k + \Abf_k^2)\leq C E_k,
\end{equation}
where $K_k \subset B_1 \subset \pi_k$ is the (closed) domain over which the graph of the Lipschitz approximation $f_k$ coincides with the current $T_k$ (cf. \cite{DLS14Lp}*{Theorem~0.4}. Note that, for the point $q = x + (f_k)_i(x)
\in K_k \times \pi_k^\perp$, $q^\perp$ denotes the orthogonal projection of $q$ to $(T_q \Gbf_{f_k})^\perp$, where $\Gbf_{f_k}$ (the current induced by the graph of the multivalued function $f_k$) is defined as  in~\cite{DLS_multiple_valued}*{Definition~1.10}.

However, since $f_k$ is Lipschitz and thus differentiable almost everywhere by Rademacher's Theorem~\cite{DLS_MAMS}*{Theorem~1.3}, we can formally compute
\begin{equation}\label{eq:radialder}
        \frac{\partial}{\partial r}\left(\frac{(f_k)_i(x)}{|x|}\right) = \frac{\partial}{\partial r}\left(\frac{x+ (f_k)_i(x)}{|x|}\right) = \frac{\partial_r\left(x + (f_k)_i(x)\right)}{|x|} -\frac{x+(f_k)_i(x)}{|x|^2}.
\end{equation}
Since the first term on the left-hand side belongs to $T_q \mathbf{G}_f$ at $q= x + (f_k)_i (x)$, we have
\[
    \left|\left[\frac{\partial}{\partial r} \left(\frac{(f_k)_i(x)}{|x|}\right)\right]^\perp\right|^2 = 
    \frac{\left|\left[x+(f_k)_i(x)\right]^\perp\right|^2}{|x|^4}.
\]
Combining this with~\eqref{eq:HSgraph}, we have
\begin{equation}\label{eq:HSgraph2}
    \int_{K_k} \sum_i \frac{|x|^4}{|x + (f_k)_i(x)|^{m+2}}\left|\left[\frac{\partial}{\partial r} \left(\frac{(f_k)_i(x)}{|x|}\right)\right]^\perp\right|^2 \dd x \leq C E_k\, .
\end{equation}
We next wish to estimate the tangential component of the right-hand side of~\eqref{eq:radialder} as follows:
\[
        \left|\left[\frac{\partial}{\partial r} \left(\frac{(f_k)_i(x)}{|x|}\right)\right]^\parallel\right|^2 \leq \|\mathbf{p}_{\vec{T}_k(q)} - \mathbf{p}_{\pi_0}\|^2 \left|\frac{\partial}{\partial r} \left(\frac{(f_k)_i(x)}{|x|}\right)\right|^2 \leq C E_k^\beta \left|\frac{\partial}{\partial r} \left(\frac{(f_k)_i(x)}{|x|}\right)\right|^2,
\]
where we have used that, at the point $q= x+ (f_k)_i (x)$ of interest, the tangent to the current coincides with the tangent to $\mathbf{G}_f$, and the distance of the latter to $\pi_0$ can be estimated with the Lipschitz constant of $f_k$ (cf. \cite{DLS14Lp}*{Theorem~2.4}). Writing
\begin{align*}
\left|\frac{\partial}{\partial r}\left(\frac{(f_k)_i(x)}{|x|}\right)\right|^2 = 
\left|\left[\frac{\partial}{\partial r} \left(\frac{(f_k)_i(x)}{|x|}\right)\right]^\perp\right|^2 +
\left|\left[\frac{\partial}{\partial r} \left(\frac{(f_k)_i(x)}{|x|}\right)\right]^\parallel\right|^2
\end{align*}
we immediately conclude
\[
\left|\frac{\partial}{\partial r}\left(\frac{(f_k)_i(x)}{|x|}\right)\right|^2
\leq 2 \left|\left[\frac{\partial}{\partial r} \left(\frac{(f_k)_i(x)}{|x|}\right)\right]^\perp\right|^2 
\]
as soon as $E_k$ is sufficiently small. Hence, by \eqref{eq:HSgraph2}, we conclude
\begin{equation}\label{eq:homogeneitybd}
\int_{K_k} \sum_i\frac{|x|^4}{|x + (f_k)_i(x)|^{m+2}}\left|\frac{\partial}{\partial r} \left(\frac{(f_k)_i(x)}{|x|}\right)\right|^2 \dd x \leq C E_k\, .
\end{equation}
Next, consider $\bar f_k := E_k^{-1/2} f_k$ and infer, from \eqref{eq:homogeneitybd} the estimate
\[
\int_{\bigcap_{j\geq k_0} K_j\setminus B_\rho} \sum_i\frac{|x|^4}{|x + E_k^{1/2} (\bar f_k)_i(x)|^{m+2}}\left|\frac{\partial}{\partial r} \left(\frac{(\bar f_k)_i(x)}{|x|}\right)\right|^2 \dd x \leq C\, ,
\]
for any $k\geq k_0$ and $\rho>0$.
Recall that:
\begin{itemize}
    \item $\bar f_k$ converges strongly in $W^{1,2} (B_{1/2})$ to $\bar f$;
    \item The height bound of \cite{Spolaor_15} implies that $\|\bar f_k\|_\infty$ is uniformly bounded. 
\end{itemize}
We can thus pass into the limit in $k$ to conclude
\[
\int_{(B_{1/2}\setminus B_\rho) \cap \bigcap_{j\geq k_0} K_j} \sum_i\frac{1}{|x|^{m-2}}\left|\frac{\partial}{\partial r} \left(\frac{(\bar f)_i(x)}{|x|}\right)\right|^2 \dd x \leq C\, .
\]
By choosing a fast converging subsequence, we can assume that the series $\sum |B_1\setminus K_j|$ is summable. Therefore, let $k_0\uparrow \infty$ and $\rho\downarrow 0$ we get 
\begin{equation}\label{e:control_singularity}
\int_{B_{1/2}} \sum_i\frac{1}{|x|^{m-2}}\left|\frac{\partial}{\partial r} \left(\frac{(\bar f)_i(x)}{|x|}\right)\right|^2 \dd x \leq C
\end{equation}
Since $\bar f$ is $\alpha$-homogeneous we have
\[
\bar f_i(x) =|x|^\alpha \bar f_i\left(\frac{x}{|x|}\right),   
\]
and so
\[
\frac{\partial}{\partial r}\left(\frac{\bar f_i(x)}{|x|}\right) = (\alpha - 1)|x|^{\alpha - 2}\bar f_i\left(\frac{x}{|x|}\right).
\]
Inserting in \eqref{e:control_singularity} and passing to polar coordinates we conclude
\[
(\alpha -1)^2 \int_{\partial B_1} |\bar f|^2 \int_0^{1/2} s^{-1-2(1-\alpha)} \dd s \leq C\, .
\]
The latter inequality implies immediately $\alpha\geq 1$, and thus completes the proof.

\section{Comparison of coarse and fine blow-ups}\label{ss:coarse-fine}

In this section we compare fine and coarse blow-ups at scales which are comparable to the left endpoints of a sequence of intervals of flattening. The main conclusion is that the average-free parts of coarse blow-ups are scalar multiples of fine blow-ups. More precisely we have the following proposition.

\begin{proposition}\label{p:coarse=fine}
Let $T$ be as in Assumption \ref{asm:3}. Let $r_k\in (s_{j(k)}, t_{j(k)}[$ be a blow-up sequence at the origin and assume that 
\begin{equation}\label{e:fine-coarse-assumption}
\liminf_{k\to \infty} \frac{s_{j(k)}}{r_k} > 0\, .
\end{equation}
Then
\eqref{e:A-E-infinitesimal} holds and we can consider a coarse blow-up $\bar f$ generated by a (subsequence) according to Section \ref{ss:coarse} and
 a fine blow-up $u$ (generated by a further subsequence) according to the procedure detailed in Section \ref{ss:compactness}.
 If we denote by $v$ the average-free part of $\bar f$, then there is a real number $\lambda > 0$ such that $v= \lambda u$. 
\end{proposition}

\begin{remark}\label{r:coarse=fine}
In general, without assumption \eqref{e:fine-coarse-assumption} it might be that \eqref{e:A-E-infinitesimal} does not hold and that we cannot, therefore, define a coarse blow-up. Even if we were to assume \eqref{e:A-E-infinitesimal}, but not \eqref{e:fine-coarse-assumption}, we could at best infer that $v=\lambda u$ for some $\lambda \geq 0$, but not that $\lambda$ is necessarily positive. Easy examples for the latter behavior can be constructed using holomorphic curves of $\mathbb C^2$ of the form $\{(z,w) : (w-h(z))^Q = z^p\}$, for a nontrivial holomorphic $h$ with $h(0)=h'(0)=0$ and a fraction $\frac{p}{Q}$ which is noninteger and larger than the order of vanishing of $h$ at $0$. 
\end{remark}

An obvious corollary of the latter proposition is that, under the above assumptions, $v$ is necessarily nontrivial and that $I_v (0) = I_u (0)$.

\subsection{Nontriviality and homogeneity of coarse blow-ups}
If we combine it with Theorem \ref{t:consequences}(i),(vi), Proposition \ref{p:coarse=fine} has the following further consequence, which will be useful in \cite{DMS}.

\begin{corollary}\label{c:coarse=fine}
Let $T$ be as in Assumption \ref{asm:1}, let $\delta_2>0$ be the parameter in \cite{DLS16centermfld}*{Assumption~1.8}  and assume the singularity degree $\Irm (T,0)$ is strictly smaller than $2-\delta_2$. Then any coarse blow-up $\bar{f}$ at $0$ is nontrivial, $\Irm (T, 0)$-homogeneous, and has average $0$ (so in particular $\bar{f}=v$ for the average-free part $v$). 

Moreover, for every $\gamma> 2 (\Irm (T, 0)-1)$, we have
\begin{equation}\label{e:lower-bound-excess}
\liminf_{r\downarrow 0} \frac{\mathbf{E} (T, \Bbf_r)}{r^\gamma} > 0 \, 
\end{equation}
and there exists a radius $r_0$ (which depends on the current $T$) such that 
\begin{equation}\label{e:lower-bound-excess-2}
\mathbf{E} (T, \Bbf_r) \geq \frac{r^\gamma}{s^\gamma} \mathbf{E} (T, \Bbf_s) \qquad
\forall r<s<r_0\, .
\end{equation}
\end{corollary}
\begin{proof}
It follows directly from Proposition \ref{p:coarse=fine} and from Theorem \ref{t:consequences}(i),(vi) that the average-free part of any coarse blow-up at $0$ is nontrivial and is $\Irm (T, 0)$-homogeneous. We therefore just need to show that the average vanishes. 

First of all observe that, if $\{\bar f_k\}$ is any family of coarse blow-ups, then $\|\bar f_k\|_{W^{1,2} (B_1)}$ is uniformly bounded and any limit $\bar f_\infty$ of any subsequence is also a coarse blow-up. Since every such $\bar f_\infty$ must have an average-free part which is nontrivial and $\Irm (T, 0)$-homogeneous, it follows immediately that for any coarse blow-up $\bar f$ there is a positive number $\omega>0$ (independent of $\bar f$) such that 
\[
\int_{B_1} |D\bar{u}|^2 \geq \omega > 0
\]
whenever $\bar{u}$ is the average-free part of $\bar f$. In particular, since the coarse blow-up $\bar f$ is itself nontrivial, we also conclude the existence of some constant $\Omega>0$ (again not depending on $\bar f$) such that
\begin{equation}\label{e:reverse-control}
\int_{B_1} |D (\boldsymbol{\eta} \circ \bar f)|^2 \leq \Omega \int_{B_1} |D \bar{u}|^2
\end{equation}
for every coarse blow-up $\bar f$, its average free part $\bar{u}$, and its average $\boldsymbol{\eta}\circ {\bar f}$.

Consider now the sequence $r_k\downarrow 0$ which generates any coarse blow-up $\bar f$ and recall that we are assuming $\pi_0$ to be an optimal plane so that 
\[
    \mathbf{E} (T, \Bbf_{8M r_k}, \pi_0) = \mathbf{E} (T, \Bbf_{8M r_k}) \eqqcolon E_k \to 0,
\]    
as explained in Section \ref{ss:coarse}. The Taylor expansion of the area functional {and \eqref{e:A-E-infinitesimal}} combined with the fact that any coarse blow-up along the sequence $8M r_k$ is $\Irm(T,0)$-homogeneous and satisfies the nontriviality property \eqref{e:reverse-control} implies that for $k$ sufficiently large we have
\[
\mathbf{E} (T, \mathbf{B}_{8M r_k}) \leq C \mathbf{E} (T, \mathbf{B}_{r_k})\, .
\]
Indeed, this is a consequence of a uniform lower bound on the corresponding ratio of normalized Dirichlet energies over $B_1$ and $B_{1/8M}$ of any such coarse blow-up. From the above, if $\tilde\pi_k$ is an optimal plane such that $\mathbf{E} (T, \Bbf_{r_k}, \tilde\pi_k) = \mathbf{E} (T, \Bbf_{r_k})$, then 
$|\pi_0 - \tilde\pi_k|\leq C \mathbf{E} (T, \mathbf{B}_{r_k})$ and thus 
\[
\mathbf{E} (T, \mathbf{B}_{8M r_k}, \tilde\pi_k) \leq C \mathbf{E} (T, \mathbf{B}_{r_k})\,.
\]
However, observe as well that for any constant $C$ fixed, the sequence $Cr_k$ also generates (up to possibly extract a subsequence) a coarse blow-up: in fact the excess must go to $0$ (because the currents $T_{0, Cr_k}$ converges to the same tangent cone as $T_{0, r_k}$, which thus must be flat) and $\mathbf{E} (T, \Bbf_{8 CM r_k})\geq C^{-m} \mathbf{E} (T, \Bbf_{8Mr_k})$, so that \eqref{e:A-E-infinitesimal} holds for the sequence $Cr_k$ as well.

{For any $j \in \N_{\geq 1}$, letting $\pi_{k,j}$ be a plane with 
\[
    E_{k,j} \coloneqq \mathbf{E} (T, \mathbf{B}_{2^{j+3}M r_k}) = \mathbf{E} (T, \mathbf{B}_{2^{j+3} M r_k}, \pi_{k,j}),
\]
we have
\[
    |\pi_0-\pi_{k,j}| = o((E_{k,j})^{1/2})
\]
and
\[
C(j)^{-1} \leq \liminf_k \frac{E_k}{E_{k,j}}
\leq \limsup_k \frac{E_k}{E_{k,j}} \leq C(j)\, ,
\]
so for $k$ sufficiently large we can apply \cite{DLS14Lp}*{Theorem 2.4} to $T_{0,r_k}$ in $\mathbf{B}_{2^{j+3}M}$ relative to the plane $\pi_0$ to get a Lipschitz approximation $g_{k,j}:B_{2^{j+2} M} (0,\pi_0) \to\Acal_Q((\pi_0)^\perp)$ in the cylinder $\mathbf{C}_{2^{j+2} M} (0,\pi_0)$, 
as in the algorithm detailed in Section \ref{ss:coarse}. This new Lipschitz approximation $g_{k,j}$ coincides with $f_k$ on $B_1(0,\pi_0)$, except for a set whose $\Hcal^m$-measure is estimated by $o (E_{k,j})$. In particular, for each $j$, as $k\to \infty$ the rescaled functions $\bar{g}_{k,j} = (E_{k,j})^{-\frac{1}{2}} g_{k,j}$ converge to a Dir-minimizing function $\bar{g}_j$ over $B_{2^{j+2} M}(0,\pi_0)$ which coincides with $\bar f$ on $B_1(0,\pi_0)$.

Next, we observe that 
\begin{align*}
 D(\boldsymbol{\eta}\circ \bar f) (0) &= \frac{1}{\omega_m} \int_{B_1(0,\pi_0)} D (\boldsymbol{\eta}\circ \bar f) 
= \frac{1}{\omega_m 2^{jm}} \int_{B_{2^j}(0,\pi_0)} D (\boldsymbol{\eta} \circ \bar{g}_j) = D(\boldsymbol{\eta}\circ \bar g_j) (0) \, ,
\end{align*}}
by the harmonicity of the two functions $\boldsymbol{\eta}\circ \bar f$ and $\boldsymbol{\eta}\circ \bar g_j$.
But we then must have $D (\boldsymbol{\eta} \circ \bar{f}) (0) = D(\boldsymbol{\eta}\circ \bar g_j) (0) = 0$, otherwise we can use the Taylor expansion of \cite{DLS_multiple_valued} to contradict the optimality of the plane $\pi_0$.  

In summary, by rescaling the domain of the functions $\bar g_{k,j}$ above to be $B_{8M}(0,\pi_k)$ (without relabeling), if $r_k\downarrow 0$ is a sequence which generates a coarse blow-up $\bar f$, then as $k\to \infty$, a subsequence of the sequence of scales $2^j r_k$ generates a coarse blow-up $\bar g_j$ with the property that $\bar f (x) = \lambda_j \bar g_j (2^{-j} x)$ for some positive nonzero number $\lambda_j$. 

Next, denote by $\bar u$ the average-free part of $\bar f$ and by $v_j$ the average-free part of $\bar{g}_j$. Observe that $D \bar u$ and $D \bar v_j$ are $(\Irm (T, 0) -1)$-homogeneous, while $D (\boldsymbol{\eta}\circ \bar f)$ and $D (\boldsymbol{\eta}\circ \bar g_j)$ are classical harmonic functions with $D (\boldsymbol{\eta} \circ \bar f) (0) = D (\boldsymbol{\eta}\circ \bar g_j) (0) = 0$ and $\boldsymbol{\eta}\circ \bar g_j (0) = \boldsymbol{\eta}\circ \bar f(0)=0$, in particular $I_{\boldsymbol{\eta}\circ \bar g_j}(0) \geq 2$. Therefore, we observe that 
\begin{align*}
\frac{\int_{B_1} |D(\boldsymbol{\eta}\circ \bar f)|^2}{\int_{B_1} |D \bar u|^2 } &= \frac{\int_{B_{2^{-j}}} |D(\boldsymbol{\eta}\circ \bar g_j)|^2}{\int_{B_{2^{-j}}} |D \bar v_j|^2 } \\
&\leq \frac{2^{-j(2I_{\boldsymbol{\eta}\circ \bar g_j}(0)-2)}\int_{B_{1}} |D(\boldsymbol{\eta}\circ \bar g_j)|^2}{2^{-j(2\Irm(T,0)-2)}\int_{B_{1}} |D \bar v_j|^2 } \\
&\leq \frac{2^{2j(\Irm(T,0)-2)}\int_{B_{1}} |D(\boldsymbol{\eta}\circ \bar g_j)|^2}{\int_{B_{1}} |D \bar v_j|^2 }
\end{align*}
On the other hand the bound \eqref{e:reverse-control} is valid also for $\bar g_j$ and $\bar v_j$ in place of $\bar f$ and $\bar u$, because $\bar{g}_j$ is a coarse blow-up and $\bar v_j$ is its average-free part. In particular, recalling that $\Irm (T, 0) < 2 - \delta_2$ we conclude
\[
\frac{\int_{B_1} |D(\boldsymbol{\eta}\circ \bar f)|^2}{\int_{B_1} |D \bar u|^2 } \leq 2^{-2 \delta_2 j} \Omega\, .
\]
Since $\Omega$ is {a fixed positive constant}, $j$ an arbitrary integer, and $\delta_2$ a positive number, we immediately conclude that $D (\boldsymbol{\eta} \circ \bar f) \equiv 0$ and $\boldsymbol{\eta}\circ \bar f$ is a constant. On the other hand recall that, since $\Theta (T, 0) = Q$, $\bar f (0) = Q \llbracket 0\rrbracket$, and in particular $\boldsymbol{\eta}\circ \bar f (0) = 0$. We thus have proved that $\boldsymbol{\eta}\circ \bar f \equiv 0$.

\medskip

Next observe that the arguments detailed so far have also the following outcome. If $r_k\downarrow 0$ is a sequence such that $\Ebf (T, \Bbf_{r_k})\to 0$, then
\[
\lim_{r\downarrow 0} \frac{\Ebf (T, \Bbf_{r/2})}{\Ebf (T, \Bbf_r)} = 2^{-(\Irm (T, 0)-1)}\, .
\]
Fix now any $\gamma< {\Irm (T, 0)-1}$. The above implies the following: there is $\bar{r}>0$ and $\bar{E}>0$ such that:
\begin{itemize}
\item If $r< \bar{r}$ and $\Ebf (T, \Bbf_r)< \bar{E}$, then 
\[
\frac{\Ebf (T, \Bbf_{r/2})}{\Ebf (T, \Bbf_r)} \geq 2^{-\gamma}\, . 
\]
\end{itemize}
We next distinguish two cases. We consider the following set 
\[
\mathcal{R} := \{0<r<\bar r : \Ebf (T, r) < 2^{-1} \bar{E}\}\, ,
\]
which can be easily checked to be open if $\bar r$ is sufficiently small.
We then argue differently depending on whether $\mathcal{R}$ contains a neighborhood of the origin or not (and notice that, when $\Irm (T, 0)>1$, we are certainly in the first case). If it contains a neighborhood of the origin, then there is $\tilde{r}>0$ such that 
\[
\frac{\Ebf (T, \Bbf_{r/2})}{\Ebf (T, \Bbf_r)} \geq 2^{-\gamma} \qquad \forall r< \tilde{r}\, .
\]
In particular, if we let $\tilde{c} := \inf \{\Ebf (T, \Bbf_r) : \frac{\tilde{r}}{2} \leq r< \tilde{r}\} > 0$, iterating the inequality above at all dyadic scales we achieve
\[
\Ebf (T, \Bbf_r) \geq \tilde{c} \left(\frac{r}{2\tilde{r}}\right)^\gamma \, .
\]
If it does not contain the origin then let $\mathcal{R} = \bigcup_k ]r_k^-, r_k^+[$ where $r_{k+1}^+ < r_k^-$ and both are infinite sequences of infinitesimal numbers. Then, $\Ebf (T, \Bbf_{r_k^+})= \frac{\bar{E}}{2}$ and, up to subsequences, $T_{0, r_k^+}$ converges to a cone $C$ which is nonplanar and such that $\Ebf (C, \Bbf_\rho) = \frac{\bar{E}}{2}$ for every $\rho$. It follows in particular that there exists $k_0$ such that 
\[
\frac{\bar{E}}{4} \leq \Ebf (T, \Bbf_r) < \bar{E} \qquad \forall r \in \bigcup_{k\geq k_0} \left]{\textstyle{\frac{r_k^+}{2}}}, r_k^+\right[\, .
\]   
In particular, arguing as above we conclude
\[
\Ebf (T, \Bbf_r) \geq \frac{\bar{E}}{4} \left(\frac{r}{2 r_k^+}\right)^\gamma \qquad \forall r\in \bigcup_{k\geq k_0} ]r_k^-, r_k^+[\, ,
\]
while 
\[
\Ebf (T, \Bbf_r) \geq \bar{E} \qquad \forall r < r_{k_0}^+ \quad \mbox{s.t. } r\not\in \bigcup_{k\geq k_0} ]r_k^-, r_k^+[\, .
\]
The combination of these two facts give that
\[
\liminf_{r\downarrow 0} \frac{\Ebf (T, \Bbf_r)}{r^\gamma} > 0
\]
and thus concludes the proof of \eqref{e:lower-bound-excess}.
\end{proof}

\subsection{Reparametrization} An important tool for proving the Proposition \ref{p:coarse=fine} is the following lemma, where we follow the notation and techniques introduced in \cite{DLS_multiple_valued}.

\begin{lemma}\label{l:reparametrization-lemma}
There are constants $\kappa (m,n,Q)>0$ and $C (m,n,Q)$ with the following property. Consider:
\begin{itemize}
    \item A Lipschitz map $g: \mathbb R^m \supset B_2 \to \mathcal{A}_Q (\mathbb R^n)$ with $\|g\|_{C^0} + {\rm Lip}\, (g) \leq \kappa$;
    \item A $C^2$ function $\boldsymbol{\varphi} : B_2 \to \mathbb R^n$ with $\boldsymbol{\varphi} (0) = 0$ and $\|D\boldsymbol{\varphi}\|_{C^1} \leq \kappa$;
    \item The function $f (x) = \sum_i \llbracket \boldsymbol{\varphi} (x) + g_i (x)\rrbracket$ and the manifold $\mathcal{M} := \{(x, \boldsymbol{\varphi }(x))\}$;
    \item The maps $N, F: \mathcal{M}\cap \Cbf_{3/2} \to \mathcal{A}_Q (\mathbb R^{m+n})$ given by \cite{DLS_multiple_valued}*{Theorem 5.1}, satisfying $F (p) = \sum_i \llbracket p+ N_i (p)\rrbracket$, $N_i(p)\perp T_p \mathcal{M}$, and $\boldsymbol{T}_F \res \Cbf_{5/4} = \Gbf_f \res \Cbf_{5/4}$.
\end{itemize}
If we denote by $\tilde{g}$ the multivalued map $x\mapsto \tilde{g} (x) = \sum_i \llbracket (0, g_i (x))\rrbracket \in \mathcal{A}_Q (\mathbb R^{m+n})$, then
\begin{equation}\label{e:comparison-estimate}
\mathcal{G} (N (\boldsymbol{\varphi} (x)), \tilde{g} (x)) \leq C \|D \boldsymbol{\varphi}\|_{C^0} (\|g\|_{C^0} + \|D \boldsymbol{\varphi}\|_{C^0}) \qquad \forall x\in B_1\, .
\end{equation}
\end{lemma}

\begin{proof} We fix a point $x\in B_1$, denote by $p\in \mathcal{M}$ the point $p=(x, \boldsymbol{\varphi} (x))$ and let $N (x) = \sum_i \llbracket q_i\rrbracket$ and $g (x) = \sum_i \llbracket p_i\rrbracket$. We fix a measurable selection for the function $g$, so that we can write $g = \sum_i \llbracket g_i\rrbracket$ and a corresponding measurable selection for $f$, where $f_i = \boldsymbol{\varphi} + g_i$. According to \cite{DLS_multiple_valued}*{Lemma 5.4}, the set of points $\{q_i\}$ can be determined as follows. If we let $\varkappa$ be the orthogonal complement of $T_p \mathcal{M}$, then $\{q_i\}$ is given by the intersection of $p+\varkappa$ with the support of the current $\mathbf{G}_f$ (i.e. the set-theoretic graph of $f$). This means that there are points $x_1, \ldots, x_Q$ such that 
\[
q_i = (x_i, f_{j(i)} (x_i)) = (x_i, \boldsymbol{\varphi} (x_i) + g_{j(i)} (x_i))\, ,
\]
where $j: \{1, \ldots , Q\} \to \{1, \ldots, Q\}$ is some unknown function.
Observe that 
\[
|x_i-x|\leq C |q_i-p| |\varkappa - \varkappa_0|\, 
\]
where $\varkappa_0$ denotes the vertical plane $\{0\} \times \mathbb R^n$. We therefore easily conclude the estimate
\[
|x_i-x|\leq C \|N\|_{C^0} \|D \boldsymbol{\varphi}\|_0\, .
\]
Since however $\|N\|_{C^0} \leq C (\|g\|_{C^0} + \|\boldsymbol{\varphi}\|_{C^0})\leq C (\|g\|_{C^0} + \|D \boldsymbol{\varphi}\|_{C^0})$, clearly
\begin{equation}\label{e:horizontal-component}
|x_i-x|\leq C \|D \boldsymbol{\varphi}\|_{C^0} (\|g\|_{C^0} + \|\boldsymbol{\varphi}\|_{C^0})\, .
\end{equation}
Given the Lipschitz bound on $g$ we conclude that there is a $\pi (i)$ such that 
\begin{equation}\label{e:vertical-component}
|g_{j(i)} (x_i) - g_{\pi (i)} (x)|\leq C \|\boldsymbol{\varphi}\|_{C^0} (\|g\|_{C^0} + \|D \boldsymbol{\varphi}\|_{C^0})\, .
\end{equation}
If $\pi: \{1, \ldots, Q\} \to \{1, \ldots , Q\}$ were injective, we would immediately conclude \eqref{e:comparison-estimate}. While this might generally not be the case, it certainly is when $Q=1$, hence establishing the estimate in this particular case. 

For the general case we argue by induction. Assume therefore to have fixed $Q$ and to have proved the estimate valid for maps which are $Q'$-valued for every $Q'<Q$. Consider now the following alternatives:
\begin{itemize}
    \item[(a)] the diameter of the set $\{g_i (x)\}$ is smaller than $\|D \boldsymbol{\varphi}\|_{C^0} (\|g\|_{C^0} + \|D \boldsymbol{\varphi}\|_{C^0})$;
    \item[(b)] the diameter of the set $\{g_i (x)\}$ is larger.
\end{itemize}
In the first case we have 
\[
|g_{j(i)} (x_i) - g_i (x)|\leq |g_{j(i)} (x_i) - g_{\pi (i)} (x)| + |g_{\pi (i)} (x) - g_i (x)|
\leq (C+1) \|D \boldsymbol{\varphi}\|_{C^0} (\|g\|_{C^0} + \|D \boldsymbol{\varphi}\|_{C^0})\, .
\]
In the second case we set $d:= \|D \boldsymbol{\varphi}\|_{C^0} (\|g\|_{C^0} + \|\boldsymbol{\varphi}\|_{C^0})$ and recall \cite{DLS_MAMS}*{Proposition 1.6}: if the Lipschitz constant of $g$ is smaller than a constant depending only on $C$, $Q$, and $n$, the map $g$ decomposes, in the ball $B_{2 d} (x)$ into two Lipschitz $Q_i$-valued maps with $Q_1+Q_2 = Q$. In particular we can use the inductive assumption to get \eqref{e:comparison-estimate}.
\end{proof}

\subsection{Comparison estimates} In order to prove Proposition~\ref{p:coarse=fine}, \eqref{e:comparison-estimate} will be combined with two important estimates comparing the Lipschitz approximation and the normal approximation over the relevant center manifold.

The first estimate is the following control on the $L^2$ height of a normal approximation in terms of the excess.
\begin{lemma}\label{lem:height-excess}
Under the assumptions of Proposition \ref{p:coarse=fine}, the estimate  \eqref{e:A-E-infinitesimal} holds. Moreover, the following holds.
\begin{itemize}
\item[(i)] Let $\mathbf{h}_k$ be as in Section~\ref{ss:compactness} for the scales $r_k$. Then we have
\begin{equation}\label{e:two-sided-control}
0<	\liminf_{k\to\infty} \frac{\mathbf{h}_k^2}{E_k} \leq 
\limsup_{k\to\infty} \frac{\mathbf{h}_k^2}{E_k} < \infty\, . 
\end{equation}
\item[(ii)] Let $f_k$ be as in Section~\ref{ss:coarse} and 
consider the map $\bar{\boldsymbol{\varphi}}_k$ on $B_2 = B_2 (0,\pi_0)$ whose graph coincides with the center manifold $(\mathcal{M}_{j(k)})_{0, r_k/t_{j(k)}}$ over the cylinder $\mathbf{C}_{3/2} = \mathbf{C}_{3/2} (0,\pi_0)$. Then we have
\begin{equation}\label{e:stima-L2}
\int_{B_{3/2}} |\bar{\boldsymbol{\vphi}}_k - \boldsymbol{\eta}\circ f_k|^2 = o(E_k)\, .
\end{equation}
\end{itemize}
\end{lemma}

\begin{proof} We fix $r_k$ as in the statement and, upon extraction of a further subsequence, we assume the existence of 
\[
\lim_{k\to \infty} \frac{r_k}{s_{j(k)}} := \tilde{c} \in [1, +\infty[\, .
\]
It is convenient to introduce the rescaled radii $\bar{r}_k := \frac{r_k}{t_{j(k)}} \in ]0,1]$ and $\bar s_{j(k)} := \frac{s_{j(k)}}{t_{j(k)}}$. Recalling the stopping condition which defines $s_{j(k)}$ in \cite{DLS16blowup}*{Section 2.1}, there is a cube $L_k\in \mathscr{W}^{j(k)}$ with $L_k\cap \mathbf{B}_{\bar s_{j(k)}}\neq\emptyset$ and $\ell (L_k) = c_s \bar s_{j(k)}$ for the specific geometric constant $c_s=\frac{1}{64\sqrt{m}}$. Observe that, since $\Theta(T,0)=Q$, \cite{DLS16centermfld}*{Proposition 3.1} implies that $L_k$ cannot belong to $\mathscr{W}^{j(k)}_h$. If $L_k\in \mathscr{W}^{j(k)}_n$, we may apply \cite{DLS16centermfld}*{Corollary 3.2} to find a nearby cube $L'_k \in \mathscr{W}^{j(k)}_e$ of comparable size. Thus, we may assume that $L_k\in \mathscr{W}^{j(k)}_e$. We can thus apply \cite{DLS16centermfld}*{Proposition 3.4} to conclude
\[
\boldsymbol{m}_{0,j(k)} \ell (L_k)^{2-2\delta_2} \leq C \mathbf{E} (T_{0, t_{j(k)}} , \mathbf{B}_{L_k})\, ,
\]
for some geometric constant $C$, where $\boldsymbol{m}_{0,j(k)}$ is as in \eqref{eq:m_0} with index $j(k)$. Recalling however that the cylinder $\mathbf{C}_{4M \bar r_k }$ as in Section \ref{ss:coarse} contains $\mathbf{B}_{L_k}$, as well as our amended definition of $\boldsymbol{m}_{0,j(k)}$, we immediately conclude that 
\begin{align*}
E_k & := \mathbf{E} (T_{0, r_k}, \mathbf{C}_{4M}, \pi_0) \geq \mathbf{E} (T_{0, t_{j(k)}}, \mathbf{B}_{L_k}) \geq C^{-1} \boldsymbol{m}_{0,j(k)} \ell (L_k)^{2-2\delta_2}\\
&\geq C^{-1} c_s^2 \bar \varepsilon^2 t_{j(k)}^{2-2\delta_2} \frac{s_{j(k)}^{2-2\delta_2}}{t_{j(k)}^{2-2\delta_2}}= C^{-1} c_s^2  \bar \varepsilon^2 s_{j(k)}^{2-2\delta_2}\, .
\end{align*}
In light of the comparability of $s_{j(k)}$ and $r_k$, it thus follows immediately that 
\begin{equation}\label{e:liminf_on_Ek}
\liminf_{k\to \infty} \frac{E_k}{r_k^{2-2\delta_2}} > 0\, ,
\end{equation}
which in turn immediately implies \eqref{e:A-E-infinitesimal}. In addition, rescaling by $t_{j(k)}$ and again using the definition of $\boldsymbol{m}_{0,j(k)}$, we have
\begin{equation}\label{e:compare_Ek-m0}
E_k \geq \tilde{C}^{-1} \boldsymbol{m}_{0,j(k)} \bar r_k^{2-2\delta_2}\, ,
\end{equation}
where $\tilde{C}$ is independent of $k$ (it is not, however, a geometric constant, namely it might depend on the blow-up sequence that we fixed at the beginning).

Next, observe that 
\[
\mathbf{h}_k^2 \leq \frac{C}{\bar{r}_k^{m+1}} \mathbf{H}_{N_{j(k)}} (2 \bar r_k) \stackrel{\eqref{e:frequency-twosided-bounds}}{\leq} \tilde{C} \bar r_k^{-m} \mathbf{D}_{N_{j(k)}} (2 \bar r_k)\, ,
\]
where $\tilde{C}$ is independent of $k$. {Note that the first inequality is a simple consequence of the scaling of $\bar N_k$ and the fact that $\mathbf{h}_k \leq C \Hbf_{\bar N_{j(k)}}(2)$.} On the other hand we recall (see for instance \cite{DLS16blowup}*{Remark 3.4}) that $\mathbf{D}_{N_{j(k)}} (2 \bar r_k) \leq C \boldsymbol{m}_{0,j(k)} \bar r_k^{m+2-2\delta_2}$. We thus conclude that 
\[
\mathbf{h}_k^2 \leq C \boldsymbol{m}_{0,j(k)} \bar r_k^{2-2\delta_2} 
\]
and we achieve the right-hand inequality of \eqref{e:two-sided-control} when combining the above with \eqref{e:compare_Ek-m0}. 

As for the left-hand inequality of \eqref{e:two-sided-control}, first recall that, by \cite{DLS16centermfld}*{Proposition 3.4} we also have the opposite inequality 
\begin{equation}\label{e:hk-from-below}
\mathbf{h}_k^2 \geq \tilde{C}^{-1} \bar r_k^{-m-2} \int_{\mathcal{L}_k} |N_{j(k)}|^2 \geq \tilde{C}^{-1} \boldsymbol{m}_0^{(k)} \bar r_k^{2-2\delta_2}\, ,
\end{equation}
where $\Lcal_k$ is the Whitney region corresponding to $L_k$. On the other hand {recall that we are assuming $\pi_0$} optimizes the excess of $T_{0, r_k}$ in $\mathbf{B}_{8M}$, which implies that it optimizes the excess of $T_{0, t_{j(k)}}$ in $\mathbf{B}_{8M \bar r_k}$. Because of the condition $s_{j(k)} \leq r_k \leq \bar{c} s_{j(k)}$, we can find a cube $H\in \mathscr{S}^{j(k)}\cup \mathscr{W}^{j(k)}$ with the property that $\mathbf{B}_{32 M \bar r_k} \supset \mathbf{B}_H \supset \mathbf{B}_{8M \bar r_k}$. Due to \cite{DLS16centermfld}*{Proposition 1.11}, we thus must have
\begin{align*}
\mathbf{E} (T_{0, t_{j(k)}}, \mathbf{B}_{8M \bar r_k}, \pi_0) &\leq \mathbf{E} (T_{0, t_{j(k)}}, \mathbf{B}_{8M \bar r_k}, \pi_H) \leq C \mathbf{E} (T_{0, t_{j(k)}}, \mathbf{B}_H, \pi_H) \\
&\leq C \boldsymbol{m}_{0,j(k)} \ell (H)^{2-2\delta_2}
\leq C \boldsymbol{m}_{0,j(k)} \bar r_k^{2-2\delta_2}\, .
\end{align*}
Combining this with the height bound~\cite{DLS16centermfld}*{Theorem A.1} on $T_{0, t_{j(k)}}$, we can write
\begin{equation}\label{e:Ek-from-above}
E_k = \mathbf{E} (T_{0, t_{j(k)}}, \mathbf{C}_{4 M \bar r_k}) \leq C  \mathbf{E} (T_{0, t_{j(k)}} \mathbf{B}_{8M \bar r_k}, \pi_0) 
\leq C \boldsymbol{m}_{0,j(k)} \bar r_k^{2-2\delta_2}\, .
\end{equation}
It thus follows immediately from \eqref{e:hk-from-below} and \eqref{e:Ek-from-above} that $\liminf_k \frac{\mathbf{h}_k^2}{E_k} > 0$.

\medskip
 
We now address the last part of the lemma, namely statement (ii).
First of all we apply a homothetic rescaling  of center $0$ and ratio $\bar r_k$ to the graphs of $\bar{\boldsymbol{\varphi}}_k$ and of $f_k$. We denote by $\bar{\boldsymbol{\varphi}}^r_k \coloneqq \bar{r}_k^{-1} \bar{\boldsymbol{\varphi}}_k(\bar{r}_k \cdot)$ and $f^r_k \coloneqq \bar{r}_k^{-1} f_k(\bar{r}_k \cdot)$ the corresponding maps and note that the desired estimate is equivalent to 
\[
\bar r_k^{-m-2} \int_{B_{3 \bar r_k/2} (0,\pi_0)} |\bar{\boldsymbol{\varphi}}^r_k - \boldsymbol{\eta} \circ f^r_k|^2 = o (E_k)\, ,
\]
and given the estimate \eqref{e:compare_Ek-m0}, it suffices to show
\begin{equation}\label{e:rescaled-comparison}
\int_{B_{3 \bar r_k/2} (0,\pi_0)} |\bar{\boldsymbol{\varphi}}^r_k - \boldsymbol{\eta} \circ f^r_k|^2 = E_k^{1/2} o (\boldsymbol{m}_{0,j(k)}^{1/2} \bar r_k^{m+3-\delta_2}) \, ,
%+o (E_k \bar r_k^m)\, .
\end{equation}
where we are keeping a factor of $E_k^{1/2}$ on the right-hand side for the purpose of convenience, since it will appear naturally in the estimates we will proceed to obtain. Consider now the plane $\pi_0 (j(k))$ which served as reference to construct the center manifold $\mathcal{M}_{j(k)}$. It is easy to see that $|\pi_0 (j(k))- \pi_0| \leq C \boldsymbol{m}_{0,j(k)}^{1/2} \leq C \bar\varepsilon$ for some geometric constant (see \cite{DLS16centermfld}*{Proposition 4.1}). Since nothing else will be used about $\pi_0 (j(k))$, except that it serves as reference to construct the center manifold $\mathcal{M}_{j(k)}$, in order to simplify our notation we will simply denote it by $\tilde{\pi}_0$, even though the plane does depend on $k$.

We now consider all the cubes $H\in \mathscr{W}^{j(k)}$ which intersect $\mathbf{B}_{2\bar r_k}$ and denote such collections by $\mathscr{C}^{(k)}$. For each $H\in \mathscr{C}^{(k)}$ we consider a cylinder $\mathbf{C}_{2 C \ell(H)} (q_H, \pi_H)$, where $C$ is a geometric constant (which will be specified later) and $q_H$ is the center of the cube $H$. We then consider the cylinder 
$\mathbf{C}_{C \ell (H)} (q_H, \pi_0)$ and, given that the height of $T_{0, r_k}$ over $\pi_0$ converges to $0$, conclude that the set $({\rm gr}\, (\bar{\boldsymbol{\varphi}}^r_k) \cup {\rm gr}\,  (f_k^r)) \cap \mathbf{C}_{C \ell (H)} (q_H, \pi_0)$ is contained in $\mathbf{C}_{2 C\ell (H)} (q_H, \pi_H)$. Further, let $\boldsymbol{\Phi}_{j(k)} (\boldsymbol{\Gamma}_{j(k)})$ be the contact set of the current $T_{0, t_{j(k)}}$ and the center manifold $\mathcal{M}_{j(k)}$, as defined in \cite{DLS16centermfld}*{Definition 1.18}, and denote by $\Gamma_k$ its projection onto the plane $\pi_0$. Finally, it will also be convenient to define the point $\bar q_L$ as the orthogonal projection onto $\pi_0$ of $q_L$. 

If $C$ is a geometric constant sufficiently large (e.g. $10\sqrt{m}$ suffices, provided $\bar \varepsilon$ is small enough), then the set $\Gamma_k$ and the disks $B_{C\ell (H)} (\bar q_H, \pi_0)$ cover the disk $B_{3\bar r_k/2} (0, \pi_0)$. It will be convenient to devise a slightly delicate cover, made of pairwise disjoint Borel sets, with the following algorithm. We enumerate the disks $B_{C\ell (H)} (\bar q_H, \pi_0)$ as $B^i$, $i\in \{1, 2, \ldots \}= \mathbb{N} \setminus \{0\}$ and set $F_0 := \Gamma_k \cap B_{3\bar r_k/2}$ and define inductively $F_{j+1} := B^{j+1} \setminus \bigcup_{i\leq j} F_i$. 

Next, for each $H$ we recall the approximating Lipschitz map $f_H$ of \cite{DLS16centermfld}*{Definition 1.13 \& Lemma 1.15} and let $\bar f_H$ be the reparametrization of ${\rm gr}\, (f_H) \cap \mathbf{C}_{C\ell (H)} (q_H, \pi_0)$ as a graph over the disk $B_{C \ell (H)} (\bar q_H, \pi_0)$, according to \cite{DLS_multiple_valued}*{Proposition 5.2}. We are now going to define a good set $G\subset B_{3\bar r_k/2}(\pi_0)$ as follows 
\begin{itemize}
\item $G\cap F_0$ consists of those points $q\in F_0$ where $f_k^r (q) = Q \llbracket\bar{\boldsymbol{\varphi}}^r_k\rrbracket$;
\item For each $j>0$, $G\cap F_j$ consists of those points $q\in F_j$ where $f_k^r$ coincides with $\bar f_H$ for the corresponding $H$ such that $B_{C\ell (H)} (\bar q_H, \pi_0) = B^j$.
\end{itemize}  
Observe that 
\begin{align*}
B_{3\bar r_k/2}(\pi_0)\setminus G \subset &
\underbrace{\mathbf{p}_{\pi_0} (({\rm spt}\, (T_{0, t_{j(k)}})\setminus {\rm gr}\,  ( f^r_k))\cap \mathbf{C}_{3\bar r_k/2} (0, \pi_0))}_{=: \Xi^1_k}\\
&\qquad\qquad
\cup \underbrace{\mathbf{p}_{\pi_0} ({\rm spt}\, (T_{0, t_{j(k)}}\setminus \mathbf{T}_{F_{j(k)}}) \cap \mathbf{C}_{3\bar r_k/2} (0, \pi_0)))}_{=: \Xi^2_k}\, .
\end{align*}
On the other hand, recalling that $\mathbf{A}_k^2 = o (E_k)$, we can use \cite{DLS14Lp}*{Theorem 2.4} to estimate
\begin{align*}
|\Xi^1_k|  &\leq \mathcal{H}^m ({\rm spt}\, (T_{0, t_{j(k)}})\setminus {\rm gr}\, ( f^r_k)) \cap \mathbf{C}_{3\bar r_k/2} (0, \pi_0)) \\
&\leq \bar r_k^m \mathcal{H}^m ({\rm spt}\, (T_{0, r_k})\setminus {\rm gr}\, (f_k))\cap \mathbf{C}_{M} (0, \pi_0)) = \bar r_k^m O (E_k^{1+\gamma_1})\, .
\end{align*}
As for $\Xi^2_k$, we instead use the analogous estimates for each $f_H$ to get 
\begin{align*}
|\Xi^2_k| &\leq \sum_{H\in \mathscr{C}^{(k)}} \mathcal{H}^m ({\rm spt}\, ((T_{0, t_{j(k)}}\setminus {\rm gr}\, (f_H)))\cap \mathbf{C}_{C\ell (H)} (q_H, \pi_0))
\\
&\leq \sum_{H\in \mathscr{C}^{(k)}} \mathcal{H}^m ((T_{0, t_{j(k)}}\setminus {\rm gr}\, (f_H)))\cap \mathbf{C}_{2 C\ell (H)} (q_H, \pi_H))
\\
&\leq \sum_{L\in \mathscr{C}^{(k)}} \ell (H)^m (\boldsymbol{m}_{0,j(k)} \ell (H)^{2-2\delta_2})^{1+\gamma_1} \leq {\boldsymbol{m}_{0,j(k)}^{1+\gamma_1}} \bar r_k^{m+2+\gamma_1/2}\, 
\end{align*}
(we recall here that the constant {$\gamma_1$} is fixed in \cite{DLS14Lp}, while $\delta_2$ is chosen later in \cite{DLS16centermfld}*{Assumption 1.8} and satifies $(2-2\delta_2)(1+\gamma_1) \leq 1+\gamma_1/2$).

On the other hand, 
\begin{align*}
\|f_k^r\|_{C^0 (B_{3\bar r_k})} 
&\leq C \mathbf{h} (T_{0, t_{j(k)}}, \mathbf{C}_{3\bar r_k} (0, \pi_0))
= C \bar r_k \mathbf{h} (T_{0, r_k}, \mathbf{C}_{3} (0, \pi_0))\leq C \bar r_k E_k^{1/2}\, ,
\end{align*}
where in the latter inequality we have used the information that $0$ is a point of density $Q$ point of $T$ and the {height bound} \cite{Spolaor_15}. Moreover, recalling that
\[
    \|N_{j(k)}\|^2_{L^2 (\mathcal{L})} \leq C \boldsymbol{m}_{0,j(k)}^{1/2} \ell (L)^{m+ 4-2\delta_2}\,,
\]
we infer in particular the existence of at least one point $x\in \mathbf{p}_{\pi_0} (\mathcal{L})$ and $y\in \pi_0^\perp$ such that $(x,y)\in {\rm spt}\, (T_{0,t_{j(k)}})$ and  
\[
|\bar{\boldsymbol{\varphi}}^r_k (x) - y|\leq C \boldsymbol{m}_{0,j(k)}^{1/2} \bar r_k^{2-\delta_2}\, ,
\]
which in turn leads to the bound $|\bar{\boldsymbol{\varphi}}^r_k (x)|\leq C (\boldsymbol{m}_{0,j(k)}^{1/2} + E_k^{1/2}) \bar r_k \leq C E_k^{1/2} \bar r_k$. Note that $\bar{\boldsymbol{\varphi}}^r_k$ is Lipschitz, with a constant uniformly controlled in $k$. We thus conclude that 
\begin{equation}\label{e:laborious_C0_estimate}
\|f_k^r\|_{C^0 (B_{3\bar r_k})} + \|\bar{\boldsymbol{\varphi}}^r_k\|_{C^0 (B_{3\bar r_k})} \leq C E_k^{1/2} \bar r_k\, .
\end{equation}
In particular, combining the latter estimate with $|B_{3\bar r_k/2}\setminus G| \leq C E_k \bar r_k^m$, we conclude that 
\begin{equation}\label{e:bad-set}
\int_{B_{3\bar r_k/2}\setminus G} |\bar{\boldsymbol{\varphi}}^r_k - \boldsymbol{\eta}\circ f_k^r|^2 \leq C E_k^2 \bar r_k^{m+2}\, .
\end{equation}
Considering that on $G\cap F_0$ the functions $\bar{\boldsymbol{\varphi}}^r_k$ and $\boldsymbol{\eta}\circ f_k^r$ coincide, we are left to estimate
\begin{align}\label{e:ugly-series}
\sum_{j\geq 1} \int_{G\cap F_j} |\bar{\boldsymbol{\varphi}}^r_k - \boldsymbol{\eta}\circ f_k^r|^2
\leq E_k^{1/2} \bar r_k \sum_{j\geq 1} \int_{B_{C\ell (H)} (\bar q_H, \pi_0)} |\bar{\boldsymbol{\varphi}}^r_k - \boldsymbol{\eta}\circ \bar f_H|\, .
\end{align}
We now wish to estimate each integral in the above summation by changing coordinates to the reference plane $\pi_H$ for each $H\in\mathscr{C}^{(k)}$. Denote by $\boldsymbol{\varphi}_H$ the function whose graph over $B_{2 C \ell (H)} (q_H, \pi_H)$ coincides with $\mathcal{M}_{j(k)}$ (which, we recall, is the graph of $\bar{\boldsymbol{\varphi}}^r_k$ over an appropriate subset of $\pi_0$). We likewise introduce $\mathbf{f}_H$ which is the function over $B_{2 C \ell (H)} (q_H, \pi_H)$ whose graph coincides with the graph of $\boldsymbol{\eta}\circ \bar f_H$. Applying \cite{DLS16centermfld}*{Lemma B.1(b)} we can then estimate
\[
\int_{B_{C\ell (H)} (\bar q_H, \pi_0)} |\bar{\boldsymbol{\varphi}}^r_k - \boldsymbol{\eta}\circ \bar f_H|
\leq C \int_{B_{2C\ell (H)} (q_H, \pi_H)} |\boldsymbol{\varphi}_H - \mathbf{f}_H|\, .
\]
Let us now estimate 
\begin{equation}\label{e:yet-more-complications}
\int_{B_{2C\ell (H)} (q_H, \pi_H)} |\boldsymbol{\varphi}_H - \mathbf{f}_H| \leq
\int_{B_{2C\ell (H)} (q_H, \pi_H)} |\boldsymbol{\varphi}_H - \boldsymbol{\eta}\circ f_H| +
\int_{B_{2C\ell (H)} (q_H, \pi_H)} |\boldsymbol{\eta}\circ f_H - \mathbf{f}_H|\, .
\end{equation}
In order to handle the second integral we wish to estimate $|\pi_0 - \pi_H|$, since we will be using $C^0$-estimates on $f_H$ here. First of all we compare the tilt between $\pi_0$ and $\pi_{H'}$ for the ancestor $H'$ of $H$ with the smallest side length such that $\mathbf{B}_{H'} \supset \mathbf{B}_{8 M \bar{r}_k}$. Observe that $\ell (H') \leq C \bar r_k$. Since $\pi_{H'}$ optimizes the excess of $T_{0, t_{j(k)}}$ in $\mathbf{B}_{H'}$, while $\pi_0$ optimizes the excess of the same current over $\mathbf{B}_{8 M \bar{r}_k}$, a simple comparison argument (cf. for instance \cite{DLS16centermfld}*{Proof of (4.5)}), implies
\[
|\pi_0 - \pi_{H'}|\leq C (\mathbf{E} (T_{0, t_{j(k)}}, \mathbf{B}_{8M})^{1/2} + \mathbf{E} (T_{0, t_{j(k)}}, \mathbf{B}_H')^{1/2}) 
\leq C E_k^{1/2} + C \boldsymbol{m}_{0,j(k)}^{1/2} \bar r_k^{1-\delta_2}\, .
\]
On the other hand, by \cite{DLS16centermfld}*{Proposition 4.1} we have 
\[
|\pi_H-\pi_{H'}|\leq C \boldsymbol{m}_{0,j(k)}^{1/2} \bar r_k^{1-\delta_2}
\]
and we thus reach
\begin{equation}\label{e:tilt}
|\pi_0 - \pi_H|\leq C E_k^{1/2} + C \boldsymbol{m}_{0,j(k)}^{1/2} \bar r_k^{1-\delta_2} \leq C E_k^{1/2}\, .
\end{equation}
We can now employ \cite{DLS16centermfld}*{Lemma 5.6} to estimate
\[
\int_{B_{2C\ell (H)} (q_H, \pi_H)} |\boldsymbol{\eta}\circ f_H - \mathbf{f}_H| \leq C (\|f_H\|_{C^0(B_{2C\ell (H)} (q_H, \pi_H))} + E_k^{1/2}) ({\rm Dir}\, (f_H) + \ell (H)^m E_k)\, .  
\]
Recall that $\|f_H\|_{C^0(B_{2C\ell (H)} (q_H, \pi_H)} \leq \boldsymbol{m}_{0,j(k)}^{1/2m} \ell (H)^{1+\beta_2}$, while ${\rm Dir}\, (f_H) \leq \boldsymbol{m}_{0,j(k)} \ell (H)^{m+2-2\delta_2} \leq \ell (H)^m E_k$. We thus easily conclude that 
\[
\int_{B_{2C\ell (H)} (q_H, \pi_H)} |\boldsymbol{\eta}\circ f_H - \mathbf{f}_H| \leq C \ell (H)^m E_k^{1+1/2m}\, .
\]
We now come to the first integral in the right hand side of \eqref{e:yet-more-complications}. First of all we recall the tilted interpolating function $h_H$ of \cite{DLS16centermfld}*{Definition 1.16} and observe that, by construction, $\boldsymbol{\varphi}_H$ and $h_H$ coincide in a neighborhood of $q_H$. Now recall that, by \cite{DLS16centermfld}*{Proposition 4.4} $\|D h_H\|\leq C \boldsymbol{m}_{0,j(k)}^{1/2}$. Since moveover $\|D^2 \boldsymbol{\varphi}_H\|$ is controlled by the second fundamental form of $\mathcal{M}_{j(k)}$, which in turn is bounded by $\boldsymbol{m}_{0,j(k)}^{1/2}$, we easily see that the estimate $\|D^2 \boldsymbol{\varphi}_H\| \leq C \boldsymbol{m}_{0,j(k)}^{1/2}$ holds as well. In particular, using a second order Taylor expansion on a point where $\boldsymbol{\varphi}_H-h_H$ and its derivative both vanish (to gain an extra factor of $\ell(H)^2$) we can estimate
\[
\int_{B_{2C\ell (H)} (q_H, \pi_H)} |\boldsymbol{\varphi}_H - \boldsymbol{\eta}\circ f_H| \leq C \boldsymbol{m}_{0,j(k)}^{1/2} \ell (H)^{m+2} +
\int_{B_{2C\ell (H)} (q_H, \pi_H)} |h_H - \boldsymbol{\eta}\circ f_H| \, .
\]
Finally we can use \cite{DLS16centermfld}*{Proposition 5.2} to estimate
\[
\int_{B_{2C\ell (H)} (q_H, \pi_H)} |h_H - \boldsymbol{\eta}\circ f_H| \leq C {\boldsymbol{m}_{0,j(k)}} \ell (H)^{m+3+\beta_2}\, .
\]
In summary, we have reached the estimate
\[
\int_{B_{C\ell (H)} (\bar q_H, \pi_0)} |\bar{\boldsymbol{\varphi}}^r_k - \boldsymbol{\eta}\circ \bar f_H| \leq 
C \boldsymbol{m}_{0,j(k)}^{1/2} \ell (H)^{m+2}\, .
\]
Inserting this into \eqref{e:ugly-series} and decomposing into cubes $H$, we then get 
\begin{align*}
 \int_G |\bar{\boldsymbol{\varphi}}^r_k - \boldsymbol{\eta}\circ \bar f_k^r|^2 
 &\leq C \bar{r}_k E_k^{1/2} \boldsymbol{m}_{0,j(k)}^{1/2} \sum_{H\in \mathscr{C}^{(k)}} \ell (H)^{m+2} \leq C E_k^{1/2} \boldsymbol{m}_{0,j(k)}^{1/2} \bar r_k^{m+3}\, .
\end{align*}
The latter, together with \eqref{e:bad-set}, gives finally \eqref{e:rescaled-comparison} and completes the proof of the lemma.
\end{proof}

\begin{proof}[Proof of Proposition \ref{p:coarse=fine}]
We wish to compare
\[
\tilde{N}_k := \frac{N_k\circ\boldsymbol{\vphi}_k}{\mathbf{h}_k} \qquad \text{and} \qquad v_k \coloneqq \frac{\sum_i \llbracket (f_k)_i - \boldsymbol{\eta}\circ f_k\rrbracket}{E_k^{1/2}}\, ,
\]
in particular we wish to show that they have the same $L^2$ limit, up to a scalar constant. Since both sequences are converging to respective Dir-minimizing maps, it suffices to compare the maps $\tilde N_k$ and $v_k$ on some nonempty open set; we will do it on $B_1(\pi_0)$ for simplicity. 

First of all we replace $\boldsymbol{\eta}\circ f_k$ with the parameterizing map $\boldsymbol{\vphi}_k$ for $\Mcal_k$ in $v_k$ to give a map $\hat{v}_k$ given by
\[
\hat{v}_k = \frac{\sum_i \llbracket (f_k)_i - \boldsymbol{\vphi}_k \rrbracket}{E_k^{1/2}} \, ,
\]
since Lemma~\ref{lem:height-excess} implies that 
\[
\lim_{k\uparrow \infty} \int_{B_{3/2}(\pi_0)} \mathcal{G} (v_k, \hat{v}_k)^2 = 0\, . 
\]
Recalling \cite{DLS14Lp}, 
\[
|\mathbf{p}_k (({\spt (T_{0,r_k})}\setminus {\rm gr}\, (f_k) \cup {\rm gr}\, f_k\setminus {\rm spt}\, (T_{0,r_k}) \cap \mathbf{C}_{3/2})| = o (E_k)\, .
\] 
Next, introduce the map $\mathbf{F} (p) := \sum_i \llbracket (N_k)_i (p)+p\rrbracket$ on $\Mcal_k$ and let $f^1_k: B_2 (0, \pi_0) \to \mathcal{A}_Q (\pi_0^\perp)$ be the map whose graph coincides with the current $\mathbf{T}_F\cap \mathbf{C}_2 (0, \pi_0)$. By \cite{DLS16centermfld}*{Theorem 2.4} and \cite{DLS16blowup}*{Section 4.2 \& Corollary 5.3},
\[
|\mathbf{p}_{\pi_0} (({\rm gr}\, (f^1_k)\setminus {\rm spt}\, (T_{0,r_k}) \cup {\rm spt}\, (T_{0,r_k})\setminus {\rm gr}\, (f^1_k))\cap \mathbf{C}_{3/2}| = o (E_k)\, .
\]
In particular, if we consider the map
\[
\hat{v}^1_k = \frac{\sum_i \llbracket (f^1_k)_i - \boldsymbol{\vphi}_k\rrbracket}{E_k^{1/2}}
\]
we have that $|\{\hat{v}^1_k \neq \hat v_k\}| \to 0$, and using that both have a uniform bound on the Dirichlet energy, we conclude that
\[
\lim_{k\to\infty}\int_{B_{3/2}} \mathcal{G} (\hat{v}^1_k, \hat{v}_k)^2 = 0\, .
\]
We also take advantage of Lemma~\ref{lem:height-excess} to assume, up to extraction of a subsequence (not relabeled), that $E_k/\mathbf{h}_k^2$ converges to some finite constant $\lambda >0$. We are therefore left to show that the maps $\tilde{N}_k$ and 
\[
\hat{v}^2_k = \frac{\sum_i \llbracket (f^1_k)_i - \boldsymbol{\vphi}_k\rrbracket}{\mathbf{h}_k^{1/2}}
\]
have the same limit. We now wish to apply Lemma \ref{l:reparametrization-lemma} to the maps $N_k$. We observe that the map $g$ in Lemma \ref{l:reparametrization-lemma} can be taken to be the map $g_k$ defined by
\[
g_k:= \sum_i \llbracket (f^1_k)_i - \boldsymbol{\vphi}_k\rrbracket \, .
\]
Moreover, observe that $\|D \boldsymbol{\varphi}_k\|_{C^0}$ converges to $0$. 
If we had a uniform bound on $\|g_k\|_{C^0}$ in terms of $\mathbf{h}_k$ we could then apply Lemma \ref{l:reparametrization-lemma} to complete the proof. Given that we only have the bound $\|g_k\|_{L^2} \leq C \mathbf{h}_k$ we need to overcome this issue. We use the following simple argument. We fix a truncation parameter $\bar{M}$ and introduce the truncation
\[
g_k^{\bar M} := \sum_i \llbracket(g_k)_i^{\bar M}\rrbracket
\] 
where the maps $(g_k)_i^{\bar M}$ are defined by replacing each component $(\xi_i)_j (x)$ of the vector $(g_k)_i (x)$ with $\max \{-\bar M, \min \{(\xi_i)_j (x), \bar M\}\}$. By the Sobolev embedding and the uniform $W^{1,2}$ bound on $g_k$ it is easy to see that 
\[
\lim_{\bar M\to\infty} \sup_k \mathbf{h}_k^{-2} \int \mathcal{G} (g_k, g_k^{\bar M})^2 = 0\, .
\]
Likewise, after defining the maps $N_k^{\bar M}$ as those corresponding to $g_k^{\bar M}$ in the same way as $N_k$ corresponds to $g_k$, we see as well
\[
\lim_{\bar M\to \infty} \sup_k \mathbf{h}_k^{-2} \int \mathcal{G} (N_k, N_k^{\bar M})^2 = 0\, .
\]
We can now apply Lemma \ref{l:reparametrization-lemma} to conclude that the limit (in $k$) of $\mathbf{h}_k^{-1} N^{\bar M}_k \circ \boldsymbol{\varphi}_k$ and the limit of $g_k^{\bar M}$ coincides on $B_1$. Letting $\bar M\to \infty$ we then reach the desired conclusion.
\end{proof}

\section{Frequency bound for fine blow-ups}\label{s:HS2}

In this section we prove the lower bound for the frequency values, which we equivalently restate as follows for the reader's convenience.

\begin{theorem}\label{t:lower-bound}
Suppose that $T$ and $\Sigma$ are as in Assumption~\ref{asm:3} and let $u$ be a fine blow-up. Then $I_u (0) \geq 1$.
\end{theorem}

In order to show the theorem, we fix a blow-up sequence $\{r_k\}$ which generates the fine blow-up $u$ through the procedure described in Section \ref{ss:compactness} and for each $k$ sufficiently large we choose the interval of flattening $]s_{j(k)}, t_{j(k)}]$ which contains the radius $r_k$. We can then reduce the proof, up to extraction of a subsequence, to three different cases. In the first case we assume that there are finitely many intervals of flattening and hence (up to subsequence), there is a positive integer $J$ such that:
\begin{equation}\label{e:case-1}
s_{J} = 0 \qquad \mbox{and}\qquad \{r_k\}_k \;\subset\; ]0, t_{J}]\, .
\end{equation}
In the remaining two cases we assume that there are infinitely many intervals of flattening and that (up to subsequence) one of the following mutually exclusive conditions hold:
\begin{align}
&\lim_k \frac{s_{j(k)}}{r_k} > 0\label{e:case-2}\\
&\lim_k \frac{s_{j(k)}}{r_k}=0 . \label{e:case-3}
\end{align}

The proof will take advantage of a first coarse lower bound proved recently by the second author, cf. \cite{Sk21}*{Theorem 7.8}, which in turn can be combined with the monotonicity computations in \cite{DLS16blowup} to give a suitable almost-monotonicity formula for $\Ibf_N$, cf. \cite{Sk21}*{Theorem 7.4} as well. We summarize these conclusions in the following theorem.

\begin{theorem}\label{t:first-lb}
Let $T$, $\Sigma$ be as in Assumption \ref{asm:3} and consider any center manifold $\mathcal{M}_j$ and any normal approximation $N_j$ for a given interval of flattening $]s_j, t_j]$ at $0$. Then, 
\begin{align}
\Ibf_{N_j}(r) &\geq c_0 \qquad \forall r\in \left]\frac{s_j}{t_j}, 3\right]\, ,\\
\Ibf_{N_j} (a) &\leq e^{C b^\alpha} \Ibf_{N_j} (b) \qquad \forall ]a, b]\subset \left]\frac{s_j}{t_j}, 3\right]\, ,\label{e:first-almost-monotonicity}
\end{align}
where $\alpha = \alpha(Q,m,n) > 0$, while $c_0$ and $C$ are positive numbers which depend on $T$ (but not on $j$).  
\end{theorem}

\subsection{Proof of Theorem \texorpdfstring{\ref{t:lower-bound}}{t:lower-bound} under assumption \texorpdfstring{\eqref{e:case-1}}{e:case-1}} We let $\Mcal$ be the center manifold related to the interval of flattening $]0, t_J]$, with corresponding normal approximation $N$. Since we are in the case with a single center manifold, we omit the dependency on $N$ for $\Ibf$ and related quantities. Observe that, by Theorem \ref{t:first-lb},
\[
c(m,Q) \leq \Ibf (a) \leq e^{C b^\alpha} \Ibf (b) \qquad \forall 0<a\leq b < 3
\]
and in particular we immediately see that 
\[
c_0 \leq \limsup_{r\downarrow 0} \Ibf (r) \leq \liminf_{r\downarrow 0} \Ibf (r) < +\infty\, .
\]
So the limit $I_0:= \lim_{r\downarrow 0} \Ibf (r)$ exists and it is positive and finite. It follows from the strong convergence of $u_k$ from the definition of $u$ being a fine blowup, that $I_{u} (r)$ is identically equal to $I_0$, and thus $I_0 = I_{u} (0)$. Therefore it just suffices to show that $I_0 \geq 1$. On the other hand, by \cite{DLS16blowup}*{Proposition 3.5}, we readily see that 
\[
\left|\frac{d}{dr} \log \frac{\Hbf (r)}{r^{m-1}}- \frac{2 \Ibf (r)}{r}\right| \leq  \frac{C \Ibf (r)}{r^{1-\gamma}}\, , 
\]
for suitable constants $C$ and $\gamma >0$. In particular, for every $\varepsilon >0$, the inequalities
\[
\frac{2I_0 -\varepsilon}{r} \leq \frac{d}{dr} \log \frac{\Hbf (r)}{r^{m-1}} \leq \frac{2I_0 +\varepsilon}{r}\, 
\]
hold as soon as $r$ is smaller than a suitable scale $r (\varepsilon)>0$. 
Integrating the latter differential inequality, we immediately conclude that 
\[
\liminf_{r\downarrow 0} \frac{\Hbf (r)}{r^{m-1+2I_0+\varepsilon}} > 0\,
\]
for every $\varepsilon > 0$. Combined with the inequality $\frac{r \Dbf (r)}{\Hbf (r)} = \Ibf (r) \geq c_0$, we also conclude that 
\[
\liminf_{r\downarrow 0} \frac{\Dbf (r)}{r^{m+2 (I_0-1) +\varepsilon}} > 0\, .
\]
On the other hand, due to the estimate \cite{DLS16blowup}*{(3.4)} and the fact that $s_{J} =0$, we must have
\[
\Dbf (r) \leq C r^{m+2-2\delta_2}
\]
where $\delta_2$ is the small positive constant of \cite{DLS16centermfld}*{Assumption 1.8}. Comparing this with the previous asymptotic estimate, we conclude in particular that 
\[
2 (I_0 -1) \geq 2-2\delta_2\, ,
\]
and since $2\delta_2 \leq \frac{1}{4m}$, we immediately get that $I_0>1$ (in fact it turns out that $I_0$ is rather close to $2$, in this case). 

\subsection{Proof of Theorem \texorpdfstring{\ref{t:lower-bound}}{t:lower-bound} under assumption \texorpdfstring{\eqref{e:case-2}}{e:case-2}} In this case we can apply Proposition \ref{p:coarse=fine} to a suitable subsequence of $\{r_k\}_k$, not relabeled, and find a coarse blow-up $f$ whose average-free part $v$ has the property that $v = \lambda u$ for some positive number $\lambda$. 
In particular $I_u (0) = I_v (0)$ and from Theorem \ref{thm:HS} we conclude $I_u (0) \geq 1$.

\subsection{Proof of Theorem \texorpdfstring{\ref{t:lower-bound}}{t:lower-bound} under assumption \texorpdfstring{\eqref{e:case-3}}{e:case-3}} We fix a blow-up sequence $\{r_k\}_k$ and a corresponding fine blow-up $u$. One crucial property that we will use is that, because of the convergence of the maps $u_k$ from Section~\ref{ss:compactness} to the fine blow-up $u$, for every positive $\rho<1$ we have 
\begin{equation}\label{e:convergence-of-I}
I_{u} (\rho) = \lim_{k\to \infty} \mathbf{I}_{N_{j(k)}} \left(\frac{\rho r_k }{t_{j(k)}}\right) 
\end{equation}
Observe that under our assumption we know as well that $\frac{s_{j(k)}}{t_{j(k)}}$ is infinitesimal. In particular, since 
\[
\mathbf{E} (T, \mathbf{B}_{s_{j(k)}}) = \mathbf{E} (T_{0, T_{j(k)}}, \mathbf{B}_{s_{j(k)}/t_{j(k)}}) \leq C \boldsymbol{m}_{0,j(k)} \frac{s_{j(k)}^{2-2\delta_2}}{t_{j(k)}^{2-2\delta_2}}\, , 
\]
we conclude that $\mathbf{E} (T, \mathbf{B}_{s_{j(k)}}) \to 0$. So $s_{j(k)}$ is itself a blow-up sequence, and we can apply the previous section to infer that, for any $u'$ coarse blow-up generated by a subsequence, we have $I_{u'} (0) \geq 1$. In particular, since along this subsequence of $\{s_{j(k)}\}$ we have comparability of the coarse and fine blow-ups due to Proposition~\ref{p:coarse=fine}, we can use the corresponding convergence \eqref{e:convergence-of-I} to infer that 
\[
\liminf_{k\to \infty} \mathbf{I}_{N_{j(k)}} \left(\frac{s_{j(k)}}{t_{j(k)}}\right) \geq 1\, .
\]
Fix now an arbitrary small parameter $\delta>0$. Our goal is to show that there is $\bar \rho >0$ such that
\begin{equation}\label{e:epsilon-delta}
\liminf_{k\to \infty} \mathbf{I}_{N_{j(k)}} \left(\frac{\rho r_k }{t_{j(k)}}\right) \geq 1-2 \delta \qquad \forall \rho \in \left]\frac{s_{j(k)}}{r_k}, \bar\rho \right[ \, .
\end{equation}
Knowing \eqref{e:epsilon-delta} and \eqref{e:convergence-of-I}, we would then infer that $I_{u} (\rho) \geq 1-2\delta$ for every positive $\rho < \bar \rho$, which in turn would imply $1-2\delta \leq I_{u} (0)$. 
The arbitrariness of $\delta$ then tells us that $I_u (0) =0$. 

In order to achieve \eqref{e:epsilon-delta}, choose first $k_0$ large enough so that 
\[
\mathbf{I}_{N_{j(k)}} \left(\frac{s_{j(k)}}{t_{j(k)}}\right) \geq 1-\delta \qquad \forall k\geq k_0\, .
\]
Next, because of \eqref{e:first-almost-monotonicity} we can choose $\sigma>0$ small enough (independent of $k$) with the property that 
\[
\mathbf{I}_{N_{j(k)}} (r) \geq 1-2\delta \qquad \forall r\in \left]\frac{s_{j(k)}}{t_{j(k)}}, \sigma\right],\quad \forall k \geq k_0\, .
\]
Since however $r_k \leq t_{j(k)}$, while $\lim_{k\to \infty} \frac{s_{j(k)}}{r_k} = 0$, for any fixed positive $\rho<\sigma$ and for every $k$ large enough we may conclude that
$\frac{\rho r_k}{t_{j(k)}}$ must belong to the interval $[\frac{s_{j(k)}}{t_{j(k)}}, \sigma]$. This implies \eqref{e:epsilon-delta} with $\bar \rho = \sigma$ and thus completes the proof.

\section{Frequency BV estimate}\label{ss:bv}

This section is dedicated to establishing a (quantitative) control on the radial variations of the frequency, which is crucial for proving Theorem~\ref{t:consequences}.

We begin by defining the~\emph{universal frequency function}, which makes sense of the frequency continuously along all blow-up scales where it is possible to construct a center manifold for $T$.

\begin{definition}[Universal frequency function]\label{def:univfreq}
	Suppose that $T$ is as in Assumption~\ref{asm:1} and let $\{]s_k, t_k]\}_{k=j_0}^J$ be a sequence of intervals of flattening with coinciding endpoints (i.e. such that $s_k= t_{k+1}$ for $k = j_0,\dots,J-1$), with corresponding center manifolds $\Mcal_k$ and $\Mcal_k$-normal approximations $N_k$. For $r \in ]s_J, t_{j_0}]$, define
	\begin{align*}
	    \Ibf(r) &\coloneqq \Ibf_{N_k}\left(\frac{r}{t_k}\right) \chi_{]s_k,t_k]}(r), \\
	    \Dbf(r) &\coloneqq \Dbf_{N_k}\left(\frac{r}{t_k}\right) \chi_{]s_k,t_k]}(r), \\
	    \Hbf(r) &\coloneqq \Hbf_{N_k}\left(\frac{r}{t_k}\right) \chi_{]s_k,t_k]}(r).
	\end{align*}
\end{definition}
Note that there may be gaps between intervals of flattening in general, but the universal frequency function only makes sense over uninterrupted strings of intervals of flattening.

Unfortunately, unlike for the linearized problem, we do not have monotonicity of the frequency but merely almost monotonicity. Nevertheless, we can hope to control the variation of the negative part of the radial derivative for the frequency function. The main result of this section is the following proposition. We will use the convention that, given a BV function $f$ of one variable, $\left[\frac{df}{dr}\right]_\pm$ will denote the positive and negative parts of its distribiutional derivatives, while $\|\mu\|_{TV}$ denotes the total variation of a measure $\mu$ on its domain of definition.

\begin{proposition}\label{prop:bv}
    There exists $\gamma_4=\gamma_4(m,n,Q) > 0$ and $C= C (m,n,Q)$ such that the following holds. Suppose that $T$ satisfies Assumption \ref{asm:3}. Let $\{]s_k, t_k]\}_{k=j_0}^{J}$ be intervals of flattening for $T$ around $0$ with coinciding endpoints. Then we have $\log (\Ibf+1) \in \BV(]s_J, t_{j_0}])$, with the quantitative estimate
    \begin{equation}\label{eq:bv}
        \left\| \left[\frac{d \log (\Ibf+1)}{dr}\right]_- \right\|_{\TV(]s_J, t_{j_0}])} \leq C \sum_{k=j_0}^J  \boldsymbol{m}_{0,k}^{\gamma_4}\, .
    \end{equation}
    Moreover, for any $]a,b]$ which is contained in a single interval of flattening $]s_k, t_k[$ we have the improved estimate
    \begin{equation}\label{eq:bv-improved}
        \left\| \left[\frac{d \log (\Ibf+1)}{dr}\right]_- \right\|_{\TV(]a,b])} \leq C \left({\frac{b}{t_k}}\right)^{\gamma_4}  \boldsymbol{m}_{0,k}^{\gamma_4}\, .
    \end{equation}
\end{proposition}

\begin{remark} In our subsequent work \cite{DLSk2} we will need the BV estimate of Proposition \ref{prop:bv} for a different definition of the universal frequency function, for which the intervals of flattening $]s_j, t_j]$ are chosen differently. We point out that, the crucial ingredients needed in proving the above estimates are the following:
\begin{itemize}
    \item[(a)] The estimate in each open interval holds because for each $r\in ]\frac{s_j}{t_j},1]$ the side length $\ell (L)$ of any cube $L\in \mathscr{W}^{(k)}$ which intersects $B_r (0, \pi_0)$ is no larger than $c_s r$ for a fixed constant $c_s = \frac{1}{64\sqrt{m}}$.
    \item[(b)] The estimate at the jumps holds because there is one cube $L\in \mathscr{W}^{(k)}$ which intersects $B_{s_j/t_j} (0, \pi_0)$ and has side length $\ell (L) \geq c_s \frac{s_j}{t_j}$.
\end{itemize}
While in (a) we cannot afford a similar control with a constant larger than $c_s$, in (b) we can afford a constant $\bar{c}_s$ smaller than $c_s$, at the price that the constant $C$ in the estimate \eqref{eq:bv} will then depend on how small $\bar{c}_s$ is.  
\end{remark}

In order to prove this, we will require a number of preliminary results, the proofs of which we will defer until later.

\subsection{Auxiliary results for Proposition \ref{prop:bv}}\label{ss:aux}

First of all, we recall some key variational identities and estimates from~\cite{DLS16blowup} for any normal approximation of $T$, which are a nonlinear analogue of the identities in~\cite{DLS_MAMS}*{Section~3.4}.

Let $]s,t]$ be an interval of flattening for $T$ around $0$ with corresponding center manifold $\Mcal$ and $\Mcal$-normal approximation $N$. We define the quantities
\begin{align*}
	\Ebf_N(r) &\coloneqq -\frac{1}{r} \int_{\Mcal} \phi'\left(\frac{d(y)}{r}\right)\sum_i N_i(y)\cdot DN_i(y)\nabla d(y) \dd y\, , \\
	\Gbf_N(r) &\coloneqq -\frac{1}{r^2} \int_{\Mcal} \phi'\left(\frac{d(y)}{r}\right) \frac{d(y)}{|\nabla d(y)|^2} \sum_i |DN_i(y) \cdot \nabla d(y)|^2 \dd y\, , \\
	\boldsymbol{\Sigma}_N(r) &\coloneqq \int_{\Mcal} \phi\left(\frac{d(y)}{r}\right)|N(y)|^2\dd y\, .
\end{align*}

We thus have the following.
\begin{lemma}\label{lem:firstvar}
	There exist $\gamma_4 (m,n,Q) > 0$ and $C (m,n,Q) > 0$ such that the following holds. Suppose that $T$, $\Sigma$ satisfy Assumption \ref{asm:3} and let $]s,t]$ be an interval of flattening for $T$ around $0$ with corresponding center manifold $\Mcal$ and $\Mcal$-normal approximation $N$. Let $\boldsymbol{m}_0$ be as in~\eqref{eq:m_0} for $]s,t]$. Then $\Dbf_N$ and $\Hbf_N$ are absolutely continouous on $]\frac{s}{t},3]$ and for a.e. $r$ we have 
	\begin{align}
		&\partial_r \Dbf_N(r) = - \int_{\Mcal} \phi'\left(\frac{d(y)}{r}\right) \frac{d(y)}{r^2} |DN(y)|^2 \ \dd y \label{eq:firstvar1} \\
		&\partial_r \Hbf_N(r) - \frac{m-1}{r} \Hbf_N (r) = O(\boldsymbol{m}_0) \Hbf_N (r) + 2 \Ebf_N(r), \label{eq:firstvar2}\\
		&|\Dbf_N(r) - \Ebf_N(r)| \leq \sum_{j=1}^5 |\Err_j^o| \leq C\boldsymbol{m}_0^{\gamma_4}\Dbf_N(r)^{1+\gamma_4} + C\boldsymbol{m}_0\boldsymbol{\Sigma}_N(r), \label{eq:firstvar3}\\
		&\left|\partial_r \Dbf_N(r)  - (m-2) r^{-1} \Dbf_N(r)- 2\Gbf_N(r)\right| \leq 2 \sum_{j=1}^5 |\Err_j^i|  + C \boldsymbol{m}_0\Dbf_N(r) \\
		&\qquad \leq Cr^{-1}\boldsymbol{m}_0^{\gamma_4}\Dbf_N(r)^{1+\gamma_4} + C\boldsymbol{m}_0^{\gamma_4}\Dbf_N(r)^{\gamma_4}\partial_r \Dbf_N(r) +C\boldsymbol{m}_0 \Dbf_N(r),\notag
	\end{align}
	where $\Err_j^o$ and $\Err_j^i$ are as in~\cite{DLDPHM}*{Proposition~9.8,~Proposition~9.9}.
\end{lemma}

We omit the proof of Lemma~\ref{lem:firstvar} here, since it involves a mere repetition of the arguments in the proofs of~\cite{DLS16blowup}*{Proposition~3.5} (see also~\cite{DLDPHM}*{Proposition~9.5,~Proposition~9.10}), combined with the observation that the constants may be optimized to depend on appropriate powers of $\boldsymbol{m}_0$. This is crucial in order to obtain the quantitative BV estimate of Proposition \ref{prop:bv}. Without such an improvement of the variational estimates one would merely obtain a constant bound on the total variation on each interval of flattening, which is insufficient to obtain a convergent series when summing over a string of uninterrupted intervals of flattening. As a consequence of the estimates in Lemma~\ref{lem:firstvar}, we have the following quantitative almost-monotonicity for the frequency in each interval of flattening.

\begin{corollary}\label{cor:freqmono}
	There exist $\gamma_4 (m,n,Q) > 0$ and $C (m,n,Q) > 0$ such that the following holds. Suppose that $T$, $\Sigma$, $]s,t]$, $\mathcal{M}$, $N$, and $\boldsymbol{m}_0$ are as in Lemma \ref{lem:firstvar}. Then $\Ibf_N$ is absolutely continuous on $]\frac{s}{t}, r]$ and for a.e. $r$ we have 
	\begin{align*}
		\partial_r\Ibf_N(r) &\geq -C(1+\Ibf_N (r)) \boldsymbol{m}_0^{\gamma_4} \left(1 + \frac{\Dbf_N(r)^{\gamma_4}}{r} + \Dbf_N(r)^{\gamma_4-1}\partial_r\Dbf_N(r)\right) \, .
	\end{align*}
\end{corollary}

In addition to the above control on the frequency variations within each interval of flattening, we will also need to control the jumps of the frequency between successive intervals of flattening. In order to establish this, we will require the following intermediate results.

\begin{lemma}[Expansion of excess]\label{lem:excessTaylor}
	There exists a dimensional constant $C = C(m,n,Q) > 0$ such that the following holds. Let $T$, $\Sigma$ be as in Assumption \ref{asm:3} and let $\Mcal$ be a center manifold for $T$ with $\Mcal$-normal approximation $N$. Let $r \in ]0,1]$ and let $f: B_r(0,\pi)\to\Acal_Q(\pi^\perp)$ be a Lipschitz map with $\Lip(f) \leq c$. Let $\boldsymbol{\vphi}_r$ be a parameterizing map for $\Mcal$ over $\pi$. Then we have
	\begin{align*}
		\left| \int_{\Cbf_r(0,\pi)} \right.&\left.|\vec\Gbf_{f} - \vec{\Mcal}\circ\mathbf{p}|^2\phi\left(\frac{|\mathbf{p}_{\pi} (z)|}{r}\right) \dd\|\Gbf_{f}\|(z) - \int_{B_r(0,\pi)} \Gcal\big(Df, Q\llbracket D\boldsymbol{\vphi}_r\rrbracket\big)^2 \dd y \right| \\
		&\leq C\int_{B_r(0,\pi)} (|Df|^4 + |D\boldsymbol{\vphi}_r|^4) \phi\left(\frac{|y|}{r}\right)\, \dd y  \\
		&\qquad +C \int_{\Cbf_r(0,\pi)} \Big| \vec{\Mcal}(\mathbf{p}(z)) - \vec{\Mcal}\big(\boldsymbol{\vphi}_r(\mathbf{p}_{\pi}(z))\big) \Big| \dd\|\Gbf_{f}\|(z).
	\end{align*}
\end{lemma}
An important consequence of Lemma~\ref{lem:excessTaylor} is the following comparability between the Dirichlet energy of $N$ at a given scale, with that of Lipschitz approximations over suitable planes. We will henceforth take $\gamma_2 > 0$ to be as in~\cite{DLS16centermfld}. Note that we may ensure that $\gamma_4 \leq \gamma_2$.

\begin{corollary}\label{cor:graphDir}
	There exists a dimensional constant $C = C(m,n,Q) > 0$ such that the following holds. Let $T$, $\Sigma$ satisfy Assumption \ref{asm:3}. Let $]s,t]$ be an interval of flattening for $T$ around $0$ with corresponding center manifold $\Mcal$ and $\Mcal$-normal approximation $N$, let $\boldsymbol{m}_0$ be as in~\eqref{eq:m_0} for $]s,t]$ and let $\pi$ be the plane used to define $\boldsymbol{\varphi}$ in the center manifold algorithm of \cite{DLS16centermfld}. Let $f: B_1(0,\pi) \to \Acal_Q(\pi^\perp)$ be a $\pi$-approximation for $T_{0,t}$ in $\Cbf_4(0,\pi)$ according to \cite{DLS14Lp} and for $\bar{r}=\frac{s}{t}$, let $f_L: B_{8r_L}(p_L,\pi_L) \to \Acal_Q(\pi_L^\perp)$ be a $\pi_L$-approximation for $T_{0,t}$ corresponding to a Whitney cube $L$ as in~\cite{DLS16blowup}*{Section~2.1~(Stop)}. Let $\pi_{\bar{r}}$ be such that $\Ebf(T_{0,t},\Bbf_{6\sqrt{m}\bar{r}}) = \Ebf(T_{0,t},\Bbf_{6\sqrt{m}\bar{r}}, \pi_{\bar{r}})$ and let $B^L \coloneqq B_{8r_L}(p_L,\pi_L)$. Let $f_{\bar{r}}:B_{\bar{r}}(0,\pi_{\bar{r}}) \to \Acal_Q(\pi_{\bar{r}}^\perp)$ be the map reparameterizing ${\rm gr}\, (f_L)$ as a graph over $\pi_{\bar{r}}$ and let $\boldsymbol{\vphi}_{\bar{r}}, \boldsymbol{\vphi}_L$ be the maps reparameterizing ${\rm gr} (\boldsymbol{\vphi})$ as graphs over $\pi_{\bar{r}}, \pi_L$ respectively. Then we have
	\begin{align}
		\left| \int_{B_1(0,\pi)} \right.&\left.\Gcal(Df, Q\llbracket D\boldsymbol{\vphi} \rrbracket)^2\phi\left(|y|\right)\dd y - \int_{\Bbf_1 \cap \Mcal} |DN|^2 \phi\left(d(y)\right)\dd y \right| \label{eq:graphDir1} \\
		&\leq C\int_{B_1(0,\pi)} (|Df|^4 + |D\boldsymbol{\vphi}|^4)\dd y  + C\boldsymbol{m}_0^{1+\gamma_2} + C \int_{\Bbf_1 \cap \Mcal} (|\Abf_{\Mcal}|^2|N|^2 + |DN|^4)\notag\\
		&\qquad +C \int_{\Cbf_1(0,\pi)} \Big| \vec{\Mcal}(\mathbf{p}(z)) - \vec{\Mcal}\big(\boldsymbol{\vphi}(\mathbf{p}_{\pi}(z))\big) \Big|  \dd\|\Gbf_{f}\|(z),\notag
	\end{align}
	and
	\begin{align}
		\left| \int_{B_{\bar{r}}(0,\pi_{\bar{r}})} \right.&\left.\Gcal(Df_{\bar{r}}, Q\llbracket D\boldsymbol{\vphi}_{\bar{r}} \rrbracket)^2\phi\left(\frac{|y|}{\bar{r}}\right)\dd y - \int_{\Bbf_{\bar{r}}\cap\Mcal} |DN|^2 \phi\left(\frac{d(y)}{\bar{r}}\right)\dd y \right| \label{eq:graphDir2}\\
		&\leq C\int_{B_{\bar{r}}(0,\pi_{\bar{r}})} (|Df_{\bar{r}}|^4 + |D\boldsymbol{\vphi}_{\bar{r}}|^4) \dd y + C\int_{B^L} (|Df_{L}|^4 + |D\boldsymbol{\vphi}_L|^4) \dd y   \notag\\
		&\qquad+ C\boldsymbol{m}_0^{1+\gamma_2}\bar{r}^{m+2+\gamma_2} + C \int_{\Bcal^L} (|\Abf_{\Mcal}|^2|N|^2 + |DN|^4) \notag\\
		&\qquad +C \int_{\Cbf_{\bar{r}}(0,\pi_{\bar{r}})} \Big| \vec{\Mcal}(\mathbf{p}(z)) - \vec{\Mcal}\big(\boldsymbol{\vphi}(\mathbf{p}_{\pi_{\bar{r}}}(z))\big) \Big| \dd\|\Gbf_{f_{\bar{r}}}\|(z).\notag
	\end{align}
\end{corollary}

We will in addition require the following comparison between the gradients of the parameterizing maps of consecutive center manifolds in the procedure~\cite{DLS16blowup}*{Section~2.1}.
\begin{lemma}\label{lem:cm}
	There exists a constant $C=C(m,n,Q) > 0$ such that the following holds. Suppose that $T$, $\Sigma$ satisfy Assumption \ref{asm:3}. Let $\Mcal_{k-1}, \Mcal_{k}$ be successive center manifolds for $T$ with respective normal approximations $N_{k-1},N_k$, associated to the respective intervals of flattening $]t_k,t_{k-1}]$ and $]t_{k+1}, t_k]$, as defined in Section~\ref{s:setup}. Assume that $\Ebf(T,\Bbf_{6\sqrt{m}t_k},\pi_k) = \Ebf(T,\Bbf_{6\sqrt{m}t_k})$ for some plane $\pi_k$ and let $\tilde{\boldsymbol{\vphi}}_{k-1}$ be the map reparametrizing ${\rm gr}\, (\boldsymbol{\vphi}_{k-1})$ as a graph over $\pi_k$.  
	%Let $f_{k-1}, f_k:\pi_0\supset B_{3/2}\to\Acal_Q(\pi_0^\perp)$ be the respective Lipschitz approximation of~\cite[Theorem~2.4]{DLS14Lp} for the rescaled currents $T_{k-1}$ and $T_k$ in $\Bbf_{6\sqrt{m}}$. 
	Letting $\tilde{\boldsymbol{\vphi}}_k \coloneqq \tilde{\boldsymbol{\vphi}}_{k-1}\left(\frac{t_k}{t_{k-1}}\cdot \right)$, we have
	\begin{equation}\label{e:comparison-cm-1}
	    \int_{{B_1}} |D\boldsymbol{\vphi}_k - D\tilde{\boldsymbol{\vphi}}_{k}|^2 \leq C \boldsymbol{m}_{0,k}^{3/2}.
	\end{equation}
 and
 \begin{equation}\label{e:comparison-cm-2}
	    \int_{{B_2}} |\boldsymbol{\vphi}_k - \tilde{\boldsymbol{\vphi}}_{k}|^2 \leq C \boldsymbol{m}_{0,k}\, .
	\end{equation}
\end{lemma}

Finally, we will need the following control on the difference between the projection $\mathbf{p}(z)$ to a center manifold $\Mcal$ of a point $z$ in the multigraph of a given Lipschitz approximation, and the image under $\boldsymbol{\vphi}$ of the planar projection $\mathbf{p}_{\pi_0} (z)$:
\begin{lemma}\label{lem:cmLip}
	There exists a constant $C=C(m,n,Q) > 0$ such that the following holds. Suppose that $T$, $\Mcal$, $\boldsymbol{m}_0$, $\bar{r}$, $f$, $f_{\bar{r}}$, $\pi$, $\pi_{\bar{r}}$, $\boldsymbol{\vphi}_{\bar{r}}$ are as in Corollary~\ref{cor:graphDir}. Then we have
	\begin{align}
		&\int_{\Cbf_{\bar{r}}(0,\pi_{\bar{r}})} \Big| \vec{\Mcal}(\mathbf{p}(z)) - \vec{\Mcal}\big(\boldsymbol{\vphi}_{\bar{r}}(\mathbf{p}_{\pi_{\bar{r}}}(z))\big) \Big| \dd\|{\Gbf_{{f}_{\bar r}}}\|(z) \leq C\bar{r}^{m+1}\boldsymbol{m}_0^{1+\gamma_2}, \label{eq:cmLip1}\\
		&\int_{\Cbf_{1}(0,\pi)} \Big| \vec{\Mcal}(\mathbf{p}(z)) - \vec{\Mcal}\big(\boldsymbol{\vphi}(\mathbf{p}_\pi(z))\big) \Big| \dd\|\Gbf_{f}\|(z) \leq C\boldsymbol{m}_0^{1+\gamma_2}.\label{eq:cmLip2}
	\end{align}
\end{lemma}

\subsection{Proof of Proposition~\ref{prop:bv}} We now have all of the relevant tools to prove the frequency variation estimate~\eqref{eq:bv}. We start with the preliminary observation that $\mathbf{I}$ is absolutely continuous on each interval $]s_k, t_k[$, while it might have jump discontinuities at the points $s_k=t_{k+1}$.

First, we control the jumps of $\Ibf$ at these points. Letting $\Dbf_k \coloneqq \Dbf_{N_k}$, $\Hbf_k \coloneqq \Hbf_{N_k}$, and letting $\bar{\Dbf}_k(r) \coloneqq r^{-(m-2)}\Dbf_k(r)$, $\bar{\Hbf}_k(r) \coloneqq r^{-(m-1)}\Hbf_k(r)$ denote the corresponding scale-invariant quantities, we claim that we have the estimate
	\begin{equation}\label{eq:BVjumps}
		\left|\Ibf(t_k^+) - \Ibf(t_k^-)\right| = \left|\frac{\bar{\Dbf}_{k-1}\left( \frac{t_k}{t_{k-1}}\right)}{\bar{\Hbf}_{k-1}\left( \frac{t_k}{t_{k-1}}\right)} - \frac{\bar{\Dbf}_{k}\left( 1\right)}{\bar{\Hbf}_{k}\left( 1\right)}\right| \leq C \boldsymbol{m}_{0,k}^{\gamma_2} (1+ \Ibf (t_k^+)).
	\end{equation}
	Rearranging and using the triangle inequality, it suffices to demonstrate that
	\begin{align}
		&\left|\frac{\bar{\Dbf}_{k-1}\left(\frac{t_k}{t_{k-1}}\right) - \bar{\Dbf}_k(1)}{\bar{\Hbf}_{k}\left(1\right)}\right| \leq C \boldsymbol{m}_{0,k}^{\gamma_2}, \label{eq:BV1} \\
		&\bar{\Dbf}_{k-1}\left(\frac{t_k}{t_{k-1}}\right)\left|\frac{1}{\bar{\Hbf}_{k-1}\left( \frac{t_k}{t_{k-1}}\right)} - \frac{1}{\bar{\Hbf}_{k}\left(1\right)}\right| \leq C \boldsymbol{m}_{0,k}\Ibf_{k-1} \left(\frac{t_k}{t_{k-1}} \right).\label{eq:BV2}
	\end{align}
	Before we proceed, given $\pi_k$ such that $\Ebf(T,\Bbf_{6\sqrt{m}t_k}) = \Ebf(T,\Bbf_{6\sqrt{m}t_k}, \pi_k)$ let us introduce the Lipschitz approximation $f_k: B_3 \subset \pi_k \to \Acal_Q(\pi_k^\perp)$ of~\cite{DLS14Lp}*{Theorem~2.4} for $T_{0,t_k}\mres\Bbf_{6\sqrt{m}}$ and the map $\tilde{f}_{k-1}\coloneqq (f_k)_{t_k/t_{k-1}}: B_{t_k/t_{k-1}}(0,\pi_{k}) \to \Acal_Q(\pi_{k}^\perp)$ from Corollary~\ref{cor:graphDir} with $\bar{r} = \frac{t_k}{t_{k-1}}$. We let $\tilde{\boldsymbol{\vphi}}_{k-1}$, $\tilde{\boldsymbol{\vphi}}_k$ be as in Lemma \ref{lem:cm} and let $\tilde{f}_k \coloneqq \tilde{f}_{k-1}\big(\frac{t_k}{t_{k-1}}\cdot\big)$. We additionally introduce the measures $\dd \mu_{k-1}(y) \coloneqq \phi_k\left(\frac{t_{k-1}}{t_k}d(y)\right)\dd y$ and $\dd\mu(y) \coloneqq \phi\left(d(y)\right)\dd y$, where $\dd y$ is the $m$-dimensional Lebesgue measure on $\pi_{k}$. We also define the balls $\Bcal^{k-1} \coloneqq \Bbf_{t_k/t_{k-1}}\cap\Mcal_{k-1}$, $B^{k-1} \coloneqq B_{t_k/t_{k-1}}(0, \pi_{k})$ and the cylinder $\Cbf^{k-1} \coloneqq \Cbf_{t_k/t_{k-1}}(0, \pi_{k})$.
	
	We begin with the estimate~\eqref{eq:BV1}. Comparing both terms with the corresponding linearized quantity (cf. Corollary~\ref{cor:graphDir}) and rescaling appropriately we have
	\begin{align*}
		&\left|\bar{\Dbf}_{k-1}\left( \frac{t_k}{t_{k-1}}\right) - \bar{\Dbf}_k(1)\right| \\
		&\qquad\leq \left(\frac{t_k}{t_{k-1}}\right)^{-(m-2)} \left|\int_{\Bcal^{k-1}}|DN_{k-1}|^2\dd\mu_{k-1} - \int_{B^{k-1}} \Gcal(D\tilde f_{k-1}, Q\llbracket D\tilde{\boldsymbol{\vphi}}_{k-1} \rrbracket)^2\dd\mu_{k-1}\right| \\
		&\qquad\qquad + \left|\int_{\Bbf_{1}\cap\Mcal_{k}}|DN_{k}|^2\dd\mu - \left(\frac{t_k}{t_{k-1}}\right)^{-(m-2)}\int_{B^{k-1}} \Gcal (D\tilde f_{k-1}, Q\llbracket D\tilde{\boldsymbol{\vphi}}_{k-1} \rrbracket)^2\dd\mu_{k-1}\right| \\
		&\qquad= \left(\frac{t_k}{t_{k-1}}\right)^{-(m-2)} \left|\int_{\Bcal^{k-1}}|DN_{k-1}|^2\dd\mu_{k-1} - \int_{B^{k-1}} \Gcal (D\tilde f_{k-1}, Q\llbracket D\tilde{\boldsymbol{\vphi}}_{k-1} \rrbracket)^2\dd\mu_{k-1}\right| \\
		&\qquad\qquad + \left|\int_{\Bbf_1\cap\Mcal_{k}}|DN_{k}|^2\dd\mu - \int_{B_1(0,\pi_{k})} \Gcal (D\tilde{f}_{k}, Q\llbracket D \tilde{\boldsymbol{\vphi}}_k \rrbracket)^2\dd\mu\right|.
	\end{align*}
	Now we may use Lemma~\ref{lem:cm} to replace $\tilde{\boldsymbol{\vphi}}_k$ with $\boldsymbol{\vphi}_k$, yielding
	\begin{align*}
	    &\left|\bar{\Dbf}_{k-1}\left( \frac{t_k}{t_{k-1}}\right) - \bar{\Dbf}_k(1)\right| \\
		&\qquad\leq \left(\frac{t_k}{t_{k-1}}\right)^{-(m-2)} \left|\int_{\Bcal^{k-1}}|DN_{k-1}|^2\dd\mu_{k-1} - \int_{B^{k-1}} \Gcal (D\tilde f_{k-1}, Q\llbracket D\tilde{\boldsymbol{\vphi}}_{k-1} \rrbracket)^2 \dd\mu_{k-1}\right| \\
		&\qquad\qquad + \left|\int_{\Bbf_{1}\cap\Mcal_{k}}|DN_{k}|^2 \dd\mu - \int_{B_1(0,\pi_k)} \Gcal (D\tilde{f}_{k}, Q\llbracket D \boldsymbol{\vphi}_k \rrbracket)^2\dd\mu\right| + C \boldsymbol{m}_{0,k}^{1+\gamma_2}.\\
	\end{align*}
	We are now in a position to make use of Corollary~\ref{cor:graphDir}, combined with the observation that $\tilde f_k$ is still a valid $\pi_k$-approximation for $T_{0,t_k}$ in $\Cbf_4(0,\pi_k)$ as in \cite{DLS14Lp}, since $f_{k-1}$ is a $\pi_{k-1}$-approximation for $T_{0,t_{k-1}}$ and we have the estimates \cite{DLS16centermfld}*{Proposition 4.1} on the tilting of $\pi_k$ relative to $\pi_{k-1}$. This gives
	\begin{align*}
	    \left|\bar{\Dbf}_{k-1}\left(\frac{t_k}{t_{k-1}}\right) - \bar{\Dbf}_k(1)\right| &\leq C\left(\frac{t_k}{t_{k-1}}\right)^{-(m-2)}\Bigg(\int_{B^{k-1}} (|Df_{k-1}|^4 + |D\tilde{\boldsymbol{\vphi}}_{k-1}|^4) \dd y \notag\\
		&\qquad + \int_{B^{L_k}} (|Df_{L_k}|^4 + |D\boldsymbol{\vphi}_{L_k}|^4) \dd y \\
		&\qquad+ \int_{\Bcal^{L_k}} (|\Abf_{\Mcal_{k-1}}|^2|N_{k-1}|^2 + |DN_{k-1}|^4) \notag\\
		&\qquad + \int_{\Cbf^{k-1}} \Big| \vec{\Mcal}_{k-1}(\mathbf{p}(z)) - \vec{\Mcal}_{k-1}\big(\boldsymbol{\vphi}(\mathbf{p}_{\pi_k}(z))\big) \Big| \dd\|\Gbf_{f_{k-1}}\|(z)\Bigg) \\
		&\qquad + C \Bigg(\boldsymbol{m}_{0,k-1}^{1+\gamma_2}\left(\frac{t_k}{t_{k-1}}\right)^{4+\gamma_2} + \int_{B_1(0,\pi_k)} (|D\tilde{f}_{k}|^4 + |D\boldsymbol{\vphi}|^4)\dd y \\
		&\qquad + \int_{\Bbf_1 \cap \Mcal_k} (|\Abf_{\Mcal_k}|^2|N_k|^2 + |DN_k|^4)\notag\\
		& \qquad + \int_{\Cbf_1(0,\pi_k)} \Big| \vec{\Mcal}(\mathbf{p}(z)) - \vec{\Mcal}\big(\boldsymbol{\vphi}(\mathbf{p}_{\pi}(z))\big) \Big|  \dd\|\Gbf_{f}\|(z) + \boldsymbol{m}_{0,k}^{1+\gamma_2}\Bigg).
	\end{align*}
	Lemma~\ref{lem:cmLip} thus yields
	\begin{align*}
	    \left|\bar{\Dbf}_{k-1}\left(\frac{t_k}{t_{k-1}}\right) - \bar{\Dbf}_k(1)\right| &\leq C\left(\frac{t_k}{t_{k-1}}\right)^{-(m-2)}\Bigg(\int_{B^{k-1}} (|Df_{k-1}|^4 + |D\tilde{\boldsymbol{\vphi}}_{k-1}|^4) \dd y \notag\\
		&\qquad + \int_{B^{L_k}} (|Df_{L_k}|^4 + |D\boldsymbol{\vphi}_{L_k}|^4) \dd y \\
		&\qquad+ \int_{\Bcal^{L_k}} (|\Abf_{\Mcal_{k-1}}|^2|N_{k-1}|^2 + |DN_{k-1}|^4)\Bigg) \notag\\
		&\qquad + C \Bigg(\boldsymbol{m}_{0,k-1}^{1+\gamma_2}\left(\frac{t_k}{t_{k-1}}\right)^3 + \int_{B_1(0,\pi_k)} (|D\tilde{f}_{k}|^4 + |D\boldsymbol{\vphi}|^4)\dd y \\
		& \qquad + \int_{\Bbf_1 \cap \Mcal_k} (|\Abf_{\Mcal}|^2|N|^2 + |DN|^4) + \boldsymbol{m}_{0,k}^{1+\gamma_2}\Bigg).
	\end{align*}
	We may now control the initial excess $\boldsymbol{m}_{0,k-1}$ of $T_{0,t_{k-1}}$ in terms of the excess $\Ebf(T_{0,t_{k-1}}, \Bbf_{L_k})$, which is in turn controlled by the initial excess $\boldsymbol{m}_{0,k}$ of $T_{0,t_k}$:
	\begin{equation}\label{eq:excessstop}
	    \boldsymbol{m}_{0,k-1}\left(\frac{t_k}{t_{k-1}}\right)^{2-2\delta_2} \leq C\boldsymbol{m}_{0,k}.
	\end{equation}
	This, in combination with the estimates \cite{DLS14Lp}*{Theorem~2.4} and \cite{DLS16centermfld}*{Theorem~1.17, Theorem~2.4, Corollary~2.5} allows us to conclude that
	\begin{align*}
	    \left|\bar{\Dbf}_{k-1}\left(\frac{t_k}{t_{k-1}}\right) - \bar{\Dbf}_k(1)\right| &\leq C\boldsymbol{m}_{0,k}^{1+\gamma_2}.
	\end{align*}
	Since the comparison of center manifolds~\cite{DLS16centermfld}*{Proposition~3.7} gives $\bar{\Hbf}_k(1) \geq c\boldsymbol{m}_{0,k}$ for some dimensional constant $c > 0$, the estimate~\eqref{eq:BV1} follows.
	
	Let us now prove~\eqref{eq:BV2}. First of all, observe that
	\[
	    \bar{\Dbf}_{k-1}\left( \frac{t_k}{t_{k-1}}\right)\left|\frac{1}{\bar{\Hbf}_{k-1}\left( \frac{t_k}{t_{k-1}}\right)} - \frac{1}{\Hbf_{k}\left( 1\right)}\right| = \frac{\Ibf_{k-1}\left( \frac{t_k}{t_{k-1}}\right)}{\Hbf_k(1)}\left|\Hbf_k(1) - \bar{\Hbf}_{k-1}\left( \frac{t_k}{t_{k-1}}\right)\right|.
	\]
	To estimate the difference between the $L^2$-heights, we may one again compare both with the height of the corresponding Lipschitz approximations over the averages of their sheets:
    \begin{align*}
		&\left|\bar{\Hbf}_{k-1}\left( \frac{t_k}{t_{k-1}}\right) - \bar{\Hbf}_k(1)\right| \\
		&\qquad\leq \left(\frac{t_k}{t_{k-1}}\right)^{-(m-1)} \left|\int_{\Bcal^{k-1}}|N_{k-1}|^2\dd\mu_{k-1} - \int_{B^{k-1}} \Gcal(f_{k-1}, Q\llbracket \tilde{\boldsymbol{\vphi}}_{k-1} \rrbracket)^2\dd\mu_{k-1}\right| \\
		&\qquad\qquad + \left|\int_{\Bbf_{1}\cap\Mcal_{k}}|N_{k}|^2\dd\mu - \left(\frac{t_k}{t_{k-1}}\right)^{-(m-1)}\int_{B^{k-1}} \Gcal (f_{k-1}, Q\llbracket \tilde{\boldsymbol{\vphi}}_{k-1} \rrbracket)^2\dd\mu_{k-1}\right| \\
		&\qquad= \left(\frac{t_k}{t_{k-1}}\right)^{-(m-1)} \left|\int_{\Bcal^{k-1}}|N_{k-1}|^2\dd\mu_{k-1} - \int_{B^{k-1}} \Gcal (f_{k-1}, Q\llbracket \tilde{\boldsymbol{\vphi}}_{k-1} \rrbracket)^2\dd\mu_{k-1}\right| \\
		&\qquad\qquad + \left|\int_{\Bbf_1\cap\Mcal_{k}}|N_{k}|^2\dd\mu - \int_{B_1} \Gcal (\tilde{f}_{k}, Q\llbracket \tilde{\boldsymbol{\vphi}}_k \rrbracket)^2\dd\mu\right|.
	\end{align*}
	Now let $\tilde{g}_{k-1}$, $\tilde{g}_k$ be as in Lemma~\ref{l:reparametrization-lemma} for $\tilde{\boldsymbol{\vphi}}_{k-1}, f_{k-1}$ and $\boldsymbol{\vphi}_{k}, \tilde{f}_k$ respectively and let $A^{k-1} \coloneqq B^{k-1}\setminus \frac{1}{2}B^{k-1}$, $A^k \coloneqq B_1(0,\pi_k)\setminus B_{1/2}(0,\pi_k)$. The reverse triangle inequality and the estimate~\eqref{e:comparison-estimate} (combined with an appropriate rescaling) then allow us to deduce that
	\begin{align*}
		\left|\bar{\Hbf}_{k-1}\left(\frac{t_k}{t_{k-1}}\right) - \bar{\Hbf}_k(1)\right| &\leq C\left(\frac{t_k}{t_{k-1}}\right)^{-(m-1)}\int_{A^{k-1}}\Gcal(N_{k-1}(\tilde{\boldsymbol{\vphi}}_{k-1}(y)),\tilde{g}_{k-1}(y))^2\dd y \\
		&\qquad+ \int_{A^k}\Gcal(N_{k}(\boldsymbol{\vphi}_k(y)),\tilde{g}_{k}(y))^2\dd y \\
		&\leq C\left(\frac{t_k}{t_{k-1}}\right)^5\left(\|D\tilde{\boldsymbol{\vphi}}_{k-1}\|_{C^0}^4 + \|D\tilde{\boldsymbol{\vphi}}_{k-1}\|_{C^0}^2\boldsymbol{m}_{0,k-1}\right) + C\|D\boldsymbol{\vphi}_{k}\|_{C^0}^4.
	\end{align*}
	The estimates in~\cite{DLS16centermfld}*{Theorem~1.17, Proposition~3.4} then give
	\[
	    \left|\bar{\Hbf}_{k-1}\left(\frac{t_k}{t_{k-1}}\right) - \bar{\Hbf}_k(1)\right| \leq C \boldsymbol{m}_{0,k}^2.
	\]
	Again using that $\bar{\Hbf}_k(1) \geq c\boldsymbol{m}_{0,k}$, we further have
	\[
	    \frac{\Ibf_{k-1}\left(\frac{t_k}{t_{k-1}}\right)}{\bar{\Hbf}_k(1)} \leq C \boldsymbol{m}_{0,k}^{-1} \Ibf_{k-1} \left(\frac{t_k}{t_{k-1}}\right).
	\]
	The desired estimate follows immediately, and thus we are able to conclude that~\eqref{eq:BVjumps} holds.
	
\medskip

From \eqref{eq:BVjumps} we immediately conclude 
\begin{align}
\sum_k (\log (\Ibf (t_k^+) +1) - \log (\Ibf (t_k^-)+1))_- &\leq C \sum_k \boldsymbol{m}_0^{\gamma_4}\label{eq:BVjumps-2}\, .
\end{align}
Indeed, if $\Ibf (t_k^+) \geq \Ibf (t_k^-)$, then $(\log (\Ibf (t_k^+) +1) - \log (\Ibf (t_k^-)+1))_-=0$, otherwise we have  
\[
(\log (\Ibf (t_k^+) +1) - \log (\Ibf (t_k^-)+1))_- =
\log (\Ibf (t_k^-)+1) - \log (\Ibf (t_k^+) +1) \leq \frac{\Ibf (t_k^-)-\Ibf (t_k^+)}{\Ibf (t_k^+) +1}\, 
\]
and we can just sum \eqref{eq:BVjumps} recalling that $\gamma_2\geq \gamma_4$ and $\boldsymbol{m}_0 \leq 1$.

\medskip

We next wish to control the absolutely continuous part of $\left[\frac{d \log (\Ibf+1)}{dr}\right]_-$. Here, we exploit the almost-monotonicity in Corollary~\ref{cor:freqmono}. 
We argue on each interval $]s_k, t_k[$ and will henceforth let $\partial_r$ denote differentiation in the variable $\frac{t}{t_k}$. Note that $\partial_r = t_k \partial_t$. Due to Corollary~\ref{cor:freqmono}, for almost-every $t \in ]s_k,t_{k}]$ we have
	\begin{align*}
		(\log (\Ibf+1))' (t) &= \frac{1}{t_{k}}\partial_r\Ibf_{k}\left(\frac{t}{t_{k}}\right) \left(1+\Ibf_{k}\left(\frac{t}{t_{k}}\right)\right)^{-1}\\
		&\geq -\frac{C}{t_k}\boldsymbol{m}_{0,k}^{\gamma_4}\left[ 1 + \left(\frac{t}{t_k}\right)^{-1}\Dbf_k\left(\frac{t}{t_k}\right)^{\gamma_4} + \Dbf_k\left(\frac{t}{t_k}\right)^{\gamma_4-1}\Dbf_k'\left(\frac{t}{t_k}\right)\right].
	\end{align*}
We are now ready to introduce a monotone function $\Ombf$ which will help us close the estimate. First of all we let $\psi_k(t) \coloneqq \frac{C}{t_k}\boldsymbol{m}_{0,k}^{\gamma_4}\boldsymbol{1}_{]s_k,t_k]}(t)$ and let the absolutely continuous part of the derivative of $\Ombf$ be 
	\begin{align*}
		\Ombf'(t) &:= {\sum_{k=j_0}^J} \psi_k(t)\left[ 1 + \left(\frac{t}{t_k}\right)^{-1}\Dbf_k\left(\frac{t}{t_k}\right)^{\gamma_4} + \Dbf_k\left(\frac{t}{t_k}\right)^{\gamma_4-1}\Dbf_k'\left(\frac{t}{t_k}\right)\right].
	\end{align*}
Next we consider the ``jump measure''
\[
\mu^j := C \sum_{k= j_0}^J \boldsymbol{m}_{0,k}^{\gamma_4} \delta_{t_k}\, .
\]
Hence we set $\Ombf (s_J) =0$ and define $\Ombf$ by integration, setting its distributional deritative to be $\mu^j + \Ombf' \mathcal{L}^1$. Observe that $\Ombf$ is monotone: $\mu^j$ is obviously a nonnegative measure, while $\Ombf'$ is a nonnegative function since both $\mathbf{D}_k$ and $\mathbf{D}_k'$ are nonnegative (recall the explicit formula for the latter). On the other hand the estimates proved so far obviously show that $\log (\Ibf+1) + \Ombf$ is nondecreasing. In particular it immediately follows that 
\begin{align*}
\left\| \left[\frac{d \log (\Ibf+1)}{dr}\right]_-\right\|_{\TV} 
& \leq \left\|\left[\frac{d\Ombf}{dr}\right]_+\right\|_{\TV} = \Ombf (t_{j_0}) - \Ombf (s_J)\\
&\leq C \sum_{k=j_0}^J \boldsymbol{m}_{0,k}^{\gamma_4} + \sum_{k=j_0}^J \int_{s_k}^{t_k} (\Ombf')_+(t) \dd t \\
&\leq C \sum_{k=j_0}^J \boldsymbol{m}_{0,k}^{\gamma_4} + C\sum_{k=j_0}^{J} \boldsymbol{m}_{0,k}^{\gamma_4}  \int_{\frac{s_k}{t_k}}^{1}\left( 1 + s^{\gamma_4 m-1} + \partial_s(\Dbf_k(s)^{\gamma_4}) + s\right) \dd s \\
&\leq C \sum_{k=j_0}^J \boldsymbol{m}_{0,k}^{\gamma_4}  + C \sum_{k=j_0}^{J-1}  \boldsymbol{m}_{0,k}^{\gamma_4} \left( s + s^{\gamma_4 m} + \Dbf_k(s)^{\gamma_4} \right)\Bigg|_{s = \frac{s_k}{t_k}}^1 \leq C \sum_{k=j_0}^J  \boldsymbol{m}_{0,k}^{\gamma_4}\, .
\end{align*}
\subsection{Proofs of auxiliary results from Section~\ref{ss:aux}}

\begin{proof}[Proof of Lemma~\ref{lem:excessTaylor}]
    	We will argue as in~\cite{DLS_multiple_valued}*{Section~3.1}, making use of the multiple-valued area formula. Consider
	\begin{align*}
		E &\coloneqq \int_{\Cbf_r(0,\pi)} |\vec\Gbf_{f} - \vec{\Mcal}\circ\mathbf{p}|^2\phi\left(\frac{|\mathbf{p}_\pi(z)|}{r}\right)\dd \|\Gbf_{f}\|(z) \\
		&= 2\int_{\Cbf_r(0,\pi)} \phi\left(\frac{\mathbf{p}_{\pi}(p)|}{r}\right)\dd\|\Gbf_{f}\|(p) - 2 \int_{\Cbf_r(0,\pi)} \langle \vec\Gbf_{f}, \vec{\Mcal}\circ\mathbf{p}\rangle \phi\left(\frac{|\mathbf{p}_{\pi}(z)|}{r}\right) \dd\|\Gbf_{f}\|(z).
	\end{align*}
	By the $Q$-valued area formula~\cite{DLS_multiple_valued}*{Corollary~1.11}, we have
	\begin{align*}
	    2\int_{\Cbf_r(0,\pi)} \phi\left(\frac{|\mathbf{p}_{\pi}(z)|}{r}\right)\dd\|\Gbf_{f}\|(z) &= 2Q\int_{B_r(0,\pi)} \phi\left(\frac{|y|}{r}\right)\dd y \\
	    &\qquad+ \int_{B_r(0,\pi)} \left(|Df|^2\phi\left(\frac{|\mathbf{p}_{\pi}(y)|}{r}\right) + O(|Df|^4)\right) \dd y.
	\end{align*}
	Meanwhile, for $\vec{\xi}$ such that $\frac{\vec\xi}{|\xi|}$ is the unitary simple $m$-frame orienting $\Mcal$, we have $\vec{\xi} = (e_1 + D\boldsymbol{\vphi}|_{\mathbf{p}_{\pi}(y,f_i(y))} e_1) \wedge (e_m + D\boldsymbol{\vphi}|_{\mathbf{p}_{\pi}(y,f_i(y))} e_m)$ and $v_k^i = e_k + Df_i|_{y} e_k$, $w_k^i = e_k + D\boldsymbol{\vphi}|_{\mathbf{p}_{\pi}(y,f_i(y))} e_k$, we have
	\begin{align*}
		2 \int_{\Cbf_r(0,\pi)} &\langle \vec\Gbf_{f}, \vec{\Mcal}\circ\mathbf{p}\rangle\phi\left(\frac{|\mathbf{p}_{\pi}(z)|}{r}\right) \dd\|\Gbf_{f}\|(z) \\
		&= 2 \int_{\Cbf_r(0,\pi)} \langle \vec\Gbf_{f}(z), \vec{\Mcal}\big(\boldsymbol{\vphi}(\mathbf{p}_{\pi}(z))\big) \rangle\phi\left(\frac{|\mathbf{p}_{\pi}(p)|}{r}\right) \dd\|\Gbf_{f}\|(z) \\
		&\qquad + 2 \int_{\Cbf_r(0,\pi)} \langle \vec\Gbf_{f}(z), \left(\vec{\Mcal}(\mathbf{p}(z)) - \vec{\Mcal}\big(\boldsymbol{\vphi}(\mathbf{p}_{\pi}(z))\big)\right) \rangle\phi\left(\frac{|\mathbf{p}_{\pi}(z)|}{r}\right) \dd\|\Gbf_{f}\|(z) \\
		&=\frac{2}{|\xi|}\int_{B_r(0,\pi)} \phi\left(\frac{|y|}{r}\right)\sum_i \langle v_1 \wedge\dots\wedge v_m, w_1 \wedge \dots \wedge w_m \rangle \dd y\\
		&\qquad + 2 \int_{\Cbf_r(0,\pi)} \langle \vec\Gbf_{f}(z), \left(\vec{\Mcal}(\mathbf{p}(z)) - \vec{\Mcal}\big(\boldsymbol{\vphi}(\mathbf{p}_{\pi}(z))\big)\right) \rangle \phi\left(\frac{|\mathbf{p}_\pi(z)|}{r}\right) \dd\|\Gbf_{f}\|(z) \\   
		&=\frac{2}{|\xi|}\int_{B_r(0,\pi)}\phi\left(\frac{|y|}{r}\right) \sum_i \det B^i \dd y \\
		&\qquad + 2 \int_{\Cbf_r(0,\pi)} \langle \vec\Gbf_{f}(z), \left(\vec{\Mcal}(\mathbf{p}(z)) - \vec{\Mcal}\big(\boldsymbol{\vphi}(\mathbf{p}_{\pi}(z))\big)\right) \rangle\phi\left(\frac{|\mathbf{p}_\pi(z)|}{r}\right) \dd\|\Gbf_{f}\|(z),
	\end{align*}
	where $B^i_{jk} = \delta_{jk} + \langle Df_i|_y e_j, D\boldsymbol{\vphi}|_{\mathbf{p}_{\pi}(y,f_i(y))} e_k \rangle$. Expanding out the first term, we have
	\[
	    \frac{1}{|\xi|} \det B^i = \left(1 - \frac{1}{2}|D\boldsymbol{\vphi}|^2 + O(|D\boldsymbol{\vphi}|^4)\right)\left(1+ Df_i: D\boldsymbol{\vphi} + O(|Df|^2|D\boldsymbol{\vphi}|^2)\right).
	\]
	Thus, we have
	\begin{align*}
		E &= \int_{B_r(0,\pi)} |Df|^2\phi\left(\frac{|y|}{r}\right)\dd y + Q\int_{B^L}|D\boldsymbol{\vphi}|^2 \phi\left(\frac{|y|}{r}\right) \dd y- 2 \sum_i \int (Df_i: D\boldsymbol{\vphi})\phi\left(\frac{|y|}{r}\right)\dd y \\
		&\qquad+ O\left(\int_{B_r(0,\pi)} (|Df|^4 + |D\boldsymbol{\vphi}|^4 + |Df|^2|D\boldsymbol{\vphi}|^2\right) \\
		&\qquad+ O\left(\int_{\Cbf_r(0,\pi)} \Big|\big\langle \vec\Gbf_{f}(z), \left(\vec{\Mcal}(\mathbf{p}(z)) - \vec{\Mcal}\big(\boldsymbol{\vphi}(\mathbf{p}_{\pi}(z))\big)\right) \big\rangle \Big| \dd\|\Gbf_{f}\|(z) \right) \\
		&= \int_{B_r(0,\pi)} \Gcal(Df, Q\llbracket D\boldsymbol{\vphi}\rrbracket)^2\phi\left(\frac{|y|}{r}\right)\dd y + O\left(\int_{B_r(0,\pi)} (|Df|^4 + |D\boldsymbol{\vphi}|^4\right) \\
		&\qquad + O\left(\int_{\Cbf_r(0,\pi)} \Big| \vec{\Mcal}(\mathbf{p}(z)) - \vec{\Mcal}\big(\boldsymbol{\vphi}(\mathbf{p}_{\pi}(z))\big) \Big| \dd\|\Gbf_{f}\|(z) \right).
	\end{align*}
\end{proof}

\begin{proof}[Proof of Corollary~\ref{cor:graphDir}]
	It suffices to prove~\eqref{eq:graphDir2}, since the argument for \eqref{eq:graphDir1} is analogous (in fact it is easier since one does not need to reparameterize the graphical approximation from the cube $L$ to the plane $\pi_{\bar{r}}$). Let us begin with the corresponding estimate for $f_L$. Letting $F$ be as in~\cite{DLS_multiple_valued}*{Assumption~3.1} for the normal approximation $N$ and letting $\Cbf^L \coloneqq \Cbf_{32r_L}(p_L,\pi_L)$ and $\Bcal^L \coloneqq \Bbf_{64r_L}(p_L)\cap\Mcal$, we have
	\begin{align*}
	    \int_{\Cbf^L} |\vec\Gbf_{f_L} - \vec{\Mcal}\circ\mathbf{p}|^2 \dd \|\Gbf_{f_L}\| &\leq \int_{\Cbf^L} |\vec{T} - \vec{\Mcal}\circ\mathbf{p}|^2 \dd \|T\| + C\| T - \Gbf_{f} \|(\Cbf^L) \\
		&\leq \int_{\mathbf{p}^{-1}(\Bcal^L)} |\vec{\Tbf}_F - \vec{\Mcal}\circ\mathbf{p}|^2 \dd \|\Tbf_F\| + C\| T - \Gbf_{f_L} \|(\Cbf^L) \\
		&\qquad+  C\| T - \Tbf_{F} \|(\mathbf{p}^{-1}(\Bcal^L\setminus\Kcal))\, ,
	\end{align*}
where $\Kcal\subset \mathcal{M}$ is the set over which $T$ (in fact the slices $\langle T, \mathbf{p}, p\rangle$) coincides with $\mathbf{T}_F$ (i.e. the corresponding slices $\langle \mathbf{T}_F, \mathbf{p}, p \rangle$, which in fact are the currents $\sum_i \llbracket F_i (p)\rrbracket = \sum_i \llbracket p+ N_i (p)\rrbracket$).

	Applying (a localized version of) \cite{DLS_multiple_valued}*{Proposition~3.4}, we have
	\begin{align*}
		\int_{\Cbf^L} |\vec\Gbf_{f_L} - \vec{\Mcal}\circ\mathbf{p}|^2 \dd \|\Gbf_{f_L}\| &\leq \int_{\Bcal^L} |DN|^2\dd y  +  C \| T - \Gbf_{f_L} \|(\Cbf^L) +  C\| T - \Tbf_{F} \|(\mathbf{p}^{-1}(\Bcal^L\setminus\Kcal)) \\
		&\qquad+ C\int_{\Bcal^L} (|\Abf_{\Mcal}|^2|N|^2 + |DN|^4).
	\end{align*}
	Let us now control $\| T - \Gbf_{f_L}\|$ and $\| T - \Tbf_{F} \|$. To do this, we make use of the estimates in~\cite{DLS14Lp}*{Theorem~2.4} and~\cite{DLS16centermfld}, combined with a Vitali covering of $\Bcal^L\setminus\Kcal$ by Whitney regions $\Lcal(L')$ and the height bound in~\cite{DLS16centermfld}*{Proposition~4.1}, to deduce that
	\begin{align*}
		\int_{\Cbf^L} |\vec\Gbf_{f_L} - \vec{\Mcal}\circ\mathbf{p}|^2\dd \|\Gbf_{f_L}\| &\leq \int_{\Bbf_L \cap \Mcal} |DN|^2 \dd y  + C\boldsymbol{m}_0^{1+\gamma_1}\ell(L)^{m+2+\gamma_1} \\
		&\qquad+ C\boldsymbol{m}_0^{1+\gamma_2}\ell(L)^{m+2+\gamma_2} + C\int_{\Bcal^L} (|\Abf_{\Mcal}|^2|N|^2 + |DN|^4).
	\end{align*}
	
	It remains to replace $f_L$ with $f_{\bar{r}}$ inside $B_{\bar{r}}(0,\pi_{\bar{r}})$, but this is trivial since $\Gbf_{f_L} \equiv \Gbf_{f_{\bar{r}}} \mres \Cbf_{\bar{r}}(0,\pi_{\bar{r}})$. Combining this with the fact that $\spt\Gbf_{f_{\bar{r}}}\cap\Cbf_{\bar{r}}(0,\pi_{\bar{r}}) \subset \spt\Gbf_{f_{\bar{r}}}\cap\Cbf^L$ and~Lemma~\ref{lem:excessTaylor}, the result follows.
\end{proof}

\begin{proof}[Proof of Lemma~\ref{lem:cm}]
    Let $\eta \in \Crm_c^\infty({B_{2};[0,1]})$ be a cutoff with $\eta \equiv 1$ on $B_{1}$. Integrating by parts and using the estimates in~\cite{DLS16centermfld}*{Theorem~1.17}, we have
	\begin{align*}
		\int_{B_1} |D\boldsymbol{\vphi}_k - D\tilde{\boldsymbol{\vphi}}_{k}|^2  &\leq \int_{B_{2}} |D\boldsymbol{\vphi}_k - D\tilde{\boldsymbol{\vphi}}_{k}|^2 \eta \\
		&= -\int_{B_{2}} (\boldsymbol{\vphi}_k - \tilde{\boldsymbol{\vphi}}_{k})\eta \Delta (\boldsymbol{\vphi}_k - \tilde{\boldsymbol{\vphi}}_{k})  - \int_{B_{2}\setminus B_1} D\eta \cdot (\boldsymbol{\vphi}_k - \tilde{\boldsymbol{\vphi}}_{k}) D(\boldsymbol{\vphi}_k - \tilde{\boldsymbol{\vphi}}_{k}) \\
		&\leq C\left(\boldsymbol{m}_{0,k}^{\frac{1}{2}} + \frac{t_k}{t_{k-1}}\boldsymbol{m}_{0,k-1}^{1/2}\right) \int_{B_{2}} |\boldsymbol{\vphi}_k - \tilde{\boldsymbol{\vphi}}_{k}|.
	\end{align*}
In particular, taking into account \eqref{eq:excessstop}, it suffices to prove \eqref{e:comparison-cm-2}. To that end, consider a Lipschitz approximation $f_k: B_3 (0, \pi_k) \to \mathcal{A}_Q (\mathbb R^n)$ as in~\cite{DLS14Lp}*{Theorem~2.4} for the current $T_{0, t_k}$ in the cylinder $\mathbf{C}_{12} (0, \pi_k)$, where the excess is bounded by $C\boldsymbol{m}_{0,k}$. We claim that 
\begin{align}
\int_{B_2} |\boldsymbol{\varphi}_k - \boldsymbol{\eta}\circ f_k| &\leq C \boldsymbol{m}_{0,k}\label{e:comparison-cm-3}\, ,\\
\int_{B_2} |\tilde{\boldsymbol{\varphi}}_k - \boldsymbol{\eta}\circ f_k| &\leq C \boldsymbol{m}_{0,k}\label{e:comparison-cm-4}\, ,
\end{align}
and obviously \eqref{e:comparison-cm-2} will follow from the latter.

First of all we observe that, since the tilt between the planes $\pi_k$ and $\pi_{k-1}$ is controlled by $\boldsymbol{m}_{0,k}^{1/2}$ due to \cite{DLS16centermfld}*{Proposition 4.1}, all the estimates of \cite{DLS14Lp}*{Theorem~2.4} apply to the map $\bar{f}_k: B_{5/2} (0, \pi_{k-1})$ which parametrizes graphically $\mathbf{G}_{f_k}$ in the cylinder $\mathbf{C}_{5/2} (0, \pi_{k-1})$. Setting $\bar{\boldsymbol{\varphi}}_k := \boldsymbol{\varphi}_{k-1} (\frac{t_k}{t_{k-1}} \cdot)$, 
\eqref{e:comparison-cm-4} will actually follow from 
\begin{equation}\label{e:comparison-cm-5}
\int_{B_2 (0, \pi_{k-1})} |\bar{\boldsymbol{\varphi}}_k - \boldsymbol{\eta}\circ \bar f_k| \leq C \boldsymbol{m}_{0,k}
\end{equation}
combined with \cite{DLS16centermfld}*{Lemma 5.6, Lemma B.1}. 

The argument leading to \eqref{e:comparison-cm-5} is entirely analogous to the one leading to \eqref{e:comparison-cm-3}, with the only difference that instead of a control with $\boldsymbol{m}_{0,k}$ it leads to a control with \[
\boldsymbol{m}_{0,k} + \left(\frac{t_k}{t_{k-1}}\right)^{2-2\delta_2}\boldsymbol{m}_{0,k-1}\, .
\]
However the latter is once again controlled by $C \boldsymbol{m}_{0,k}$ because of \eqref{eq:excessstop}.

We now come to the proof of \eqref{e:comparison-cm-3}. We recall the algorithm leading to the construction of $\boldsymbol{\varphi}_k$. In particular, $B_2$ is covered by the union of contact set $\boldsymbol{\Gamma}$ and the Whitney cubes $L\in \mathscr{W}$ described in \cite{DLS16centermfld}*{Section~1}. We discard the cubes which are not intersecting $B_2$ and denote the family of remaining ones by $\mathscr{W}'$. Since the sidelength of each such cube is at most $2^{-N_0}$, we can assume that each cube $L\in \mathscr{W}'$ is fully contained within $B_3(0,\pi_k)$, where $f_k$ is defined. We can then estimate
\begin{equation}\label{e:comparison-cm-7}
\int_{B_2} |\boldsymbol{\varphi}_k - \boldsymbol{\eta}\circ f_k|
\leq \int_{\boldsymbol{\Gamma}\cap B_2} |\boldsymbol{\varphi}_k - \boldsymbol{\eta}\circ f_k| +
\sum_{L\in \mathscr{W}'} \int_L |\boldsymbol{\varphi}_k - \boldsymbol{\eta}\circ f_k| \, .
\end{equation}
Before coming to the estimates of each integrand in the above sums, we record the following important consequence of \cite{DLS14Lp}*{Theorem~2.4} and \cite{DLS16centermfld}*{Theorem 1.17}:
\begin{equation}\label{e:comparison-cm-8}
\|\boldsymbol{\varphi}_k - \boldsymbol{\eta}\circ f_k\|_{C^0} \leq C \boldsymbol{m}_{0,k}^{\gamma}\, ,
\end{equation}
for $\gamma = \min\{\frac{1}{2m},\gamma_1\}$, where $\gamma_1 > 0$ is as in \cite{DLS14Lp}*{Theorem 2.4}. We moreover let $K\subset B_3(0,\pi_k)$ be the set of \cite{DLS14Lp}*{Theorem~2.4} for $f_k$, namely the set over which, loosely speaking, the graph of $f_k$ coincides with the current $T_{0, t_k}$.

In order to estimate the first integrand in the sum on the right-hand side of \eqref{e:comparison-cm-7}, observe that the identity 
\[
T_{0, t_k} \res (\boldsymbol{\Gamma}\times \pi_k^\perp) = Q \llbracket \mathbf{G}_{\boldsymbol{\varphi}_k}\rrbracket
\]
follows from \cite{DLS16centermfld}*{Corollary~2.2}. In particular $\boldsymbol{\varphi}_k \equiv \boldsymbol{\eta}\circ f_k$ on $\boldsymbol{\Gamma}\cap K$ and so we can estimate
\begin{equation}\label{e:comparison-cm-9}
\int_{\boldsymbol{\Gamma}\cap B_2} |\boldsymbol{\varphi}_k - \boldsymbol{\eta}\circ f_k| \leq 
|B_3\setminus K| \|\boldsymbol{\varphi}_k - \boldsymbol{\eta}\circ f_k\|_{C^0}
\leq C \boldsymbol{m}_{0,k}^{1+2\gamma}\, .
\end{equation}

As for the remaining summands in the right hand side of \eqref{e:comparison-cm-7}, we introduce the plane of reference $\pi_L$ of \cite{DLS16centermfld}*{Definition~1.14}, the $\pi_L$-approximation $f_L$ of Lemma \cite{DLS16centermfld}*{Lemma 1.15}, and the tilted interpolating function $h_L$ and the interpolating function $g_L$ of \cite{DLS16centermfld}*{Definition~1.16}. We start by appealing to \cite{DLS16centermfld}*{Proposition~4.4(v)\& Theorem~1.17(ii)} to estimate 
\begin{equation}\label{e:comparison-cm-11}
\int_L |\boldsymbol{\varphi}_k - g_L|
\leq C \boldsymbol{m}_{0,k} \ell (L)^{m+3+\beta_2/3}\, .
\end{equation}
Next, let $f'_L$ and $(\boldsymbol{\eta}\circ f_L)'$ be the functions defined on $L$ and taking values, respectively, on $\mathcal{A}_Q (\pi_L^\perp)$ and $\pi_L^\perp$, whose graphs coincide with the graphs of $f_L$ and $\boldsymbol{\eta}\circ f_L$ on $L\times \pi_k^\perp$. We first use \cite{DLS16centermfld}*{Lemma~B.1(b)} to estimate 
\begin{equation}\label{e:comparison-cm-12}
\int_L |g_L - (\boldsymbol{\eta}\circ f_L)'|\leq C \int_{B_{2\sqrt{m} \ell (L)} (p_L, \pi_L)} |h_L - \boldsymbol{\eta}\circ f_L|,
\end{equation}
where $p_L$ is the center of $L$, while by \cite{DLS16centermfld}*{Proposition~5.2}, we have
\begin{equation}\label{e:comparison-cm-13}
 \int_{B_{2\sqrt{m} \ell (L)} (p_L, \pi_L)} |h_L - \boldsymbol{\eta}\circ f_L| \leq C \boldsymbol{m}_{0,k} \ell (L)^{m+3+\beta_2}\, .
\end{equation}
In addition, \cite{DLS16centermfld}*{Lemma~5.6} gives us the estimate
\begin{equation}\label{e:comparison-cm-14}
\int_L |(\boldsymbol{\eta}\circ f_L)' - \boldsymbol{\eta} \circ (f_L')|
\leq C \boldsymbol{m}_{0,k} \ell (L)^{m+3+\beta_2/2}\, .
\end{equation}
Putting \eqref{e:comparison-cm-11}, \eqref{e:comparison-cm-12}, \eqref{e:comparison-cm-13}, and \eqref{e:comparison-cm-14} together we then reach 
\begin{equation}\label{e:comparison-cm-15}
\int_L |\boldsymbol{\varphi}_k - \boldsymbol{\eta}\circ (f_L')|
\leq C \boldsymbol{m}_{0,k} \ell (L)^{m+3+\gamma}\, ,
\end{equation}
for some $\gamma > 0$. Next, observe that by \cite{DLS16centermfld}*{Lemma~1.15} there is a set $K'_L\subset L$ such that in $K'_L \times \pi_k^\perp$, the current $T$ coincides with the graph of $f_L'$ and such that 
\begin{equation}\label{e:comparison-cm-16}
|L\setminus K_L'|\leq C \boldsymbol{m}_{0,k}^{1+\gamma_1} \ell (L)^{m}
\end{equation}
It thus turns out that $f_L'$ and $f_k$ coincide over $K_L'\cap K$. In particular we can estimate
\begin{equation}\label{e:comparison-cm-17}
\int_L| \boldsymbol{\eta}\circ f_L' - \boldsymbol{\eta}\circ f_k|
\leq C (|L\setminus K|+ |L\setminus K'_L|) \boldsymbol{m}_{0,k}^{\gamma_1} \leq C\boldsymbol{m}_{0,k}^{1+2\gamma_1}\ell(L)^m \, ,
\end{equation}
which combined with \eqref{e:comparison-cm-15} gives
\begin{equation}\label{e:comparison-cm-18}
\int_L| \boldsymbol{\varphi}_k - \boldsymbol{\eta}\circ f_k|
\leq C \boldsymbol{m}_{0,k} \ell (L)^{m+3+\gamma} + C\boldsymbol{m}_{0,k}^{1+2\gamma_1}\ell(L)^m\, .
\end{equation}
Since the collection $\mathscr{W}'$ consists of disjoint cubes contained in $B_3$, we can sum \eqref{e:comparison-cm-18} over $L\in \mathscr{W}'$ to reach 
\begin{align}
\sum_{L\in \mathscr{W}'} \int_L |\boldsymbol{\varphi}_k - \boldsymbol{\eta}\circ f_k| &
\leq C \boldsymbol{m}_{0,k} + C \boldsymbol{m}_{0,k}^{1+2\gamma_1} \leq C \boldsymbol{m}_{0,k}\, .\label{e:comparison-cm-19}
\end{align}
Clearly, \eqref{e:comparison-cm-7}, \eqref{e:comparison-cm-9}, and \eqref{e:comparison-cm-19} imply \eqref{e:comparison-cm-3} and thus complete the proof.
\end{proof}

\begin{proof}[Proof of Lemma~\ref{lem:cmLip}]
	We begin with the estimate~\eqref{eq:cmLip1}. Due to the fact that $\|\boldsymbol{\vphi}_{\bar{r}}\|_{C^2} \leq C \boldsymbol{m}_0^{1/2}$ and the estimates in~\cite{DLS14Lp}*{Theorem~2.4}, we have
	\begin{align*}
		\int_{\Cbf_{\bar{r}}(0,\pi_{\bar{r}})} \Big| \vec{\Mcal}(\mathbf{p}(z)) - \vec{\Mcal}\big(\boldsymbol{\vphi}_{\bar{r}}(\mathbf{p}_{\pi_{\bar{r}}}(z))\big) \Big| \dd\|\Gbf_{f}\|(z) &\leq C \boldsymbol{m}_0^{1/2} \int_{\Cbf_{\bar{r}}(0,\pi_{\bar{r}})} |\mathbf{p} - \boldsymbol{\vphi}_{\bar{r}}\circ\mathbf{p}_{\pi_{\bar{r}}}| \dd\|\Gbf_{f}\| \\
		&\leq C \boldsymbol{m}_0^{1/2} \int_{K \times \pi_{\bar{r}}^\perp} |\mathbf{p} - \boldsymbol{\vphi}_{\bar{r}}\circ\mathbf{p}_{\pi_{\bar{r}}}| \dd\|T\| \\
		&\qquad + C\bar{r}^{m+1}\boldsymbol{m}_0^{1+\gamma_1}.
	\end{align*}
	Now by the definition of the scale $\bar{r}$, we may use the height bound~\cite{DLS16centermfld}*{Corollary~2.2}, the estimates in~\cite{DLS16centermfld}*{Proposition~4.1} and to deduce that
	\begin{align*}
		\int_{K \times \pi_{\bar{r}}^\perp} |\mathbf{p}-\boldsymbol{\vphi}_{\bar{r}}\circ\mathbf{p}_{\pi_{\bar{r}}}| \dd\|T\| &\leq \int_{K \times \pi_{\bar{r}}^\perp} |\mathbf{p}(z) - z| \dd\|T\|(z) \\
		&\qquad + \int_{K \times \pi_{\bar{r}}^\perp} |z - \boldsymbol{\vphi}_{\bar{r}}\circ\mathbf{p}_{\pi_{\bar{r}}}(z)| \dd\|T\|(z) \\
		&\leq C\bar{r}^{m+1+\beta_2}\boldsymbol{m}_0^{1/2 +1/2m}.
	\end{align*}
	This gives the claimed estimate~\eqref{eq:cmLip1}. The estimate~\eqref{eq:cmLip2} follows analogously, only at unit scale and via the cover of $B_1$ with Whitney cubes of $\mathscr{W}$ and the coincidence region $\boldsymbol{\Gamma}$, as in~\cite{DLS16blowup}*{Section~4}.
\end{proof}

\subsection{Frequency jumps} While this completes the proof of the desired BV bound, we wish to isolate one more general version of the estimates on the ``jumps'' of the frequency function at the endpoint scales $t_j$, only this time, we want to compare the frequency functions at comparable scales, relative to two center manifolds with different centers. This will prove crucial in our subsequent work \cite{DLSk2}. It follows directly from the above arguments, after observing that we are just using the presence of a ``stopping cube'' in one of the two center manifolds construction, at the desired scale, which is not ``too small'', together with the fact that at all larger scales there are no stopping cubes which are too large. We are in addition using the fact that all constants in the estimates on the center manifold and the associated normal approximation are independent of the center point of the construction (cf. \cite{Sk21}).

\begin{lemma}\label{lem:compare-frequency}
Consider $T$ and $\Sigma$ as in Assumption \ref{asm:1}, let $z$ and $w$ be such that $\Theta (T, z)=\Theta (T,w) = Q$ and let $r\leq r_0, r_1$ be three positive numbers such that:
\begin{itemize}
    \item[(a)] $T_{z, r_0}$ falls under the Assumptions of \cite{DLS16centermfld}*{Theorem 1.17} and $\boldsymbol{\varphi}_0: [-4,4]^m \supset \pi_0 \to \pi_0^\perp$ is the graphical map describing the center manifold $\mathcal{M}_0$ constructed in that theorem applied to $T_{z,r_0}$.
    \item[(b)] $T_{w, r_1}$ falls under the Assumptions of \cite{DLS16centermfld}*{Theorem 1.17} and $\boldsymbol{\varphi}_1: [-4,4]^m \supset \pi_1 \to \pi_1^\perp$ is the graphical map describing the center manifold $\mathcal{M}_1$ constructed in that theorem applied to $T_{w,r_1}$.
    \item[(c)] For the families of Whitney cubes $\mathscr{W}_0$ and $\mathscr{W}_1$ of \cite{DLS16centermfld}*{Definition 1.10} used in the construction of the respective center manifolds, we have
    \begin{align}
    \ell (L) &{<} c_s \rho \qquad \forall \rho\in \left[\frac{r}{r_0}, 4\right] \qquad \forall L\in \mathscr{W}_0 \mbox{ s.t. } L\cap B_\rho (0, \pi_0)\neq \emptyset\\
    \ell (L) &{<} c_s \rho \qquad \forall \rho\in \left[\frac{r}{r_1}, 4\right] \qquad \forall L\in \mathscr{W}_1 \mbox{ s.t. } L\cap B_\rho (0, \pi_1)\neq \emptyset\, ,
    \end{align}
    where $c_s$ is the geometric constant of \cite{DLS16blowup}*{Section 2}.
\end{itemize}
Define
\begin{align*}
\bar{c}_s&:= \max \{\ell (L):  L\in \mathscr{W}^e_0 \;\mbox{and}\; L\cap B_{r/r_0} (0, \pi_0)
\neq\emptyset\}\,   
\end{align*}
and let $N_0$ and $N_1$ be the graphical approximations of $T_{z,r_0}$ on $\mathcal{M}_0$ and $T_{w, r_1}$ on $\mathcal{M}_1$ respectively. Consider the points $x_1=(0, \boldsymbol{\varphi}_1 (0))\in \mathcal{M}_0$ and $x_0 = (\boldsymbol{p}_{\pi_0} (r_1^{-1} (w-z)), \boldsymbol{\varphi}_0 (r_0^{-1} (w-z)))\in\Mcal_1$. Then we have (cf. \eqref{eq:BVjumps}) the estimate
\[
|\mathbf{I}_{N_0} (x_0, r_0^{-1} r)) - \mathbf{I}_{N_1} (x_1, r_1^{-1} r)|
\leq \bar{C} \boldsymbol{m}_0^{\gamma_2} (1+\mathbf{I}_{N_0} (x_0, r_0^{-1} r))\, ,
\]
where the constant $\bar{C}$ depends on $m$, $n$, $\bar n$, $Q$, and $\bar{c}_s$. 
\end{lemma}

\section{Proof of Theorem \texorpdfstring{\ref{thm:uniquefv}}{thm:uniquefv}: the case \texorpdfstring{$\Irm(T,0) > 1$}{I(T,0)>1}}\label{ss:excess-decay}

The goal of this section is to prove that the singular frequency value is unique when $\Irm (T,0) >1$. The proof will also show that the tangent cone is then a unique flat plane and that the rescaled currents converge polynomially fast to it. In particular this section will settle Theorem \ref{t:consequences}\eqref{itm:consequences4}, but also Theorem \ref{t:consequences}\eqref{itm:consequences1},\eqref{itm:consequences2}\&\eqref{itm:consequences3} when $\Irm (T,0)>1$.

\begin{proposition}\label{prop:excessdecay}
Let $T$ be as in Theorem \ref{thm:uniquefv}. Then the conclusions \eqref{itm:consequences1}-\eqref{itm:consequences4} of Theorem \ref{t:consequences} hold whenever $\Irm (T, 0) > 1$.
\end{proposition}

In fact, since it will be useful in our further studies in the papers \cite{DLSk2} and \cite{DMS} we record a consequence of our analysis which is more quantitative.

\begin{proposition}\label{prop:excessdecay-quantitative}
Let $T$ be as in Theorem \ref{thm:uniquefv}. For every $I_0>1$ there are positive constants $C(m,n,Q)$ and $\alpha (I_0,m,n,Q)$ with the following property. If $0$ is a flat singular point at which $\Irm (T, 0) \geq I_0$, then there is a radius $r_0=r_0(T)>0$ (which also implicitly depends on the center point, which we are here assuming is the origin) such that 
\begin{equation}\label{e:excessdecay-quantitative}
\mathbf{E} (T, \mathbf{B}_r) \leq C \left(\frac{r}{r_0}\right)^{\alpha} \max \{\mathbf{E} (T, \mathbf{B}_{r_0}), \bar\varepsilon^2 r_0^{2-2\delta_2}\}\qquad \forall r<r_0\, .
\end{equation}
Moreover, we can choose $\alpha$ to be any number which satisfies the inequalities $\alpha < 2-2\delta_2$ and $\alpha < 2(\Irm (T, 0) -1)$, at the price of a constant $C$ which depends also upon $\alpha$.
\end{proposition}

Before coming to the proof of the proposition we state the following technical fact which will prove to be very useful.

\begin{lemma}\label{l:goodtk}
Let $T$ be as in Theorem \ref{thm:uniquefv}. If there are infinitely many intervals of flattening, then 
\[
\liminf_{k\to\infty} \mathbf{E} (T, \mathbf{B}_{6\sqrt{m} t_k}) = 0\,  
\]
and hence 
\[
\liminf_{k\to\infty} \boldsymbol{m}_{0,k} = 0
\]
\end{lemma}
\begin{proof} The second conclusion is an obvious consequence of the first. In order to prove the first
take a sequence $r_j$ such that $r_j\to 0$ and $\mathbf{E} (T, \mathbf{B}_{6\sqrt{m} r_j}) \to 0$. Then $r_j$ belongs to some interval of flattening $]s_{k(j)}, t_{k(j)}]$. We claim that 
\begin{equation}\label{e:excess-left-endpoint}
\lim_{j\to\infty} \mathbf{E} (T, \mathbf{B}_{6\sqrt{m} s_{k(j)}}) = 0\, ,
\end{equation}
which clearly would imply $s_{k(j)}=t_{k(j)+1}$ and hence the conclusion of the lemma.

Up to extraction of a further subsequence, we distinguish two cases:
\begin{itemize}
\item[(i)] If $\frac{s_{k(j)}}{r_j} \to 0$, since
\[
\mathbf{E} (T, \mathbf{B}_{6\sqrt{m} s_{k(j)}}) \leq C \left(\frac{s_{k(j)}}{t_{k(j)}}\right)^{2-2\delta_2} \boldsymbol{m}_{0,k(j)} \leq C \left(\frac{s_{k(j)}}{r_j}\right)^{2-2\delta_2} \varepsilon_3^2\, ,
\]
we conclude immediately that \eqref{e:excess-left-endpoint} holds.
\item[(ii)] If $\inf_j \frac{s_{k(j)}}{r_j} = \sigma>0$, we then estimate
\[
\mathbf{E} (T, \mathbf{B}_{6\sqrt{m} s_{k(j)}}) \leq \sigma^{-m} \mathbf{E} (T, \mathbf{B}_{6\sqrt{m} r_j}) 
\]
and again \eqref{e:excess-left-endpoint} follows immediately.
\end{itemize}
\end{proof}

We will also need the following two facts about Dir-minimizing functions. For the first one we refer to \cite{DLS_MAMS}, while the second is a well-known fact about classical harmonic functions and can be proved, for instance, using the expansion into spherical harmonics.

\begin{lemma}\label{l:harm-1}
If $u: \mathbb R^m \supset B_1 \to \mathcal{A}_Q (\mathbb R^n)$ is a Dir-minimizing function with $I_u (0) = I_0$, then
\begin{equation}\label{e:decay_harmonic_1}
\int_{B_\rho} |Du|^2 \leq \rho^{m+2I_0-2} \int_{B_1} |Du|^2 \qquad \forall \rho<1\, .
\end{equation}
\end{lemma}

\begin{lemma}\label{l:harm-2}
If $w: \mathbb R^m \supset B_1 \to \mathbb R^n$ is a classical harmonic function, then
\begin{equation}\label{e:decay_harmonic_2}
\int_{B_\rho} |Dw - Dw (0)|^2 \leq \rho^{m+2} \int_{B_1} |Dw|^2 \qquad \forall \rho<1\, .
\end{equation}
\end{lemma}

In other words, after subtracting an optimal affine map, the frequency (at zero scale) of a classical harmonic map must be at least two. In particular, we can draw the following simple corollary.

\begin{corollary}\label{c:harm-3}
Let $u: \mathbb R^m \supset B_1 \to \mathcal{A}_Q (\mathbb R^n)$ be Dir-minimizing. Then
\begin{equation}\label{e:decay_harmonic_3}
\int_{B_\rho} \mathcal{G} (Du, Q \llbracket D (\boldsymbol{\eta} \circ u) (0) \rrbracket)^2 \leq \rho^{m-2 + 2 \min \{I_u (0), 2\}} \int_{B_1} |Du|^2 \qquad \forall \rho<1\, .
\end{equation}
\end{corollary}

\begin{proof}[Proof of Proposition \ref{prop:excessdecay} {and Proposition \ref{prop:excessdecay-quantitative}}] From now on we assume that $\Irm (T,0) >1$. The main point will be to show the following decay property:
\begin{itemize}
\item[(Dec)] There are $\varepsilon=\eps(T)\in]0,\eps_3]$, $\alpha=\alpha(I_0,m,n,Q)>0$ and $\kappa\in \mathbb N$ such that, if 
\[
\mathbf{E} (T, \mathbf{B}_{6\sqrt{m} t_k})< \varepsilon^2
\]
and $k\geq \kappa$, then:
\begin{enumerate}[(a)]
\item\label{itm:dec1} The intervals of flattening $]s_k, t_k], ]s_{k+1}, t_{k+1}], \ldots , ]s_{k+\kappa}, t_{k+\kappa}]$ satisfy $s_{k+j-1} = t_{k+j}$ for $j = 1,\dots,\kappa$.
\item\label{itm:dec2} $\boldsymbol{m}_{0, k+\kappa} \leq \left(\frac{s_{k+\kappa}}{t_k}\right)^{\alpha} \boldsymbol{m}_{0,k}$.
\end{enumerate}
\end{itemize}
Before coming to the proof of (Dec), observe that
thanks to Lemma \ref{l:goodtk}, there is at least one integer $k_0 \in \N$ {such that
\[
    \Ebf(T,\Bbf_{6\sqrt{m}t_{k_0}}) < \eps^2\,,
\]}
and since it can be iterated, we may use (Dec) to conclude that
\[
\boldsymbol{m}_{0,k_0+j \kappa} \leq \left(\frac{t_{k_0+j\kappa}}{t_{k_0}}\right)^\alpha \boldsymbol{m}_{0,k_0} \leq \eps_3^2 \left(\frac{t_{k_0+j\kappa}}{t_{k_0}}\right)^\alpha  \qquad \forall j \in \mathbb N\, .
\]
On the other hand, when we have intervals of flattening with coinciding endpoints $s_{k} = t_{k+1}$, we can iterate the estimate
\[
\boldsymbol{m}_{0, k+1} \leq C \left(\frac{t_{k+1}}{t_k}\right)^{2-2\delta_2} \boldsymbol{m}_{0,k} \leq C \boldsymbol{m}_{0,k}\, ,
\]
for $C=C(m,n,Q)>0$, to conclude that indeed
\begin{equation}\label{e:decay_m_0k}
\boldsymbol{m}_{0, k} \leq C \left(\frac{t_k}{t_{k_0}}\right)^\alpha \qquad \forall k\geq k_0\, .
\end{equation}
We then also recall
\[
\mathbf{E} (T, \mathbf{B}_r) \leq C \left(\frac{r}{t_k}\right)^{2-2\delta_2} \boldsymbol{m}_{0,k} \qquad \forall r\in [t_{k+1}, t_k]\, .
\]
Combined with \eqref{e:decay_m_0k}, we infer {the conclusion of Proposition \ref{prop:excessdecay-quantitative}} with $r_0 = t_{k_0}$, which implies immediately the uniqueness of the tangent cone and the polynomial convergence of the rescalings (i.e. point \eqref{itm:consequences4} of Theorem~\ref{t:consequences}). 

Note moreover that, from \eqref{e:decay_m_0k}, the fact that $t_k\downarrow 0$ at least geometrically fast and the frequency BV estimate of the previous section, we conclude the existence of the limit 
\[
I_0 = \lim_{r\downarrow 0} \mathbf{I} (r)\, ,
\]
where $\Ibf$ is the universal frequency function. This immediately implies that every fine blow-up is $I_0$-homogeneous, which in turn gives all the other conclusions of the proposition.

It therefore remains to show (Dec). First of all we choose $\alpha < \min 2\{\Irm (T,0)-1, 1-\delta_2\}$. The choice of $\kappa$ will be more complicated, while those of $k_0$ and $\varepsilon$ are subordinate to $\kappa$. We therefore fix $\kappa$ at the moment, without specifying its choice, and treat it as a constant in order to obtain the choice of $k_0$ and $\varepsilon$. We start by showing that the first point \eqref{itm:dec1} of (Dec) holds and to this effect we impose that $k_0$ is sufficiently large so that
\begin{equation}\label{e:condition_k0_epsilon}
\bar\varepsilon^2 t_{k_0}^{2-2\delta_2} \leq \varepsilon^2.
\end{equation}
Next we recall that 
\[
\mathbf{E} (T, \mathbf{B}_{6\sqrt{m} s_{k}}) \leq C \left(\frac{s_{k}}{t_k}\right)^{2-2\delta_2} \boldsymbol{m}_{0, k} \leq C \boldsymbol{m}_{0,k} 
= C \max\{ \bar\varepsilon^2 t_{k}^{2-2\delta_2}, \varepsilon^2\} \leq C \varepsilon^2\, ,
\]
for each $k \geq k_0$, where $C$ is a geometric constant, independent of $\varepsilon$. In particular, if we choose $\varepsilon$ sufficiently small, we conclude that $\mathbf{E} (T, \mathbf{B}_{6\sqrt{m} s_{k}}) \leq \varepsilon_3^2$, which in turn forces $t_{k+1} = s_{k}$. Observe also that $\boldsymbol{m}_{0, k+1} \leq C \boldsymbol{m}_{0, k}$, where the latter is the same constant of the previous estimate. In particular, as long as $t_{k+i+1} = s_{k+i}$ for $i\in \{0, \ldots , j\}$, we get $\mathbf{E} (T, {\mathbf{B}_{6\sqrt{m} s_{k+j}}}) \leq C^j \boldsymbol{m}_{0, k}$. Since this must be repeated $\kappa$ times, under the assumption that $C^{\kappa_0} \varepsilon^2 \leq \varepsilon_3^2$, we get by induction that $t_{k+j+1} = s_{k+j}$ and $\boldsymbol{m}_{k+j+1} \leq C \boldsymbol{m}_{k+j} \leq C^{j+1} \boldsymbol{m}_{0, k}$.

We next show the second point \eqref{itm:dec2} of (Dec). First of all we observe that it suffices to show
\begin{equation}\label{e:excess-decay-many-scales}
\mathbf{E} (T, \mathbf{B}_{6\sqrt{m} s_{k+\kappa-1}}) \leq \left(\frac{s_{k+\kappa-1}}{t_{k}}\right)^\alpha \boldsymbol{m}_{0, k}\, .
\end{equation} 
In fact, if $\boldsymbol{m}_{0,k} = \bar \varepsilon^2 t_{k}^{2-2\delta_2}$, since $2-2\delta_2 > \alpha$, we then have 
\begin{align*}
\boldsymbol{m}_{0,k+\kappa} &= \max\{\mathbf{E} (T, \mathbf{B}_{6\sqrt{m} s_{k+\kappa-1}}), \bar \varepsilon^2 s_{k+\kappa-1}^{2-2\delta_2}\}
\leq \left(\frac{s_{k+\kappa-1}}{t_{k}}\right)^\alpha \bar \varepsilon^2 t_{k}^{2-2\delta_2}\\
 & = \left(\frac{s_{k+\kappa-1}}{t_{k}}\right)^\alpha \boldsymbol{m}_{0,k}\, .
\end{align*}
But if $\boldsymbol{m}_{0,k} = \mathbf{E} (T, \mathbf{B}_{6\sqrt{m} t_{k}})$, then $\mathbf{E} (T, \mathbf{B}_{6\sqrt{m} t_{k}}) \geq \bar\varepsilon^2 t_k^{2-2\delta_2}$ and hence again 
\[
\bar\varepsilon^2 s_{k+\kappa-1}^{2-2\delta_2} \leq \left(\frac{s_{k+\kappa-1}}{t_{k}}\right)^\alpha \mathbf{E} (T, \mathbf{B}_{6\sqrt{m} t_{k}})
\leq \left(\frac{s_{k+\kappa-1}}{t_{k}}\right)^\alpha m_{k+\kappa}\, .
\]
Towards \eqref{e:excess-decay-many-scales}, we first argue as for the proof of point \eqref{itm:consequences1} of Theorem \ref{t:consequences} to estimate
\begin{equation}
\mathbf{E} (T, \mathbf{B}_{6\sqrt{m} s_{k+\kappa-1}}) \leq C^\kappa \left(\frac{s_{k+\kappa-1}}{t_{k}}\right)^{2-2\delta_2} \boldsymbol{m}_{0, k}\, .
\end{equation}
Since $\kappa$ and $C$ are fixed and $2-2\delta_2 > \alpha$, then clearly \eqref{e:excess-decay-many-scales} follows if $\frac{s_{k+\kappa-1}}{t_k}$ is sufficiently small. We are thus left to prove \eqref{e:excess-decay-many-scales} under the addititional assumption that 
\begin{equation}\label{e:lower-bound-st}
\frac{s_{k+\kappa-1}}{t_{k}} \geq \rho_\ell>0\, ,
\end{equation}
where $\rho_\ell$ is a fixed constant which depends on $\kappa$. Next, recall that $\frac{s_k}{t_k} \leq 2^{-5}$ by \cite{DLS16blowup}*{Proposition 2.2}. We therefore infer that 
$s_{k+\kappa-1} \leq 2^{-5 \kappa} t_{k}$. In fact $\kappa$ will be chosen large enough so that the ratio $\frac{s_{k+\kappa-1}}{t_k}$ is sufficiently small, a condition which we specify here by
\begin{equation}\label{e:upper-bound-st}
\frac{s_{k+\kappa-1}}{t_{k}} \leq \rho_u\, .
\end{equation}
The claim is now that, for an appropriate choice of $\rho_u$ (which in turn fixes the choice of $\kappa$ and of $\rho_\ell$), once $\varepsilon$ and $k_0^{-1}$ are sufficiently small, then 
\eqref{e:excess-decay-many-scales} holds. Towards this we argue by contradiction and assume that, no matter how small we choose $\varepsilon$ and how large we choose $k_0$ (satisfying \eqref{e:condition_k0_epsilon}), there is always a choice of $k\geq k_0$ for which \eqref{e:excess-decay-many-scales} fails. This implies the existence of a sequence $t_k\downarrow 0$ with the property that 
\begin{equation}\label{e:infinitesimal-excess}
\boldsymbol{m}_{0,k}\downarrow 0 \qquad \mbox{and}\qquad \mathbf{E} (T, \mathbf{B}_{6\sqrt{m} s_{k+\kappa-1}}) > \left(\frac{s_{k+\kappa-1}}{t_{k}}\right)^\alpha \boldsymbol{m}_{0, k}\, ,
\end{equation}
while 
\begin{equation}\label{e:two-sided-st}
\rho_\ell \leq \frac{s_{k+\kappa-1}}{t_{k}} \leq \rho_u
\end{equation}
We now choose the radius $r_k$ so that $8M r_k = 6\sqrt{m} t_k$, where $M$ is the constant of \eqref{eq:E_k}. We will assume that $\kappa$ is large enough so that $r_k \geq s_{k+\kappa-1}$. Observe that we can now apply Proposition \ref{p:coarse=fine} and generate the coarse blow-up $\bar f: B_M \to \mathcal{A}_Q$ along the scales $r_k$, which is Dir-minimizing. In light of the comparability of the scales $r_k$ and $s_{k+\kappa-1}$, the average-free part $v$ of $\bar f$ is, up to a positive scalar multiple, a fine blow-up $u$, and we thus infer that $I_v (0) = I_u (0) \geq \Irm (T, 0)$. We can then apply Corollary \ref{c:harm-3} to infer that 
\[
\frac{1}{\sigma^m} \int_{B_\sigma} \mathcal{G} (D\bar f, Q \llbracket D (\boldsymbol{\eta}\circ \bar f (0)))^2 \leq C \left(\frac{\sigma}{M}\right)^{2\alpha} \frac{1}{M^m}
\int_{B_M} |D\bar f|^2\, .
\]
We can now use the Taylor expansion of the excess in \cite{DLS_multiple_valued} to infer that, for all $\sigma \in [\rho_\ell, \rho_u]$, 
\[
\mathbf{E} (T, \mathbf{B}_{6\sqrt{m} \sigma t_k}) \leq 8^m \sigma^{2\alpha} \mathbf{E} (T, \mathbf{B}_{6\sqrt{m} t_k}) + C (\mathbf{E} (T, \mathbf{B}_{6\sqrt{m} t_k} + t_k^2 \mathbf{A}^2)^{1+\gamma}\, .
\]
Since $\mathbf{A} t_k^2$ is controlled by $\boldsymbol{m}_{0,k}$, we easily conclude that, once we choose $\rho_u$ small enough so that $8^m \rho_u^{2\alpha} \leq \frac{1}{2} {\rho_\ell^\alpha}$ and choose $k$ large enough so that
\[
C (\mathbf{E} (T, \mathbf{B}_{6\sqrt{m} t_k}) + t_k^2 \mathbf{A}^2)^{1+\gamma} \leq C \boldsymbol{m}_{0,k}^{1+\gamma} {\leq \frac{1}{2} \rho_\ell^\alpha \boldsymbol{m}_{0,k}}\, ,
\] 
we achieve 
\[
\max_{[\rho_\ell \leq \sigma \leq \rho_u]} \sigma^{-\alpha}\mathbf{E} (T, \mathbf{B}_{6\sqrt{m} \sigma t_k}) \leq \boldsymbol{m}_{0,k}
\]
for all $k$ sufficiently large. However this is in contradiction with \eqref{e:infinitesimal-excess} and \eqref{e:upper-bound-st}.

Observe that the threshold $\eps$ in (Dec) may be made independent of $T$ (and the center point, which it also implicitly depends on). This may be done by replacing the above contradiction compactness argument with one in which a sequence of currents $T_k$ and varying centers $x_k$ are taken. However, in order to do this one must also verify that the conclusion of Proposition \ref{p:coarse=fine} holds for ``diagonal" coarse and fine blow-ups taken along such a varying sequence of currents and centers. This is indeed true, but we omit the details here, since this is unnecessary for the remainder of our arguments.
\end{proof}

\section{Proof of Theorem \texorpdfstring{\ref{thm:uniquefv}}{thm:uniquefv}: the case \texorpdfstring{$\Irm (T,0)=1$}{I(T,0)=1}}

In this section we complete the proof of Theorem \ref{thm:uniquefv} by handling the case $\Irm (T,0)=1$. We will moreover complete the proof of the points \eqref{itm:consequences1}, \eqref{itm:consequences2}, and \eqref{itm:consequences3} in Theorem \ref{t:consequences}. 

\begin{proposition}\label{p:I=1}
Let $T$ be as in Theorem \ref{thm:uniquefv}. Then the conclusions \eqref{itm:consequences1}, \eqref{itm:consequences2}\&\eqref{itm:consequences3} of Theorem \ref{t:consequences} hold whenever $\Irm (T, 0) = 1$.
\end{proposition}

A key ingredient in the proof is a decay lemma which is a refinement of the one used in the proof of Proposition \ref{prop:excessdecay}:

\begin{lemma}\label{l:finer-decay}
Let $T$ be as in Theorem \ref{thm:uniquefv}. For every $\gamma>0$ and every $\eta>0$ there are $\varepsilon>0$ and $\rho > 0$ with the following property.
Assume $]a,b]$ is an interval of radii such that
\begin{itemize}
\item[(a)] $0<a<b\leq \rho$;
\item[(b)] $\mathbf{E} (T, \mathbf{B}_{6\sqrt{m} r}) \leq \varepsilon$ for all $a\leq r\leq b$;
\item[(c)] $\mathbf{I} (r) \geq 1+\gamma$ for all $a\leq r\leq b$.
\end{itemize}
Consider the intervals of flattening $]s_{\bar k+ \bar j}, t_{\bar k+\bar j}]\cup ]s_{\bar k+\bar j-1}, t_{\bar k+ \bar j-1}] \cup \ldots \cup ]s_{\bar k}, t_{\bar k}]$ covering $]a,b]$ with the property that $t_{\bar k+\bar j} = s_{\bar k+\bar j-1}$, \ldots, $t_{\bar k+1}=s_{\bar k}$ are contained in $]a,b]$. Then
\begin{equation}\label{e:smallness}
\sum_{i=1}^{\bar j} \boldsymbol{m}_{0, \bar k+i}^{\gamma_4} \leq \eta\, .
\end{equation}
\end{lemma}
\begin{proof} Observe that $\boldsymbol{m}_{0, \kappa+i}\leq \varepsilon^2$ for $i\geq 1$ just by assumption. 
Since by assumption we know that $\boldsymbol{m}_{0,k}\leq \bar \varepsilon^2$, it suffices to prove the decay of (Dec) as long as $k+\kappa \leq \bar k + \bar j - L$ where $L$ is a fixed natural number. In the argument by contradiction leading to the proof of (Dec) we are thus also allowed to assume that $L$ gets arbitrarily large, which in turn means that $\frac{t_k}{a_k}$ tends to infinity (where $]a_k, b_k]$ are corresponding intervals as above). In particular, notice that in the argument given for (Dec) the key point was to infer that the average-free part of the coarse blow-up $v$ has $I_v (0) = I_u (0)$ for some fine blow-up $u$ while $I_u (0) > 1$. In our situation the bound $I_u (0) \geq \Irm (T,0)$ just gives $I_u (0) \geq 1$. On the other hand, using the fact that $\frac{a_k}{t_k}\to 0$ and our assumption that $\mathbf{I} (r) \geq 1+\gamma$ for all $r\in ]a_k, t_k]$, we can use the convergence of the frequency function to conclude
\[
I_{u} (\rho) = \lim_{k\to \infty} \mathbf{I} (\rho r_k) \geq 1+\gamma 
\]
for an arbitrary positive $\rho$. This in turn gives $I_u (0) \geq 1+\gamma$. 
\end{proof}

\begin{proof}[Proof of Proposition \ref{p:I=1}]
As we have already argued at the start of the proof of Proposition \ref{prop:excessdecay}, the key is in fact to prove the second part of Theorem \ref{t:consequences}\eqref{itm:consequences3}. We thus assume that there is some other blow-up sequence $r_k\to 0$ with the property that $\mathbf{I} (r_k) \to 1+2\gamma$ for some $\gamma > 0$. Our aim is then to show that this leads to a contradiction. We apply Lemma \ref{l:finer-decay} from the previous section with some parameter $\eta>0$ which will be chosen later. Fix the corresponding $\varepsilon>0$ and $\rho > 0$ given by Lemma \ref{l:finer-decay} and consider the set 
\[
\mathcal{R}:= \left\{r \in ]0,\rho[ \ : \ \mathbf{E} (T, \mathbf{B}_{6\sqrt{m} r}) \leq \varepsilon^2 \quad \mbox{and} \quad \mathbf{I}^+ (r) \geq 1+\gamma\right\} \, ,
\]
(since the universal frequency function has jumps, at the jump points we let $\mathbf{I}^+ (r)$ be the right-hand limit). We might later need to choose $\varepsilon$ even smaller than that prescribed by Lemma \ref{l:finer-decay}; the only property needed is that the conclusion of the Lemma still applies.

Observe that $\mathcal{R}$ cannot contain a neighborhood of the origin, otherwise we would have $\mathbf{I} (r) \geq 1+\gamma$ for all $r$ sufficiently small, which in turn would imply that, if $u$ is any fine blow-up, then
\[
I_{u} (\rho) \geq 1+\gamma \qquad \forall \rho>0\, .
\]
This shows that $I_u (0) \geq 1+\gamma$ for every fine blow-up, in turn implying that $\Irm (T, 0) \geq 1+\gamma$. On the other hand $\mathcal{R}$ must have $0$ as an accumulation point, namely $\mathcal{R}$ consists of countably many disjoint intervals, which might or might not include any of their endpoints. We enumerate these intervals in order of decreasing scales, and for each one we consider its interior $]a_k, b_k[$. Note that $r_\ell \in ]a_{k(\ell)}, b_{k(\ell)}[$ for all $\ell$ sufficiently large, due to the nature of our chosen sequence of blow-up scales. 

Now notice that the intervals $]a_k, b_k[$ are contained within the full collection of intervals of flattening $]s_j,t_j]$ (with the excess threshold $\bar\eps$). Thus, we can find a sequence of radii $\tilde{\rho}_{k} > b_k$ approaching $b_k$ asymptotically, with $]b_k, \tilde{\rho}_{k}] \cap  \mathcal{R} = \emptyset$, such that one of the following two possibilities holds:
\begin{itemize}
\item[(a)] there are $\rho_k \in ]b_k, \tilde{\rho}_k]$ with $\mathbf{E} (T, \mathbf{B}_{6 \sqrt{m} \rho_{k}}) > \varepsilon^2$ for infinitely many $k$;
\item[(b)] for infinitely many $k$ the inequalities $\mathbf{E} (T, \mathbf{B}_{6 \sqrt{m} r}) \leq \varepsilon^2$ and $\mathbf{I} (r) < 1+\gamma$ hold for all $r$ in the interval $]b_k, \tilde{\rho}_k]$.
\end{itemize}
We first argue that, if $\varepsilon$ is chosen sufficiently small, (a) cannot happen. We argue by contradiction; if this is not true, a subsequence of $T_{0, \tilde\rho_k}$ (and thus of $T_{0,b_k}$), not relabeled, must be converging to a cone which is not flat. We denote it by $C$. Repeat now the procedure above for each $\varepsilon = \frac{1}{j}$ and assume that for each we find a corresponding sequence $b_{k,j}$, with the property that $T_{0, b_{k,j}}$ is converging to a non-flat cone $C_j$. Letting $]s_{\ell (k,j)}, t_{\ell (k,j)}]$ denote the interval of flattening containing $b_{k,j}$, clearly we first have 
 \[
 \lim_{k\to \infty} \frac{s_{\ell (k, j)}}{t_{\ell (k,j)}} \geq c(j) > 0 \qquad \forall j \in \N,
 \] 
 for some constant $c (j)$ which depends only on $C_e$ and $\delta_2$ of the excess stopping condition in the center manifold construction (cf. \cite{DLS16centermfld}) and on $\varepsilon = \frac{1}{j}$, just using that 
 \[
 \mathbf{E} (T, \Bbf_r) \leq C C_e \left(\frac{r}{t_{k,j}}\right)^{2-2\delta_2} \bar \varepsilon^2 \qquad \forall r\in ]s_{k,j}, t_{k,j}[
 \]
 while $b_{k,j}\in ]s_{k,j}, t_{k,j}]$ and 
 \[
 \mathbf{E} (T, \Bbf_{b_{k,j}}) \geq \varepsilon_j^2\, .
 \]
 On the other hand because of the convergence of $T_{0, t_{k,j}}$ to the cone $C_j$ we have 
 \[
 \lim_{k\to \infty} \frac{\Ebf (T, \Bbf_{t_{k,j}})}{\Ebf (T, \Bbf_{s_{k,j}})} =1\, .
 \]
 In turn this implies, again because of the excess stopping condition in the center manifold construction, that 
 \[
 \liminf_{k\to \infty} \frac{s_{\ell (k, j)}}{t_{\ell (k,j)}} \geq c>0
 \]
 for a constant $c$ which this time is independent of $j$. 
 In particular for any sequence $k(j)\uparrow \infty$ which explodes sufficiently fast we have 
 \[
 \lim_{j\to \infty} \frac{s_{\ell (k(j), j)}}{t_{\ell (k(j),j)}} \geq \frac{c}{2} > 0 \, .
 \]
 We can therefore apply Proposition \ref{p:coarse=fine} to any such $b_{k(j), j}$ and infer that the corresponding fine and {average-free part of the} coarse blow-ups coincide. 
 
We now argue that at least one such coarse blow-up has to be $1$-homogeneous. First of all, for each $k$ and $j$ we denote by $f_{k,j}$ the Lipschitz approximation of the current $T_{0, b_{k,j}}$ given by \cite{DLS14Lp}*{Theorem 2.4} and by $\bar f_{k,j}$ its normalization $f_{k,j}/E_{k,j}^{\frac{1}{2}}$, where 
\[
    E_{k,j}\coloneqq \Ebf(T_{0, b_{k,j}},\Bbf_{6\sqrt{m}})
\]
as in Section \ref{ss:coarse}.

Observe next that by our definition of the endpoints $b_{k,j}$, for each fixed $j$ we have
\[
    E_{k,j} \overset{k\to\infty}{\longrightarrow} \mathbf{E} (C_j, \mathbf{B}_{6 \sqrt{m}}) = \varepsilon_j^2.
\]
% Indeed $C_j$ is conical and so $\rho^{-m} \mathbf{E} (C_j, \mathbf{B}_\rho))$ must be constant. On the other hand because of our assumptions the latter ratio must be no less than $\varepsilon_j^2$ for $\rho=1$ and no large than it for some other radius $\tilde{\rho}\leq 1$. 
 For every fixed $j$ we then conclude that the sequence of maps $\{\bar f_{k,j}\}_k$ are equi-Lipschitz and we can assume they converge uniformly to some map $\bar f_j$, up to subsequence (not relabeled). Moreover, this map is actually the limit of $\tilde{f}_{k,j} := j f_{k,j} = \varepsilon_j^{-1} f_{k,j}$. Recall however that $\bar{f}_{k,j}$ has a uniform $W^{1,2}$ bound, which is independent of both $k$ and $j$ (unlike $\tilde{f}_{k,j}$, where it clearly depends on $j$). This bound is thus valid for $\bar f_j$ too and we can assume it converges, up to subsequences, strongly in $L^2$ to some $W^{1,2}$ map $\bar f$. By taking a suitable diagonal sequence, and noting that $C^{-1} E_{k,j} \leq \Ebf(T_{0, b_{k,j}},\Bbf_{8M}) \leq C E_{k,j}$, the latter can be assumed to be (up to a scalar multiple $\lambda> 0$) the coarse blow-up generated by the sequence $b_{k(j), j}$.
 
Now \cite{DLS14Lp}*{Theorem 2.4} guarantees the existence of a compact set $K_{k,j}\subset B_1$ over which the graph of $f_{k,j}$ coincides with the current $T_{0, b_{k,j}}$ and enjoying the estimate $|B_1 \setminus K_{k,j}|\leq C j^{-2 (1+\beta)}$ for some constants $C$ and $\beta$.
 Recall that in {$K_{k,j}\times\pi_0^\perp$} the supports of $T_{0, b_{k,j}}\mres \cl{\mathbf{B}}_{5 \sqrt{m}}$ converge in Hausdorff distance to the support of $C_j\mres \cl{\mathbf{B}}_{5 \sqrt{m}}$. 
 
Denote by $\lambda^a_j$ the ``anisotropic rescaling map'' which maps $(x,y)\in \pi_0 \times \pi_0^\perp$ into $(x, jy))$, where we assume that $\pi_0$ is the plane over which we are considering the graphical approximations $f_{k,j}$ of $T_{0, b_{k,j}}$ (up to a rotation we can indeed assume that the plane is a given fixed one). 
Now, $\mathbf{G}_{\tilde{f}_{k,j}} \res K_{k,j} \times \pi_0^\perp = (\lambda^a_j)_\sharp T_{0, b_{k,j}} \res K_{k,j} \times \pi_0^\perp$. On the other hand, for each fixed $j$, the currents $(\lambda^a_j)_\sharp T_{0, b_{k,j}}$ converge to the current $(\lambda^a_j)_\sharp C_j$ (the convergence is in the sense of currents, but it also implies the local Hausdorff convergence of the supports {in $K_{k,j}\times \pi_0^\perp$}, given that $j$ is fixed). Let $K_j$ be the Hausdorff limit as $k\to\infty$ of the compact sets $K_{k,j}$. By the uniform convergence of the functions $\bar{f}_{k,j}$ to $\bar f_j$ (as $k \to\infty$, with $j$ fixed) it is easy to see that $\mathbf{G}_{\bar f_j} \res K_j \times \pi_0^\perp = (\lambda^a_j)_\sharp C_j \res K_j \times \pi_0^\perp$. 

 Next, observe that $(\lambda^a_j)_\sharp C_j$ is still a cone. Thus $\bar f_j$ coincides with a $1$-homogeneous function over $K_j$. Observe also that $|K_j|\geq \limsup_k |K_{k,j}|$ and therefore $|B_1\setminus K_j|\leq C j^{-2(1+\beta)}$. Since $|B_1\setminus K_j|\downarrow 0$ it is easy to conclude that $\bar f$, which is the $L^2$ limit of $\bar f_j$, must in fact be $1$-homogeneous.
 
Having concluded that the coarse blow up $\bar f$ is $1$-homogeneous, we immediately infer that the average-free part is $1$-homogeneous as well, which means that the fine blow-up is too. This however would be incompatible with the fact that $\mathbf{I}^- (b_{k(j), j}) \geq 1+\gamma$.
 
We thus fix now a choice of $\varepsilon$ sufficiently small which forces the alternative (b). Recall that the frequency $\BV$ bound gives that $|\mathbf{I}^- (b_k)- \mathbf{I}^+ (b_k)|\leq C \varepsilon^{\gamma_4}$, which, combined with the fact that $\mathbf{I}^+ (b_{k}) \leq 1+\gamma$ in turn implies that 
\begin{equation}\label{e:too-small-frequency}
\mathbf{I}^- (b_k) \leq 1+ \frac{3}{2} \gamma\,,
\end{equation}
once we take $\eps$ small enough. We now wish to show that $\left\|\left[\frac{d\mathbf{I}}{dr}\right]_-\right\|_{\TV (]a_k, b_k[)}$ can be made arbitrarily small, by choosing $\eta$ and $\varepsilon$ correspondingly small and $k$ sufficiently large. This would imply that $\mathbf{I}$ has to be below $1+ \frac{7}{4} \gamma$ on all $]a_k, b_k[$ with $k$ sufficiently large, thereby concluding the proof (since all but finitely many elements of the initial blow-up sequence $r_k$, on which $\mathbf{I} (r_k) \to 1+2 \gamma$, must in fact be contained in $\mathcal{R}$, while we just showed that in a neighborhood of $0$ relative to $\mathcal{R}$ the value of the universal frequency function is strictly below $1+2\gamma$). Let $]s_{j(k)}, t_{j(k)}]$ be the interval of flattening containing $b_k$. Using Lemma \ref{l:finer-decay} and the BV estimate of Proposition \ref{prop:bv}, we already have that the desired estimate
\[
\left\|\left[\frac{d\mathbf{I}}{dr}\right]_-\right\|_{\TV (]a_k, s_{j(k)}[)} \leq \eta \qquad \mbox{if $s_{j(k)}> a_k$,}
\]
provided that $\eps$ is again chosen sufficiently small. 
Note that, even though the estimate is for $\log (\Ibf+1)$, we know apriori that $\Ibf$ is bounded, so we can invert the log and get a an estimate for $\left\|\left[\frac{d\mathbf{I}}{dr}\right]_-\right\|_{\TV (]a'_k, b_k[)}$ as in \eqref{eq:bv}. The only caveat is that the constant $C$ in the right hand side of \eqref{eq:bv} will now depend upon $\|\Ibf\|_\infty$ if we replace the left hand side with $\left\|\left[\frac{d\mathbf{I}}{dr}\right]_-\right\|_{\TV}$. However, we only need a constant $C$ which is independent of the radii, though it might depend on $T$.  

We therefore set $a'_k:= \max \{a_k, s_{j(k)}\}$ and we wish to show that $\left\|\left[\frac{d\mathbf{I}}{dr}\right]_-\right\|_{\TV (]a'_k, b_k[)}$ can be assumed arbitrarily small, 
provided $\varepsilon$ is chosen wisely and $k$ is sufficiently large. We observe that now $]a'_k, b_k[$ is contained in a single interval of flattening, and that the almost monotonicity estimate on the absolutely continuous part of frequency \eqref{eq:bv-improved} gives 
\[
\left\|\left[\frac{d\mathbf{I}}{dr}\right]_-\right\|_{\TV (]a'_k, b_k[)} \leq C \left(\frac{a'_k}{t_{j(k)}}\right)^{\gamma_4} \boldsymbol{m}_{0, j(k)}\, .
\]
Now, $\boldsymbol{m}_{0, j(k)}$ is at most $\bar\varepsilon^2$, and thus, if the ratio $\frac{a'_k}{t_{j(k)}}$ is sufficiently small we reach the desired threshold. We can therefore assume that 
\[
\frac{a'_k}{t_{j(k)}} \geq \bar c >0 
\]
for some constant $\bar c$. With the latter lower bound at disposal it is simple to see that $\boldsymbol{m}_{0, j(k)}$ can be made arbitrarily small choosing $\varepsilon$ small and $k$ large. In fact, if we choose $\varepsilon = \frac{1}{i}$ and $k(i)\uparrow \infty$, we find that $T_{0, b_{k(i)}}$ converges to a flat plane, which in turn shows that  $\mathbf{E} (T, \mathbf{B}_{6\sqrt{m} t_{j(k(i))}})$ must converge to $0$.
\end{proof}

\section{Proof of Theorem \texorpdfstring{\ref{t:consequences}\eqref{itm:consequences5}\&\eqref{itm:consequences6}}{consequences}}

In this last section of the paper we will prove the last two statements of Theorem \ref{t:consequences}.

\subsection{The case \texorpdfstring{$\Irm (T, 0) < 2-\delta_2$}{I(T,0)<2-delta2}}
Choose $\alpha \in ]\Irm (T,0)-1, 1-\delta_2[$. Since all coarse and fine blow-ups are $I (T,0)$-homogeneous, a simple compactness argument yields the following corollary.
\begin{itemize}
\item[(ND)] There are $\varepsilon >0$ and $\rho>0$ such that, if $r< \rho$ and $\mathbf{E} (T, \mathbf{B}_{6\sqrt{m} \rho}) \leq \varepsilon$, then
\begin{equation}\label{e:nondecay}
\int_{\mathbf{B}_{\rho/2} \cap \mathcal{M}_j} |DN_j|^2 \geq 2^{-(m+2\alpha-2)} \int_{\mathbf{B}_\rho\cap \mathcal{M}_j} |DN_j|^2
\end{equation}
where $]s_j, t_j]\ni \rho$.
\end{itemize}
From \eqref{e:nondecay} we immediately infer that the intervals of flattening cannot be finite. Indeed suppose this is not the case and let $J$ be such that $s_J =0$. Observe that under this assumption there is a unique flat tangent cone to $T$: indeed the center manifold $\mathcal{M}_J$ contains the origin and $Q \llbracket T_0 \mathcal{M}_J \rrbracket$ is the unique tangent cone to $T$. We thus conclude $\mathbf{E} (T, \mathbf{B}_{6\sqrt{m} r})\to 0$ as $r\downarrow 0$. In particular \eqref{e:nondecay} must hold for all $\rho \leq \bar \rho$ for some positive $\bar \rho$ and we immediately conclude that there is a positive constant $C$ such that
\[
\int_{\mathcal{M}_J \cap \mathbf{B}_\rho} |DN_J|^2 \geq C^{-1} \rho^{m+2\alpha-2} \qquad \forall \rho< \bar \rho\, .
\] 
On the other hand, in light of \cite{DLS16blowup}*{Remark 3.4} we also have 
\[
\int_{\mathcal{M}_J \cap \mathbf{B}_\rho} |DN_J|^2 \leq C \boldsymbol{m}_{0, J} \left(\frac{\rho}{t_j}\right)^{m+2 - 2\delta_2}\, .
\]
This however forces the condition $\alpha - 1 \geq 1-\delta_2$, which gives a contradiction. There are therefore infinitely many intervals of flattening $]s_j,t_j]$. 

Now assume for a contradiction that, up to subsequence (not relabelled), we have
\[
\lim_{j\to \infty} \frac{s_j}{t_j} = 0.
\]
If $\mathbf{E} (T, \mathbf{B}_{6\sqrt{m} t_j})$ does not converge to $0$ as $j \to \infty$, then, up to subsequence, we can assume that $T_{0, t_j}$ converges to a cone $C$. Clearly, by definition, $\boldsymbol{m}_{0,j} = \mathbf{E} (T, \mathbf{B}_{6\sqrt{m} t_j})$ for $j$ large enough, and moreover $\boldsymbol{m}_{0,j} \to \mathbf{E} (C, \mathbf{B}_{6\sqrt{m}})$. On the other hand, for every fixed $\rho > 0$ sufficiently small, we can pass into the limit in the inequality
\[
\mathbf{E} (T_{0, t_j}, \mathbf{B}_\rho) \leq C \rho^{2-2\delta_2} \boldsymbol{m}_{0, j}\, ,
\]
which is valid for those infinitely many $j$'s such that $\frac{s_j}{t_j} < \rho$, and conclude
\[
\mathbf{E} (C, \mathbf{B}_\rho) \leq C \rho^{2-2\delta_2} \mathbf{E} (C, \mathbf{B}_{6\sqrt{m}})\, ,
\]
which is impossible because the radial invariance of $C$ guarantees that $\mathbf{E} (C, \mathbf{B}_\rho)$ is constant in $\rho$. 

We have thus concluded that $\mathbf{E} (T, \mathbf{B}_{6\sqrt{m} t_j})$ converges to $0$. In particular, so does $\boldsymbol{m}_{0,j}$. We thus conclude that, for every $j$ sufficiently large, the inequality $\mathbf{E} (T, \mathbf{B}_{6\sqrt{m} \rho}) \leq \varepsilon^2$ must be valid for all $\rho\in [s_j, t_j]$. This however can be combined with \eqref{e:nondecay} to deduce that
\[
\int_{\mathcal{M}_j \cap \mathbf{B}_{s_j}} |DN_j|^2 \geq C^{-1} \left(\frac{s_j}{t_j}\right)^{m+2\alpha-2} \int_{\mathcal{M}_j \cap \mathbf{B}_{t_j}} |DN_j|^2 \, .
\]
On the other hand using \cite{DLS16centermfld}*{Proposition 3.4} we immediately get 
\[
\int_{\mathcal{M}_j \cap \mathbf{B}_{t_j}} |DN_j|^2\geq C^{-1} \boldsymbol{m}_{0,j}\, .
\]
In particular we conclude
\[
\int_{\mathcal{M}_j \cap \mathbf{B}_{s_j}} |DN_j|^2 \geq C^{-1} \left(\frac{s_j}{t_j}\right)^{m+2\alpha-2} \boldsymbol{m}_{0,j}\, .
\]
But, as for the case already discussed above, this is at odds with the reverse inequality
\[
\int_{\mathcal{M}_j \cap \mathbf{B}_{s_j}} |DN_j|^2 \leq C \left(\frac{s_j}{t_j}\right)^{m+2-2\delta_2} \boldsymbol{m}_{0,j}
\]
when $\frac{s_j}{t_j}$ is allowed to become too small.

\subsection{The case \texorpdfstring{$\Irm (T, 0) > 2-\delta_2$}{I(T,0)>2-delta2}} In this case we fix $\alpha \in ]1-\delta_2,\Irm (T, 0)-1[$. Note that in this case we know that the intervals of flattening cover a neighborhood of $0$ and thus we can infer, again using the compactness and the fact that fine blow-ups are all $\Irm (T,0)$-homogeneous, the following decay lemma:
\begin{itemize}
\item[(D)] There is $\rho>0$ such that, if $r< \rho$, then
\begin{equation}\label{e:nondecay-2}
\int_{\mathbf{B}_{\rho/2} \cap \mathcal{M_j}} |DN_j|^2 \leq 2^{-(m+2\alpha-2)} \int_{\mathbf{B}_\rho\cap \mathcal{M}_j} |DN_j|^2
\end{equation}
when $]s_j, t_j]\ni \rho$.
\end{itemize}
This immediately implies that, if the intervals of flattening are infinitely many, then they must satisfy
\[
\liminf_j \frac{s_j}{t_j} > 0\, .
\]
To see this, we in fact argue by contradiction as above, using this time \cite{DLS16centermfld}*{Proposition 3.4}, to infer that
\begin{equation}\label{e:cm-proposition-3.4}
\int_{\mathcal{M}_j \cap \mathbf{B}_{s_j}} |DN_j|^2 \geq C^{-1} \left(\frac{s_j}{t_j}\right)^{m+2-2\delta_2} \boldsymbol{m}_{0,j}\, ,
\end{equation}
while iterating (D) we instead would get
\[
\int_{\mathcal{M}_j \cap \mathbf{B}_{s_j}} |DN_j|^2 \leq C \left(\frac{s_j}{t_j}\right)^{m+2\alpha-2} \boldsymbol{m}_{0,j}\, ,
\]
which this time is a contradiction because it would force $\alpha -1 \leq 1-\delta_2$ if $\frac{s_j}{t_j}$ is allowed to become too small, which does not hold.

We can now argue as in the proof of Proposition \ref{prop:excessdecay} to obtain, for every fixed $\kappa$ large enough and every $k$ sufficiently large (depending on $\kappa$), a decay of type
\[
\mathbf{E} (T, \mathbf{B}_{6\sqrt{m} s_{k+\kappa}}) \leq C \left(\frac{s_{k+\kappa}}{t_k}\right)^{2\alpha} \mathbf{E} (T, \mathbf{B}_{6 \sqrt{m} t_k}) + C t_k^2\, .
\]
It is not difficult to see that, if $\kappa$ is chosen large enough, an iteration of this inequality (combined with the information that $\liminf \frac{s_j}{t_j} >0$) gives a decay of type
\begin{equation}\label{e:decay-beta}
\mathbf{E} (T, \mathbf{B}_{6\sqrt{m} r}) \leq C r^{2\beta} 
\end{equation}
for every $\beta < \alpha$. In particular we can choose $\beta> 2-\delta_2$, and therefore conclude that, for a sufficiently large $j$, we must have $\boldsymbol{m}_{0,j} = \bar \varepsilon^2 t_j^{2-2\delta_2}$. But then \eqref{e:decay-beta} would imply
\begin{equation}\label{e:decay-beta-2}
\mathbf{E} (T, \mathbf{B}_{6\sqrt{m} r}) \leq C s_j^{2\beta-(2-2\delta_2)} \left(\frac{s_j}{t_j}\right)^{2-2\delta_2} \boldsymbol{m}_{0,j} \leq C s_j^{\beta+\delta_2} \left(\frac{s_j}{t_j}\right)^{2-2\delta_2} \boldsymbol{m}_{0,j}\, . 
\end{equation}
But of course the latter is at odds with \eqref{e:cm-proposition-3.4} when $s_j$ is sufficiently small. This reaches a contradiction and thus shows that there could not be infinitely many intervals of flattening.

We record here the following more quantitative consequence of our analysis, since it will be useful for the further study of flat singular points in our papers \cite{DLSk2} and \cite{DMS}.

\begin{proposition}\label{prop:I>2-2delta-quantitative}
Let $T$ be as in Theorem \ref{thm:uniquefv}. For every $\mu>0$ there is a positive constant $C(\mu ,m,n,Q)$, with the following property. If $\Irm (T, 0) >  2-\delta_2 + \frac{\mu}{2}$ at the flat singular point $0$, then there is $r_0> 0$ such that 
\begin{equation}\label{e:excessdecay-quantitative-higherI}
\mathbf{E} (T, \mathbf{B}_r) \leq C \left(\frac{r}{r_0}\right)^{2-2\delta_2 + \mu} \max \{\mathbf{E} (T, \mathbf{B}_{r_0}), \bar\varepsilon^2 r_0^{2-2\delta_2}\} \qquad \forall r<r_0\, .
\end{equation}
\end{proposition}

%%      ---------------------------------------------------------------------
%%      --------------------------- BIBLIOGRAPHY ----------------------------
%%      ---------------------------------------------------------------------
%% PUT HERE THE BIBLIOGRAPHY IN YOUR FAVOURITE FORMAT
%% Please check that the format of the bibliography is uniform and coherent

\end{document}